\documentclass[12pt]{article}
\usepackage[utf8]{inputenc}
\usepackage{amssymb}
\usepackage{afterpage}
\usepackage{enumerate}
\usepackage[OT1]{fontenc}
\usepackage{bm}
\usepackage{multirow}
\usepackage{diagbox}
\usepackage{amsmath,amssymb,mathrsfs}
\usepackage{natbib}
\usepackage{bbm, dsfont}
\usepackage{makecell}
\usepackage{upgreek}
\usepackage{mathtools}
\usepackage[usenames]{color}
\usepackage{dsfont,makecell,chngcntr}
\usepackage{smile}
\usepackage{multirow,stfloats,booktabs}
\usepackage{setspace}
\usepackage[utf8]{inputenc}
\usepackage[english]{babel}
\usepackage{graphicx} 
\usepackage{subcaption}
\usepackage[colorlinks,
linkcolor=red,
anchorcolor=blue,
citecolor=blue
]{hyperref}

\def\tr{\mathop{\text{tr}}\kern.2ex}

\def\supp{\mathop{\text{supp}}}

\long\def\comment#1{}

\def\tr{\mathop{\text{Tr}}}

\def\cS{{\mathcal{S}}}

\newcommand{\bel}{\begin{eqnarray}\label}
\newcommand{\eel}{\end{eqnarray}}
\newcommand{\bes}{\begin{eqnarray*}}
	\newcommand{\ees}{\end{eqnarray*}}

\newcommand{\LRT}{\mathop{\mathrm{LRT}}}

\let\hat\widehat
\let\tilde\widetilde

\def\EE{{\mathbb E}}

\def\supp{\mathop{\text{supp}\kern.2ex}}

\def\tr{{\rm{Tr}}}

\def\supp{\mathop{\text{supp}}}

\def\tr{\mathrm{Tr}}

\usepackage{mathrsfs}

\usepackage{fullpage}

\usepackage{tabularx,array}
\newcolumntype{M}{>{\centering\arraybackslash}m{0.17\textwidth}}

\usepackage{hyperref}
\usepackage[    protrusion=true,
expansion=true,
final,
babel
]{microtype}
\def\##1\#{\begin{align}#1\end{align}}
\def\$#1\${\begin{align*}#1\end{align*}}
\usepackage{enumitem}

\theoremstyle{plain}

\theoremstyle{mytheoremstyle}

\begin{document}

\title{Combinatorial-Probabilistic Trade-Off:  Community Properties Test in the Stochastic Block Models} 
\author{Shuting Shen\thanks{Department of Biostatistics, Harvard School of Public Health, Boston, MA 02115; Email: \texttt{shs145@g.harvard.edu}}, Junwei Lu\thanks{Department of Biostatistics, Harvard School of Public Health, Boston, MA 02115; Email:
\texttt{junweilu@hsph.harvard.edu}. The paper is  supported by NSF1916211, NIH1R35CA220523-01A1, and NIH5U01CA209414-02. }}
\date{}

\DeclarePairedDelimiter{\nint}\lfloor\rceil

	\maketitle

	\begin{abstract}
	    In this paper, we propose an inferential framework testing the general community combinatorial properties of the stochastic block model. Instead of estimating the community assignments, we aim to test the hypothesis on whether a certain community property is satisfied. For instance, we propose to test whether a given set of nodes belong to the same community or whether different network communities have the same size.
	   We propose a general inference framework that can be applied to all symmetric community properties.  To ease the  challenges caused by the combinatorial nature of communities properties, we develop a novel shadowing bootstrap testing method. By utilizing the symmetry, our method can find a shadowing representative of the true assignment and the number of assignments to be tested in the alternative can be largely reduced. 
	    In theory, we introduce  a combinatorial distance between two community classes and show a  combinatorial-probabilistic trade-off phenomenon in the community properties test. Our test is honest as long as the product of  combinatorial distance between two communities and the probabilistic distance between two assignment probabilities is sufficiently large.
	     On the other hand, we shows that such trade-off also exists in the information-theoretic lower bound of the community property test. We also implement numerical experiments on both the synthetic data and the protein interaction application to show the validity of our method.
	  \end{abstract}
	\noindent {\bf Keyword:}
	Combinatorial inference; stochastic block models; community properties; minimax lower bound.

	\section{Introduction}\label{intro}
	Clutering is an important feature for network studies, which refers to the presence of node communities in the underlying graph. Community partitions the nodes into subgroups, within which a higher level of connectivity is perceived. The broad spectrum of applications for inferring the network community include the fields of sociology \citep{wasserman1994social}, biology \citep{barabasi2004network}, physics \citep{newman2003structure} and internet \citep{albert1999diameter}. Stochastic block model (SBM) \citep{holland1983stochastic} is one of the most widely studied statistical model to depict the community structures in networks. It is a random graph model which divides the nodes into disjoint communities and assigns the probability of connection between two nodes according to their community memberships. 
	
	 One of the central problem in the study of the stochastic block model is the community detection. Many existing research focused on estimating the community labeling and showing the weak and strong consistency of the community estimation \citep{choi2012stochastic, airoldi2013stochastic,mossel2012stochastic, mossel2018proof, massoulie2014community, hajek2016achieving,abbe2015exact}. Some fundamental limits of the community recovery have also been established in the previous studies. For example, \cite{abbe2015exact} showed the optimal phase transition for the exact recovery of the community assignments using the  maximum likelihood. The semidefinite relaxation methods    \citep{abbe2015exact, hajek2016achieving, agarwal2017multisection, bandeira2018random} and the spectral methods \citep{abbe2015community, yun2014accurate, gao2017achieving,abbe2017entrywise}  are  also shown to be optimal in exact recovery. Besides the exact recovery, \cite{zhang2016minimax} quantified the statistical rate of the community estimation via the mis-match ratio and 
	showed the minimax rate of the mis-match ratio for community detection. 
	 
	 The consistency of the community estimation has two major limits: 1) it does not provide the uncertainty assessment of the quality of the estimation, and 2) it requires the recovery of community assignments for all nodes, while in many scientific applications we are interested in the community properties of a  specific subset of nodes. For instance, \cite{tabouy2019variational} studied the ESR1 protein-protein interaction network in breast cancer and aimed to test if a given set of cancer-related proteins belongs to the same community. Another example is the application in human 
	brain connectome: \cite{joshuabrainconn} studied whether two specific  areas of brains belongs to the same cluster.  This reduces to the statistical hypothesis that if two sets of cerebral nodes belong to the same community.  
	We can formulate the above applications as the following examples of statistical hypotheses.
	
\begin{example}[Same community test for $m$ nodes]\label{ex:same}
We want to test whether $m$ given nodes are in the same cluster or not. Without loss of generality, we have the hypothesis: 
\begin{align*}
 &\mathrm{H}_0: \text{Nodes $1, \ldots, m$ belong to the same community,}\\
 &\mathrm{H}_1: \text{There exists two nodes $1\le j\neq k \le m$ belonging to two different communities}.
\end{align*}
\end{example}
\begin{example}[Group community test]\label{ex:group}
  Like the applications in human 
	brain connectome \citep{joshuabrainconn}, we have two group of nodes 
	and within each group, we know in prior that they belong to the same community. We aim to further test whether these two groups of nodes belong to the same cluster. We denote one node set as $S_m = \{1,\ldots, m\}$ and the other node set as $S_{m'} = \{m+1,\ldots, m+m'\}$. The group community hypothesis is 
\begin{align*}
 &\mathrm{H}_0: \text{Nodes in $S_m \cup S_{m'}$ belong to the same community},\\
 &\mathrm{H}_1: \text{Nodes in $S_m$ belong to community $a$, but nodes in $S_{m'}$ belong to community $b \neq a$}.
\end{align*}
\end{example}
\begin{example}[Equal-sized communities test]\label{ex:num}
Given an SBM of $n$ nodes and $K$ communities, we
aim to test whehter each community has the same size. Namely, we aim to test the hypothesis:
\begin{align*}
 &\mathrm{H}_0: \text{Each community has the size $n/K$},\\
 &\mathrm{H}_1: \text{At least one of the communities size is not equal to $n/K$}.
\end{align*}
\end{example}

In order to conduct hypothesis tests including the above examples, we develop a general inference community property test. We consider the SBM with $n$ nodes and $K$ communities.  Denote the community assignment of the nodes as   $z=(z(1),...,z(n)) \in \{1, \ldots, K\}^n$ such that $z(j) = k$ implies that the $j$-th node belongs to the $k$-th community. In order to specify  the true assignment, we assume $z$ is deterministic.   The homogeneous SBM assumes that the edges of the random graph are independent Bernoulli random variables, i.e.,  the probability of the nodes $i$ and $j$ being connected is  $p$ if $z(i) = z(j)$ and the probability of the nodes $i$ and $j$ being connected is  $q$ if $z(i) \neq z(j)$. Let $\cC_0, \cC_1 \subset \{1, \ldots, K\}^n$ be two disjoint communities assigment families. We are interested in the {\it general community property test}:
\begin{equation}\label{eq:commtest}
    \mathrm{H}_0: z \in \cC_0 \text{~versus~} \mathrm{H}_1: z \in \cC_1.
\end{equation}
The concrete examples of $\cC_0$ and $\cC_1$ are listed in Examples \ref{ex:same} - \ref{ex:num}.
We characterize the  hardness of differentiating the null hypothesis from the alternative by two kinds of ``distances": the probabilistic  distance: how close between $p$ and $q$ and the combinatorial distance: how close between $\cC_0$ and $\cC_1$. The existing literature in the study of the community detection only focused on the probability distance, e.g., $\sqrt{p} - \sqrt{q}$ \citep{abbe2015community} or the Renyi divergence \citep{zhang2016minimax}
\begin{equation}
I(p,q)=-2 \log \big (\sqrt{pq}+\sqrt{(1-p)(1-q)} \big ).
\end{equation}
In comparison to these results, our paper introduce a novel combinatorial distance between $\cC_0$ and $\cC_1$ denoted as $d(\cC_0, \cC_1)$ measuring the number of misalignments between two families (we refer the exact definition to Definition \ref{def:clusterdist}).
The main result of our paper is that, for a wide range of SBM models, we can propose a general testing method which is honest and powerful when 
\begin{equation}\label{eq:pqdistance}
\text{Combinatorial-Probabilistic Trade-Off:~}  I(p,q) d(\cC_0, \cC_1) = \Omega(n^{\epsilon}) 
\end{equation}
for some arbitrarily small $\epsilon >0$. On the other hand, we show the minimax lower bound of the test in the sense that $\mathrm {\rm H}_0$ and $\mathrm {\rm H}_1$ in \eqref{eq:commtest} cannot be differentiated  when $ I(p,q) d(\cC_0, \cC_1) \le c\log n$ for some constant $c>0$.\footnote{We refer to Theorem \ref{t1-mtd} and Theorem \ref{lb-main} for the rigorous arguments about the upper and lower bounds.}
The multiplication between $I(p,q)$ and $d(\cC_0, \cC_1)$ reveals the trade-off between the probabilistic  distance and the combinatorial distance in the general community property test. Our paper makes the following specific contributions to achieve such trade-off.

{\bf $\bullet$ Methodology.}
We propose a likelihood ratio test for the community property test in \eqref{eq:commtest}. We show that our test is generally honest and powerful as long as the tested community properties $\cC_0$ and $\cC_1$ are symmetric under community assignment permutation transforms, which covers all examples above. Comparing to the likelihood ratio test on the community numbers \citep{wang2017likelihood}, our likelihood ratio test could be  evaluated over a much larger family of community assignments such that the limiting distribution of our test statistic is no longer always normal. Therefore, the method of \cite{wang2017likelihood} is no longer applicable and  we need to develop a new multiplier bootstrap method to estimate the quantile of our statistic. To achieve this, there are two major challenges. Firstly, 
the possible assignments in the alternative space $\cC_1$ are so large such that the multiplier bootstrap statistic in \cite{chernozhukov2013gaussian} cannot be applied directly. To overcome this, instead of considering the entire alternative class $\cC_1$, our testing method shows that it suffices to consider  the boundary of $\cC_1$, which will significantly reduce the computation complexity. Secondly, the naive multiplier bootstrap requires to know the true assignment. We propose a ``shadowing bootstrap" method by utilizing the symmetry of $\cC_0$ and $\cC_1$. Instead of using the true community assignment, we use a ``shadowing assignment" in the bootstrap which remains to be valid due to the symmetry of the community properties.

{\bf $\bullet$ Theory.}
We show the validity and power of the proposed test when the probability and combinatorial distances satisfies the general relationship in \eqref{eq:pqdistance}. We also prove the minimax lower bound of the general community property test and show that the proposed test is nearly optimal. Due to the generality of the property test, the existing theoretical results on the community detection, e.g., \cite{abbe2015exact} and \cite{zhang2016minimax},  can not be directly applied. To derive the general lower bound, we take a set of the hardest assignments in the alternative $\cC_1$ which are closest to $\cC_0$. These hardest assignments are dependent among each other and we control their dependency via comparing to the distribution of  high dimensional Gaussian vectors. To the best of our knowledge, it is the first time for our paper to derive a minimax lower bound
for general community properties.

\subsection{Related Papers}	 
 There are several existing papers discussing the inference of the community properties of the stochastic block model.   \cite{bickel2015hypothesis} designed a recursive bipartitioning algorithm based on the test statistic derived from the principal eigenvalue of the normalized adjacency matrix to automatically determine the number of clusters $k$. Similarly, \cite{lei2016goodness} developed a goodness-of-fit test based on the largest singular value of the residual matrix for estimating the number of communities. Also interested in inferring the number of the communities, \cite{wang2017likelihood} employed a likelihood ratio statistic and showed the asymptotic normality of the proposed statistic. Compared to the above works that only focus on an exact community property, the hypotheses testing problem in our paper is more general.  Under the framework of the mixed membership model, \cite{fan2019simple} studied the one-sample test of the weight vector of the mixed membership via a singular value decomposition based method. It covers Example \ref{ex:same} for $m = 2$ in the stochastic block model, however, it cannot be directly applied to other examples mentioned above. In comparison, our method
 covers wider range of scenarios. Moreover, we also study the general lower bound of the community test which reveals a novel bridge between the probability distance and the combinatorial distance between the null and alternative community families.
 
Besides inferring the number of communities, \cite{gao2018community} proposed a community detection algorithm in degree-corrected block models that can be reduced to the hypothesis tests on the membership of a given node when the cluster labels of other nodes are the truth. \cite{gao2018community} integrated the hypothesis tests as a technical procedure in community detection without providing any uncertainty assessment. Rich literature can also be found in the testing of underlying random graph models \citep{bubeck2016testing, karwa2016exact, gao2017testing, ghoshdastidar2017two, ghoshdastidar2017twouse, tang2017semiparametric, tang2017nonparametric, shumovskaia2018towards}.

\subsection{Organization of the Paper}
The rest of the paper is organized as follows. Section~\ref{sec: def and ntn} provides the definitions and background knowledge that will be useful for inference on SBM models. At the end of Section~\ref{sec: def and ntn}, we introduce the inference method for community properties test with \textit{symmetric} structures, where we mainly focus on SBM with even cluster sizes, and we provide concrete case studies to illustrate the method procedure. Theoretical results for the methods are developed in Section~\ref{sec: theory}. In Section~\ref{lwrbd}, we focus on the lower bound of the community property test, and in Section~\ref{sec: general framework} we generalize our method to SBM with uneven cluster sizes. Finally in Section~\ref{sec: num} we conduct numerical analysis both on synthetic data and real-world protein interaction data to evaluate the performance of our method. 

\subsection*{Notations}
We denote $|\cdot|$ to be the cardinality of a set. For two positive sequences $x_n$ and $y_n$, we say $x_n = \Omega(y_n)$ if there exists a positive constant $C$ not depending on $n$ such that $x_n \ge C y_n$ for all $n$ sufficiently large. We say $x_n \lesssim y_n$ or $x_n = O(y_n)$ if $x_n \le C y_n$ for $C>0$ not depending on $n$. We say $x_n \asymp y_n$ if $x_n \lesssim y_n$ and $y_n \lesssim x_n$. If $\lim_{n \rightarrow \infty} x_n/y_n = 0$ then we say $x_n=o(y_n)$.

	\section{Community Properties of the Stochastic Block Model}\label{sec: def and ntn}
	
 In our paper, we consider the fixed assignment stochastic block model. Recall that $p$ and $q$ are respectively the within-community and between-community probabilities. Denote $[n] = \{1, \ldots ,n\}$ for any integer $n$. For the stochastic block model with $n$ nodes and $K$ communities, the fixed assignment SBM assumes the community assignment $z = (z(1), \ldots, z(n)) \in [K]^n$ is the prefixed parameter of the model. For simplicity, we start with considering the situation that the community sizes of the assignment are even. We denote the even assignment class $\cK^n := \{ z\in [K]^n: |\{i: z(i) = k \}| = n/K, \forall k \in [K]\}$. We will generalize our analysis to the uneven case in Section \ref{sec: general framework}. We denote the fixed assignment stochastic block model as $\cM(n,K,p,q,z)$. In our paper, we assume the number of communities $K$ is bounded. Let $\Ab \in \{0,1\}^{n \times n}$ be the symmetric adjacency  matrix of the random graph generated from the SBM. We say $\Ab \sim \cM(n,K,p,q,z)$ if the upper triangular entries of $\Ab$ are independent Bernoulli random variable and $\PP(\Ab_{ij} = \Ab_{ji} = 1) = p$ if $z(i) = z(j)$ and $\PP(\Ab_{ij} = \Ab_{ji} = 1) = q$ if $z(i) \neq z(j)$ for any $i \neq j \in [n]$. In the following part of the paper, we will study the community property test with an observation of the adjacency matrix $\Ab \sim \cM(n,K,p,q,z)$.

	\subsection{Symmetric Community Properties}\label{hypo test pre}
	
	In this section, we aim to define the community property and the distance between two community families. In general, we say a community property is a subset of  $[K]^n$. However, such definition is too general and may include some ill-posed examples. For example, if we can transfer one assignment to another under certain permutation of the community labels, they are essentially the same assignment and should belong to the same community property. This motivates us to give the following definition of equivalent assignments. 
	
		\begin{definition}[Equivalent community assignments] Let $S_K$ be the symmetric group containing all bijections from  $[K]$ to itself.  We say two assignments $z$ and $z' \in [K]^n$ are equivalent, denoted as $z \simeq z'$, if there exists a permuation $\sigma \in S_K$ such that $\sigma(z) = z'$. Here $\sigma(z)$ means implementing the permutation $\sigma$ to each entry of the vector $z$. More generally, given a node set $\cN \subseteq [n]$, we denote $z_{\cN}$ as the sub-vector of $z$ with entries in $\cN$. We say $z_{\cN} \simeq z'_{\cN}$ if there exists a permutation $\sigma \in S_K$ such that $\sigma(z_{\cN}) = z'_{\cN}$. 
			\end{definition}
			
	With the concept of equivalent assignments, we can give the definition of symmetric community properties as follows.
	
		\begin{definition}[Symmetric community properties]\label{def:cluster class} We say a community property $\cC_0$ is symmetric, if there exist a node set $\cN \subseteq [n]$ and a specified assignment $\tilde z\in \cK^n$, such that $\cC_0 = \{ z \in \cK^n:  z_{\cN} \simeq \tilde z_{\cN}\}$. We say some $\cC_1 \subseteq \cK^n \backslash \cC_0$ is an alternative property of $\cC_0$ if $\cC_1$ is closed under permutations on the support $\cN$, i.e., for any $z \in \cC_1  \subseteq \cK^n \backslash \cC_0$, if some $z'$ satisfies $z'_\cN \simeq z_\cN$, then $z' \in \cC_1$ as well. 
		\end{definition}

	Intuitively, the node set $\cN$ and the assignment $\tilde z$ in Definition \ref{def:cluster class} are the representative node set and assignment generating all possible assignments in the community property via permutation. The community property $\cC_0$ is ``symmetric" in the sense that all its assignments are equivalent on the support of node set $\cN$.  Therefore, we impose the following assumption on testing symmetric properties.

	\begin{assumption} [Symmetric community property test] \label{asmp:symtest}
	In the hypothesis test ${\rm H}_0: z \in \cC_0$ v.s. ${\rm H}_1: z \in \cC_1$, we assume  $\cC_0, \cC_1 \subseteq \cK^n$ and $\cC_0$ is symmetric and $\cC_1$ is an  alternative property of $\cC_0$.
	\end{assumption}

By Definition \ref{def:cluster class}, $\cC_1 = \cC_0^c$ is an alternative property of $\cC_0$. Meanwhile, $\cC_1$ satisfying the assumption above could be a strict subset of $\cC_0^c$, which allows more examples in practice. In fact, we can show that Examples \ref{ex:same}
 and \ref{ex:group} given in the introduction satisfies Assumption \ref{asmp:symtest}.  We will give concrete forms of the representative node set $\cN$ and assignment $\tilde z$ for these two examples below.  Before going to the detailed discussion, we also want to remark that the assumption that $\cC_0, \cC_1 \subseteq \cK^n$, i.e., the community sizes are even, is only for the simplicity of our statement. We will discuss the uneven cases of  Examples \ref{ex:same}
 and \ref{ex:group} as well as  Example \ref{ex:num} in Section \ref{sec: general framework}.

$\bullet$ Example \ref{ex:same}: Same community test for $m$ nodes. For the null hypothesis that nodes $1, \ldots, m$ belong to the same community, we can define 
\begin{equation}\label{eq:c0ex1}
\cC_0 = \{z \in \cK^n: z(1) = \cdots = z(m)\} \text{ and } \cC_1 = \cK^n \backslash \cC_0.
\end{equation}
Consider $\cN =[m]$ and $\tilde z$ is any assignment satisfying $\tilde z_\cN = (1, \ldots, 1) \in [K]^m$. As $\tilde z$ represents an assignment whose first $m$ nodes belong to one community, we can check that $\cC_0 = \{z: z_{\cN} \simeq \tilde z_{\cN}
\}$ and thus is symmetric. 

$\bullet$ Example \ref{ex:group}: Same community test for groups. Recall that the null hypothesis is that nodes $1,\ldots, m, m+1, \ldots, m+m'$ belong to the same community. Therefore, the null property is similar to Example \ref{ex:same}. Following the same argument of the previous example, $\cC_0$ is symmetric by choosing  $\cN =[m+m']$. 
On the other hand, the alternative hypothesis is different from the previous example. In fact, we have
\begin{equation}\label{eq:c0ex2}
\begin{aligned}
   \cC_0 &= \{z \in \cK^n : z(1) = \cdots = z(m) = z(m+1) = \cdots = z(m+m')\} \\
   \cC_1 &= \{z \in \cK^n: z(1) = \cdots = z(m) \neq z(m+1) = \cdots = z(m+m')\}
\end{aligned}
\end{equation}
 Notice $\cC_1$ is a strict subset of $\cK^n \backslash \cC_0$. We can check $\cC_1$ is an alternative property of $\cC_0$ by Definition \ref{def:cluster class}.

\subsection{Combinatorial Distance Between Community Properties}\label{sec: comb dist}

The major difference between the two examples  above is their alternative properties. With similar null properties, $\cC_1$ in  \eqref{eq:c0ex1} is the complement of $\cC_0$, while  in \eqref{eq:c0ex2}, $\cC_1$  is a strict subset.
From this perspective, the distance between the null and alternative hypotheses in Example \ref{ex:same}  is smaller than Example \ref{ex:group}. In other words, Example \ref{ex:group} is easier to test in comparison to Example \ref{ex:same}. 
 Therefore, in order to  depict the relationship between $\mathcal{C}_0$ and $\mathcal{C}_1$, we will propose a metric of distance in terms of the number of misaligned edges. We first define the set of misaligned edges between two assignments.
\begin{definition}\label{def:e12}
	For any two assignments $z_0 \in \mathcal{C}_0$ and  $z_1 \in \mathcal{C}_1$, we define the two sets of misaligned edges as 
	\begin{align*}
	\mathcal{E}_1({z}_0,{z}_1) &= \{ (i,j):i<j, i,j \in [n],{z}_0(i)={z}_0(j), {z}_1(i) \neq {z}_1(j)\} \text{ and }\\
	\mathcal{E}_2({z}_0,{z}_1) &= \{ (i,j): i<j, i,j \in [n],{z}_0(i)\neq {z}_0(j), {z}_1(i) = {z}_1(j)\},
	\end{align*}
	where $\mathcal{E}_1(z_0,z_1)$ contains the edges whose corresponding nodes are assigned to the same community in $z_0$ but to two different communities by $z_1$, and $\mathcal{E}_2(z_0,z_1)$ is the opposite. See Figure \ref{fig:e12} for an illustration. We denote $n_i(z_0,z_1)=|\cE_i(z_0,z_1)|$, for $i=1,2$ as the cardinality of the two edge sets. 
\end{definition}
 With the definition of misaligned edges, we are ready to propose the metric of assignment distance defined as follows.
\begin{definition}[Community property distance]\label{def:clusterdist}
We define the distance between two assigments $z_0$ and $z_1$ as 
	$d(z_0,z_1) = n_1(z_0,z_1) \vee n_2(z_0,z_1)$.
	Correspondingly, we also define
	$d(z_0, \mathcal{C}_1) = \inf_{z_1 \in \mathcal{C}_1} d(z_0,z_1)$  and the distance between two community properties $$d(\mathcal{C}_0, \mathcal{C}_1) = \inf_{z_0 \in \mathcal{C}_0,z_1 \in \mathcal{C}_1} d(z_0,z_1).$$	
\end{definition}

\begin{figure}[t]
		\centering	\includegraphics[width=0.5\textwidth]{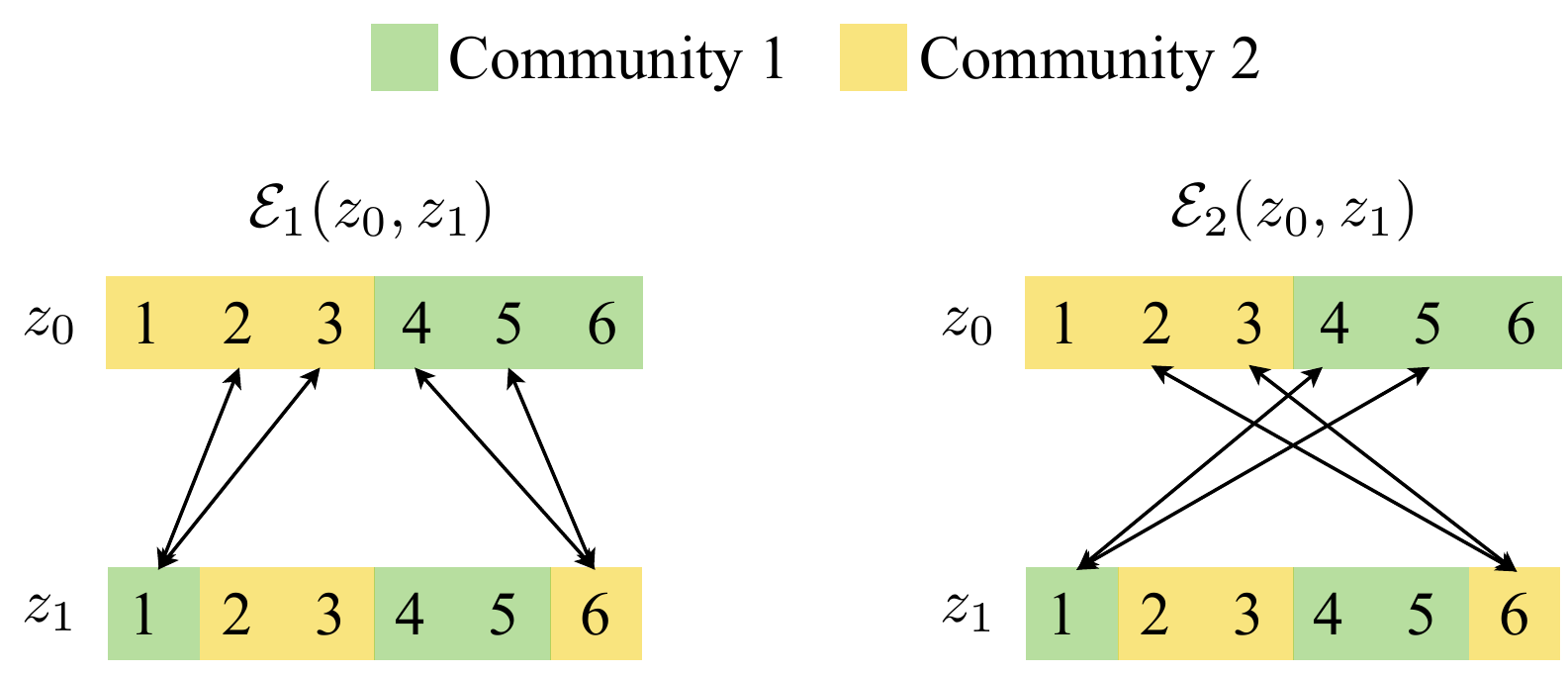}
		\caption{Example of misaligned edges in Definition \ref{def:e12}.}\label{fig:e12}
	\end{figure}

By its defiintion, the distance $d(\mathcal{C}_0,\mathcal{C}_1)$ is the minimal number of misaligned edges between  $\mathcal{C}_0$ and $\mathcal{C}_1$. For instance, consider $\cC_0 = \{z \in \cK^n : z(1) = z(2)  \}$ and  $\cC_1 = \{z \in \cK^n : z(1) \neq z(2)  \}$ with $K=2$. On the one side, since the alignment of nodes 1 and 2 are different, we have $d(\cC_0, \cC_1) \ge |\cE_1(z_0, z_1)| \ge n-2$ given any $z_0 \in \cC_0$ and $z_1 \in \cC_1$. On the other hand, we can easily find two concrete examples of $z_0, z_1$, illustrated in Figure \ref{fig:e12}, such that $d(z_0, z_1) = n-2$. We refer the computation of $d(\cC_0, \cC_1)$ for more general examples to  Section \ref{sec:case}.

	\subsection{Likelihood-Ratio Test for Community Properties}\label{sec: LRT test intro}
 Our method starts with defining a likelihood-ratio test statistic. We denote the observed adjacency matrix from the true model as $\Ab \sim \cM(n, K, p,q,z^*)$, where $z^*$ is the true assignment. The likelihood function of the stochastic block model is
	$$f(\Ab;z,p,q)=\Pi_{i<j}p^{\ind\left(z(i) = z(j) \right) \Ab_{ij}}(1-p)^{\ind\left(z(i) = z(j) \right) (1-\Ab_{ij})}q^{\mathbbm{1}\left(z(i) \neq z(j) \right) \Ab_{ij}}(1-q)^{\mathbbm{1}\left(z(i) \neq z(j) \right) (1-\Ab_{ij})}.$$
We then denote the log-likelihood ratio statistic as 
\begin{equation}\label{eq:lrt}
    \LRT=\log \frac{\sup_{z \in \mathcal{C}_1} f(\Ab;z,p,q)}{\sup_{z \in \mathcal{C}_0 \cup \mathcal{C}_1} f(\Ab;z,p,q)}.
\end{equation}

In order to conduct the property test, we aim to study the limiting distribution of the likelihood ratio statistic. In specific, we are able to decompose the LRT as follows
\begin{align}\label{eq:lrt-lead}
  \nonumber  \LRT &= \sup_{z \in \mathcal{C}_1} \log f(\Ab;z,p,q) - \log f(\Ab;z^*,p,q) + o(1)\\
    &=  \sup_{z \in \mathcal{C}_1} g(p,q)\bigg( \sum_{(i,j) \in \mathcal{E}_2(z^*,z)} \Ab_{ij}-\sum_{(i,j) \in \mathcal{E}_1(z^*,z)} \Ab_{ij} \bigg)+o(1),
\end{align}
where $g(p,q)= \log p(1-q)/\left( q(1-p)\right)$. The first equality above is due to the consistency of the maximum likelihood estimator and the second equality is derived via controlling the remainder term. We defer the proof details to Appendix \ref{sec:lrt-proof}. We observe that the leading term in \eqref{eq:lrt-lead} is the difference of edges in two edge sets: $\mathcal{E}_2(z^*,z)$ and $\mathcal{E}_1(z^*,z)$. By Definition \ref{def:clusterdist}, the property distance $d(\cC_0, \cC_1)$ is larger when the two edge sets are larger, which makes the leading term larger as well. This implies why $d(\cC_0, \cC_1)$ characterizes the difficulty of the test. 

\begin{remark}\label{rm:wang}
  The likelihood ratio statistic in \eqref{eq:lrt} is similar to the one proposed in \cite{wang2017likelihood}. 
  They considered the hypothesis on a specific community property: the number of communities. 
  \begin{equation} \label{eq:commnum}
      \cC_0 = \{z| z \in [K-1]^n \} \text{ and } \cC_1 = \{z| z \in [K]^n \} \backslash \cC_0
  \end{equation}
  They show that the suprema $\sup_{z \in \cC_1} f(\Ab; z, p,q)$ used in \eqref{eq:lrt} is unique, as illustrated in Figure \ref{fig:projection}(a). This makes the LRT  in \cite{wang2017likelihood}  asymptotically normal for $\cC_0, \cC_1$ in \eqref{eq:commnum}. However, this is not always true for the general community properties. For some properties, there will be an exponential number of candidate assignments maximizing the likelihood in \eqref{eq:lrt}, as illustrated in Figure \ref{fig:projection}(b). Thus the LRT is no longer asymptotically normal for the general case.  Therefore, despite the similar formality of the likelihood ratio statistic comparing to the one in  \cite{wang2017likelihood}, our testing procedure will be different from their method.
\end{remark}

	\begin{figure}[t]
		\centering
		\begin{tabular}{cc}
			(a) Unique projection & (b) Non-unique projection\\
			\includegraphics[height=0.3\textwidth]{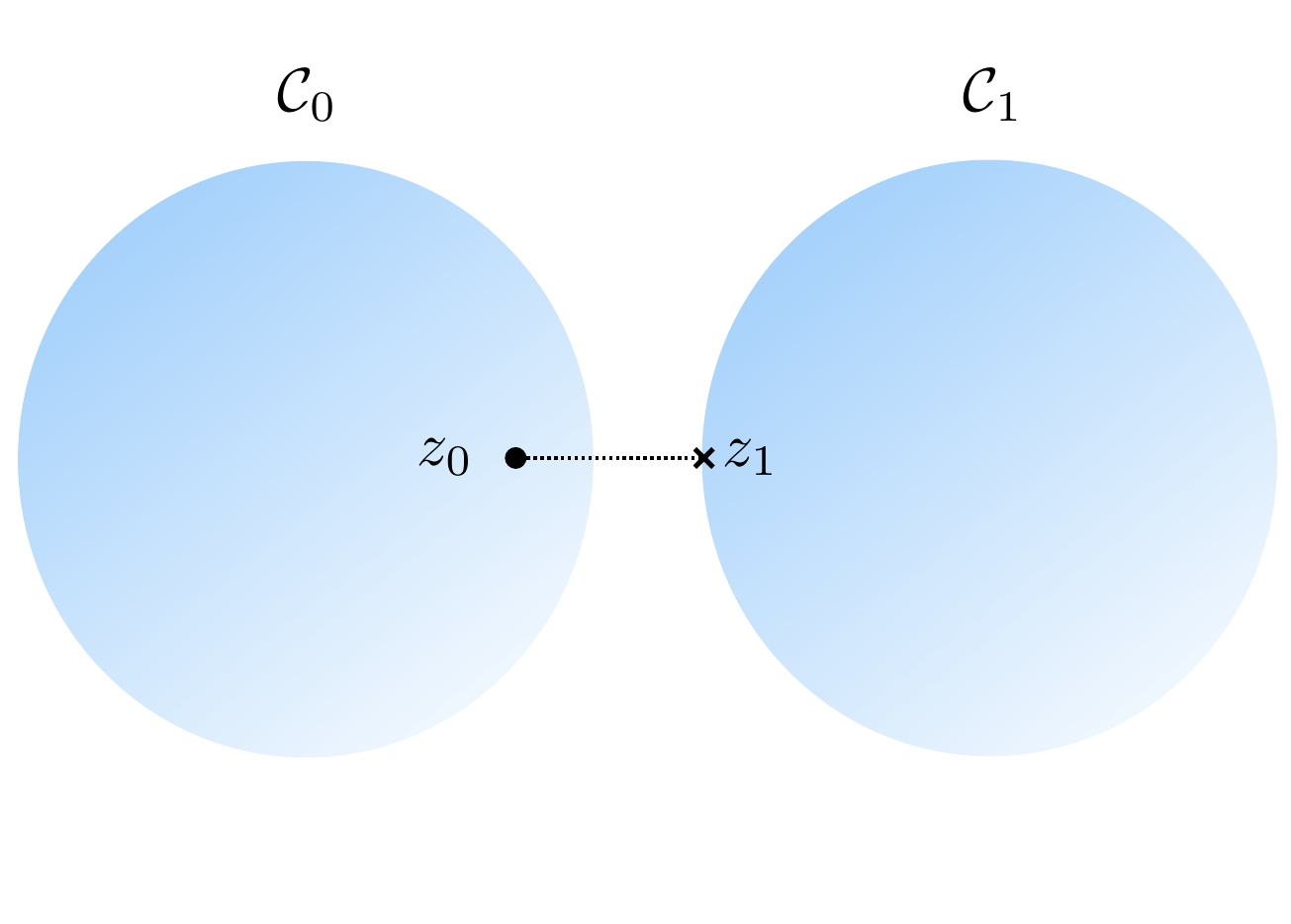}
			&\includegraphics[height=0.3\textwidth]{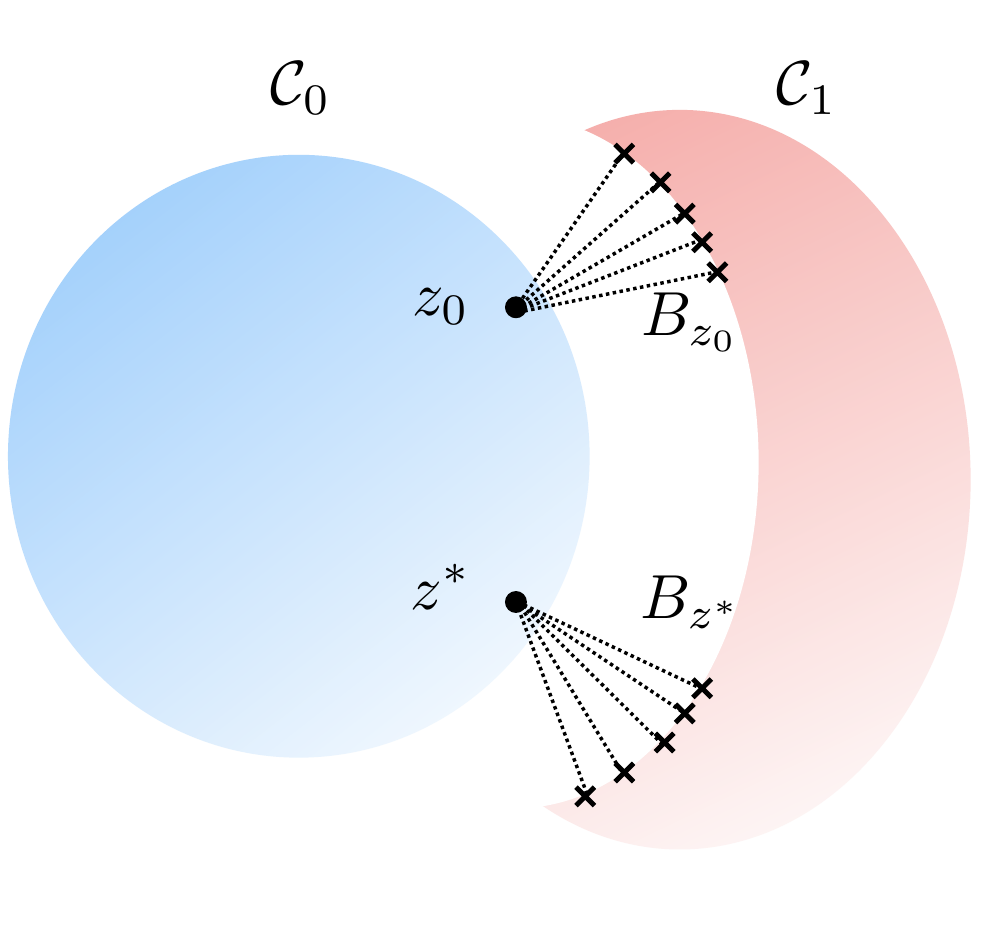}
		\end{tabular}
		\caption{The boundary of $\cC_1$ given $z_0$. On the left panel, the projection is unique and $B_{z_0} = z_1$. On the right panel, the projection is not unique. $z_0$ is the shadowing assignment of $z^*$. $B_{z^*}$ and $B_{z_0}$ have similar structures.}\label{fig:projection}
	\end{figure}

Since we want to characterize the suprema in the LRT, this motivates us to define the boundary of the alternative properties as follows.
	\begin{definition}[Boundary of communities class]\label{def:boundary}
		For a given null assignment $z_0 \in \cC_0$, we define the boundary of $\cC_1$ as:
		$$B_{z_0}= \big \{ z \in \cC_1: d(z_0,z ) = d(z_0,\cC_1) \big \}$$ 
		\end{definition}
By the definition above, $B_{z_0}$ is the projection of $z_0$ onto $\cC_1$ using the distance  in Definition \ref{def:clusterdist}. As we explained in Remark \ref{rm:wang}, the maximizer of the likelihood in $\cC_1$ might not be unique. Our analysis shows that the likelihood maximizer is asymptotically equivalent to the boundary $B_{z*}$.\footnote{See Lemma \ref{lm1-t1} in the Appendix for the rigorous argument.}   The later scenerio will be more challenging. In the next section, we will study the asymptotic property of the LRT under such case by studying the structure of $B_{z_0}$.

	\subsection{Shadowing Bootstrap for the Property Test}\label{mtd}
	In this section, we propose a bootstrap method to estimate the asymptotic quantile for the likelihood ratio statistic in \eqref{eq:lrt}. 

There are two major challenges to estimate the quantile of LRT. First, the suprema of the likelihood is not unique and therefore the limiting distribution of LRT is not necessarily normal. By \eqref{eq:lrt-lead}, we can in turn study the limiting distribution of the leading term
\begin{equation}\label{eq:T}
     L:= \sup_{z \in \mathcal{C}_1} g(p,q)\bigg( \sum_{(i,j) \in \mathcal{E}_2(z^*,z)} \Ab_{ij}-\sum_{(i,j) \in \mathcal{E}_1(z^*,z)} \Ab_{ij} \bigg),
\end{equation}
which is a maximum of a sequence of empirical processes indexed by $z$. \cite{chernozhukov2013gaussian} studied the limiting distribution of the maximal of high dimensional empirical process and proposed to estimate its quantile by multiplier bootstrap. However, their method restricts the scaling condition that the dimension $d$ of the empirical process and the sample size $n$ satisfies $\log d/n^{1/5} = o(1)$. However, in \eqref{eq:T}, the dimension $d = |\cC_1|$ could be of the order $
K^{n}$ and violates the scaling condition. To handle such problem, our key observation is that the supreme over the alternative $\cC_1$ can be represented by the supreme over its boundary $B_{z*}$.  In particular, we show that the leading term $L$ is asymptotically the same as the following statistic:
\begin{equation}\label{eq:T0}
    L_0 := \sup_{z \in B_{z^*}} g(p,q)\bigg( \sum_{(i,j) \in \mathcal{E}_2(z^*,z)} \Ab_{ij}-\sum_{(i,j) \in \mathcal{E}_1(z^*,z)} \Ab_{ij} \bigg).
\end{equation}
We refer to Section \ref{sec:lrt-proof} in the Appendix for the detailed proof. The cardinality  of $B_{z^*}$ is much smaller than the one of $\cC_1$. In Table \ref{tab:rate-even}, we can see that for example, $|B_{z^*}|$ is of the order  polynomial to $n$ and therefore satisfies the scaling condition of high dimensional multiplier bootstrap. 

Although $B_{z^*}$ is much smaller than $\cC_1$, we cannot construct $B_{z^*}$ in practice as $z^*$ is unknown. This leads to the second challenge: how to find $B_{z^*}$ in practice? Our key insight to solve the second challenge is to utilize the symmetry property in Definition \ref{def:cluster class}. This insight relies on the following lemma characterizing the covariance of two processes.

	\begin{lemma}[Shadowing symmetry]\label{lm2-t1}
	  For a given $z \in \mathcal{C}_0$, we list the assignments in the boundary $B_{z}$ as $z_1, z_2, \ldots, z_{|B_{z}|}$. Define a $|B_{z}|$-dimensional vector $\bL_{z}$ as
	 \[
(\bL_{z})_k = g(p,q)\bigg( \sum_{(i,j) \in \mathcal{E}_2(z,z_k)} \Ab_{ij}-\sum_{(i,j) \in \mathcal{E}_1(z,z_k)} \Ab_{ij} \bigg), \text{ for } k=1,2,\ldots, |B_{z}|.
	 \]
	 	Suppose Assumption \ref{asmp:symtest}	holds. For any $z_0,z_0' \in \cC_0$, we have $|B_{z_0}| = |B_{z_0'}|$ and  $\Cov(\bL_{z_0})$ equals to $\Cov(\bL_{z_0'})$ up to permutation, i.e., there existing a permutation $\uptau \in S_{|B_{z_0}|}$ such that $\Cov(\bL_{z_0})_{kl} = \Cov(\bL_{z_0'})_{\uptau(k)\uptau(l)}$ for all $k,l = 1, \ldots, |B_{z_0}|$.
	 \end{lemma}

We refer to Section \ref{sec: proof lm2-t1-g} in the Appendix for proof of Lemma \ref{lm2-t1}. We call the above lemma as the shadowing symmetry lemma, because it implies that the covariance of $\bL_{z_0}$ is same up to permutation to any other ``shadowing assignment" $z'_0 \in \cC_0$. Therefore, we can avoid directly constructing $B_{z^*}$. Instead, we can choose any $z \in \cC_0$ as a ``shadowing assignment" and consider
the shadowing statistic
\begin{equation}\label{eq:Lz0}
    L_0(z_0) := \sup_{z \in B_{z_0}} g(p,q)\bigg( \sum_{(i,j) \in \mathcal{E}_2(z_0,z)} \Ab_{ij}-\sum_{(i,j) \in \mathcal{E}_1(z_0,z)} \Ab_{ij} \bigg), 
\end{equation}
illustrated in Figure \ref{fig:projection}(b). Applying Lemma \ref{lm2-t1}, the following proposition shows that  the quantile of $L_0(z_0)$ is asymptotically same as the quantile of $L_0$.

\begin{proposition}\label{prop: eq-quant}
  Suppose Assumption \ref{asmp:symtest}
holds, $\log |B_{z^*}| = O(\log n)$ and $1/\rho_n =o(n^{1-c_2})$ for some constant $c_2 >0$. For any $z_0 \in \cC_0$, we have
\[ 
     \lim_{n \rightarrow \infty} \sup_{t \in \RR}
     |\PP(L_0 <t) - \PP(L_0(z_0) < t)| = 0. 
\]
\end{proposition}



We defer the proof to Appendix \ref{sec: proof of invary quant on z0}. Now we are ready to present the shadowing bootstrap procedure. Based on the previous discussion, we aim to estimate the quantile of $L_0(z_0)$. To achieve this, we take an arbitrary $z_0 \in \mathcal{C}_0$, and generate one realization of the adjacency matrix $\hat{\Ab} \sim \cM(n, K, \hat p, \hat q, z_0)$. Here $\hat p$ and $\hat q$ are the maximum likelihood estimator
\begin{equation}\label{eq: p q hat}
    (\hat{p},\hat{q})=\operatorname{argmax}_{(p,q)}\sup_{{z} \in \cC_0 \cup \cC_1}f(\Ab;z,p,q)
\end{equation}
 The likelihood ratio statistic is
\begin{equation}\label{eq:lrthat}
       \hat{\LRT} = \log \sup_{{z} \in \mathcal{C}_1} f(\Ab;z,\hat{p},\hat{q}) - \log\sup_{{z} \in \mathcal{C}_0 \cup \mathcal{C}_1} f(\Ab;z,\hat{p},\hat{q}).
\end{equation}

The next step is to find the assignments in the boundary $B_{z_0}$. We can construct $B_{z_0}$ by Definition \ref{def:boundary} in general. We refer to Section \ref{sec:case} on how to construct $B_{z_0}$ for the concrete examples. To estimate the quantile of $L_0(z_0)$, we apply the Gaussian multiplier bootstrap. Let $\{e_{ij}\}_{1\le i<j \le n}$ be independent standard Gaussian random variables and define
\begin{equation} \label{eq:Wn}
   W_n = \sup_{z \in B_{z_0}}\sum_{1\le i<j \le n}\big( {\hat{\Ab}}_{ij} - \mathbb{E}_{\hat p, \hat q} ({\hat{\Ab}}_{ij})\big) \big( \mathbbm{1}[(i,j) \in \mathcal{E}_2(z_0,z)] -\mathbbm{1}[(i,j) \in \mathcal{E}_1(z_0,z)]\big)e_{ij},
\end{equation}
where $\mathbb{E}_{\hat p, \hat q} ({\hat{\Ab}}_{ij}) =\hat p$ if $z_0(i) = z_0(j)$ and $\mathbb{E}_{\hat p, \hat q} ({\hat{\Ab}}_{ij}) =\hat q$ otherwise. Let $C_W(\alpha)$ be the $1-\alpha$ quantile of $W_n$ conditioning on $\hat \Ab$ and $\Ab$, i.e., $\PP(W_n \le C_W(\alpha) | \hat \Ab, \Ab)  = 1- \alpha$. We then estimate the quantile of  $\LRT$ by 
\begin{equation}\label{eq:qa}
    q_{\alpha} = g(\hat{p},\hat{q})C_W(\alpha) + g(\hat{p},\hat{q})\hat{\mu}_0,   
\end{equation}
where $\hat{\mu}_0 =d(\mathcal{C}_0,\mathcal{C}_1)(\hat{q}-\hat{p})$ is the estimator of mean of the process in \eqref{eq:Lz0}. Finally, we reject the null $\mathrm{H}_0: z^* \in \cC_0$ if $\hat \LRT \ge q_{\alpha}$ and do not reject $\mathrm{H}_0$ otherwise.

	\subsubsection{Case Study of the Boundary}\label{sec:case}
	In this part, we provide concrete algorithm to construct the boundary $B_{z_0}$ for Examples \ref{ex:same} and \ref{ex:group}. We can find  $B_{z_0}$  via computationally  efficient algorithms for both examples. Meanwhile, we will also calculate $d(\cC_0, \cC_1)$ needed in \eqref{eq:qa}.
		\begin{figure}[htpb]
		\centering
		\begin{tabular}{cc}
			\includegraphics[width=0.35\textwidth]{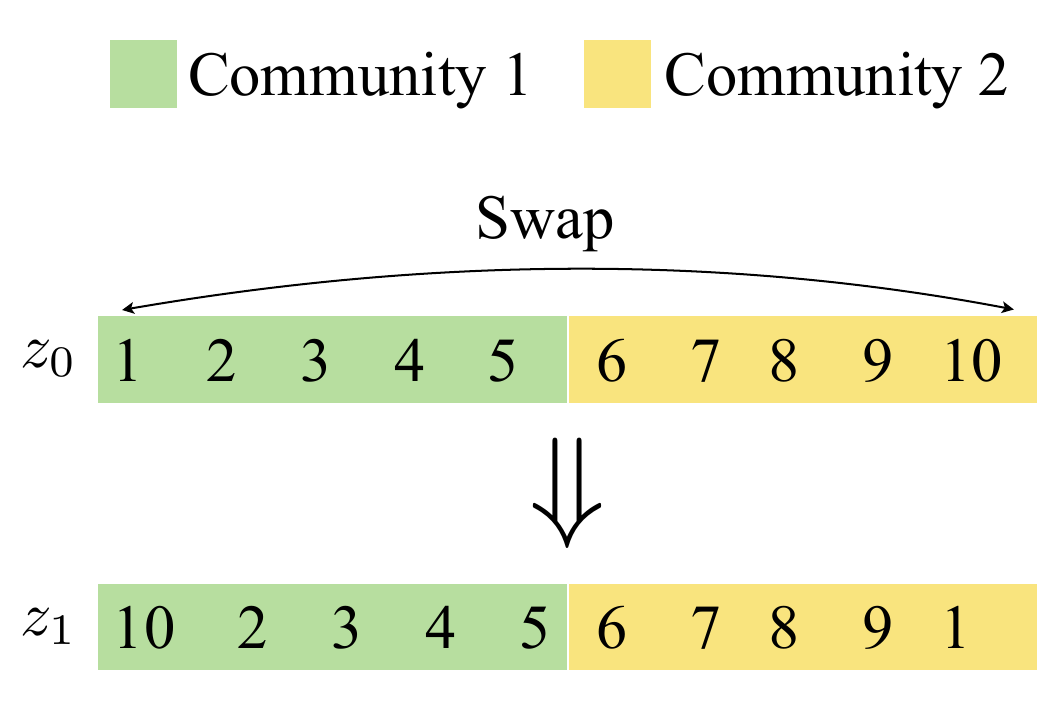}
			&\includegraphics[width=0.35\textwidth]{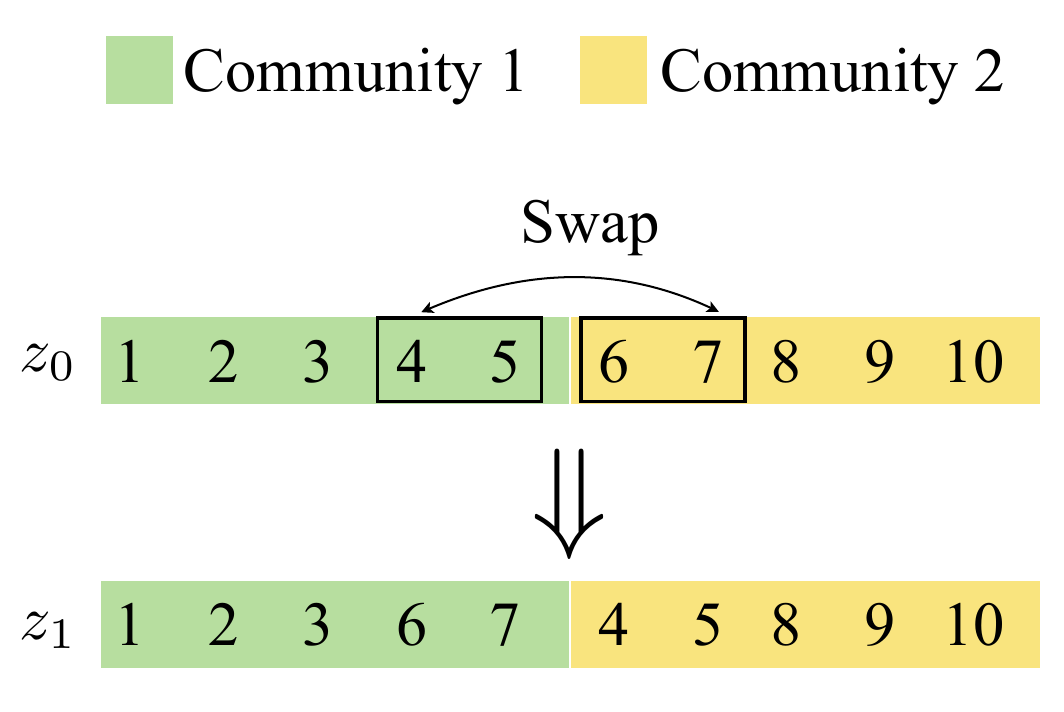}\\
			(a) Example \ref{ex:same} & (b) Example \ref{ex:group}
		\end{tabular}
		\caption{Procedure to construct the assignment in the boundary $B_{z_0}$.  Panel (a) is to test whether the first 3 nodes belong to the same cluster, and Panel (b) is to test whether node sets $\{1,2,3\}$ and $\{4,5\}$ belong to the same cluster.}\label{fig:swap}
	\end{figure}
	
$\bullet$ Example \ref{ex:same}: Same community test for $m$ nodes. Recall that $\cC_0$ and $\cC_1$ are defined in \eqref{eq:c0ex1}. For any $z_0 \in \cC_0$, to find assignments in $B_{z_0}$, we aim to find assignments whose distance to $z_0$ is $d(z_0, \cC_1)$. The simplest way is to exchange the community assignment of one node $s \in [m]$ with another node $s'$ from a different community (see $z_1$ in Figure~\ref{fig:swap}(a) for an example when $n=10$, $m=3$, and $K=2$). It is easy to check all such assignments belong to $B_{z_0}$. On the other hand, any other operation will incur more node-wise misclassification and the edge-wise misalignment will be much larger. To find $d(\cC_0, \cC_1)$, we start with evaluating $n_1(z_0,z_1)$ for some $z_1 \in B_{z_0}$. The edges whose connection probability is changed from $p$ to $q$ will be the edges between node $s$ and the rest of the nodes in its original community and between node $s'$ and the rest of the nodes in its original community. Therefore, we have $n_1(z_0, z_1) = 2(n/K-1)$. Similarly, $n_2(z_0, z_1) = 2(n/K-1)$ and thus $d(\mathcal{C}_0,\mathcal{C}_1)= 2(n/K-1)$.

In summary,
$B_{z_0}$ is composed of all the assignments which can be obtained from reassigning the label of one of $z_0$'s node in $[m]$ to a different community.
The distance between two classes is $d(\mathcal{C}_0,\mathcal{C}_1)= 2(n/K-1)$.

$\bullet$ Example \ref{ex:group}: Same community test for groups. Recall that $\cC_0$ and $\cC_1$ are defined in \eqref{eq:c0ex2}. In this example, without loss of generality we can assume that $m' \le m$. Then to project an arbitrary $z_0 \in \cC_0$ onto $\mathcal{C}_1$, we will exchange the cluster assignment of the set $\mathcal{S}_{m'}=\{m+1,\ldots, m+m'\}$ with another set $\mathcal{S}_{m'}'$ from a different cluster of cardinality $m'$ to obtain the smallest number of edges that are misaligned. See Figure~\ref{fig:swap}(b).  Correspondingly, for $z_1 \in B_{z_0}$, we have $d(z_0,z_1)=n_1(z_0, z_1)=n_2(z_0, z_1)=2m\wedge m'(n/K-m \wedge m')$, and thus $d(\mathcal{C}_0,\mathcal{C}_1)= 2m\wedge m'(n/K-m\wedge m')$.

In summary, $B_{z_0}$ is composed of all the assignments which can be obtained from reassigning the label of  nodes $m+1,\ldots, m+m'$ in $z_0$. The distance between two classes is $d(\mathcal{C}_0,\mathcal{C}_1)= 2m\wedge m'(n/K-m\wedge m')$.

	\section{Validity of Community Property Test}\label{sec: theory}
In this section, we show the theoretical results that our testing method is honest and powerful. Before presenting our theorems, we first give the following assumption for the alternative class $\mathcal{C}_1$.

	\begin{assumption}[{Scattering of  $\mathcal{C}_1$}]\label{asmp: sct} For any $z_0 \in \mathcal{C}_0$, we have $|B_{z_0}| = O(n^{c_0})$ for some constant $c_0>0$. \end{assumption}
	\begin{remark}
	We call this assumption as the {\it scattering assumption} as it ensures that the assignments in $\mathcal{C}_1$ are uniformly scattered in $\cC_1$  and there are not too many assignments  concentrating  on the boundary. In specific, we assume the cardinality of the boundary $B_{z_0}$ is at most polynomial to $n$. In Section \ref{sec:case}, we construct   $B_{z_0}$ for Examples \ref{ex:same} and \ref{ex:group} and they both satisfy this assumption. We refer to Proposition \ref{prop:sct} or Table \ref{tab:rate-even} for the specific rates of $|B_{z_0}|$ under each example.
	\end{remark}
	Recall that  $q_{\alpha}$ in \eqref{eq:qa} is our estimator of the $1-\alpha$ quantile of the likelihood ratio statistic $\hat \LRT$. We reparamterize $p,q$ as $p =\rho_n \lambda_1$ and $q =\rho_n \lambda_2$, where we assume $\lambda_1$ and $\lambda_2$ are constants independent to $n$.  The following main theorem shows that our test is honest and powerful for general symmetric community properties.
	
	\begin{theorem}\label{t1-mtd}
Suppose Assumptions \ref{asmp:symtest}  and \ref{asmp: sct} hold, $d(\mathcal{C}_0, \mathcal{C}_1) = o(n^{c_1})$ for some constant $c_1 < 2$, and $1/ \rho_n =o(n^{1-c_2})$ for some constant $c_2>0$. We have $$\lim_{n \rightarrow \infty} \sup_{z^* \in \cC_0} \mathbb{P}(\hat\LRT \geq q_{\alpha}) =\alpha \text{ and }
\lim_{n \rightarrow \infty} \sup_{z^* \in \cC_0} \mathbb{P}(\text{reject } {\rm H}_0) = \alpha.
$$
Moreover, if $d(\cC_0,\cC_1)I(p,q)= \Omega(n^{\varepsilon})$ for some arbitrarily small constant $\varepsilon>0$, we have
\[
\lim_{n \rightarrow \infty} \inf_{z^* \in \cC_1} \mathbb{P}(\text{reject } {\rm H}_0) = 1.
\]
	\end{theorem}
	\begin{remark}
We defer the proof of theorem to Appendix~\ref{sec:lrt-proof}. The scaling condition $d(\cC_0,\cC_1)I(p,q)= \Omega(n^{\varepsilon})$ in the theorem demonstrates the {\it combinatorial-probabilistic} trade-off in the community property test. In order to differentiate two community properties $\cC_0$ versus $\cC_1$, we know that both the combinatorial distance between $\cC_0$ and $\cC_1$ and the distance  between two assignment probabilities  $p$ and $q$ should be large enough. Theorem \ref{t1-mtd} implies that the combinatorial distance can be measured by $d(\cC_0, \cC_1)$ and the probabilistic distance can be  measured by the Renyi divergence $I(p,q)$. Our test is powerful if the product of two distance increases faster than $n^{\varepsilon}$ for some arbitrarily small constant $\varepsilon>0$.
	\end{remark}

We now apply Theorem \ref{t1-mtd} to Examples \ref{ex:same} and \ref{ex:group} by checking Assumption~\ref{asmp: sct}. Due to the discussion in Section \ref{sec:case}, for Example \ref{ex:same}, $B_{z_0}$ is composed of all the assignments in $\mathcal{C}_1$ that can be obtained by swapping one node in $[m]$ with another node from a different cluster (see Figure \ref{fig:swap}(a) for illustration). Therefore, $|B_{z_0}| = m (n/K)(K-1) =O(mn)$. For Example~\ref{ex:group}, without loss of generality, we assume $m' \le m$. Therefore, $B_{z_0}$ is composed of all the assignments which can be obtained from reassigning the label of nodes  nodes $m+1,\ldots, m+m'$ in $z_0$ (see Figure \ref{fig:swap}(b) for illustration). Therefore, $|B_{z_0}| = (K-1) {n/K \choose m\wedge m'} = O(K (n/K)^{m\wedge m'})$. The $ {n/K \choose m\wedge m'}$ term is for choosing the set $\cS_{m'}'$ from a different community. Combining the discussion on $d(\cC_0, \cC_1)$ in Section \ref{sec:case}, we summarize the results in the following proposition.

\begin{proposition}\label{prop:sct}
For Examples \ref{ex:same}, we have $|B_{z_0}| = O(mn)$ and thus it satisfies Assumption \ref{asmp: sct}. We also have $d(\cC_0, \cC_1) = 2(n/K-1) = O(n/K)$. For Examples \ref{ex:group}, we have $|B_{z_0}| = O(K (n/K)^{m \wedge m'})$ which  satisfies Assumption \ref{asmp: sct}. We also have $d(\cC_0, \cC_1) = 2(m \wedge m')(n/K - m \wedge m')$.
\end{proposition}
Plugging these results to the general Theorem \ref{t1-mtd}, we have the following two corollaries.

	\begin{corollary}[Examples \ref{ex:same} and \ref{ex:group}]\label{case1-mtd-col}
		Suppose $1/ \rho_n =o(n^{1-c_2})$ for some constant $c_2>0$. For $\cC_0$ and $\cC_1$ in \eqref{eq:c0ex1} for any $m \le n/K$ or \eqref{eq:c0ex2} for $m \wedge m' = O(1)$, our test for the hypothesis ${\rm H}_0: z^* \in \cC_0$ versus ${\rm H}_1: z^* \in \cC_1$ is honest, i.e.,
		$$\lim_{n \rightarrow \infty} \sup_{z^* \in \cC_0} \mathbb{P}(\text{reject } {\rm H}_0) = \alpha.
		$$
		Moreoever, if $I(p,q)n/K= \Omega(n^{\varepsilon})$ for some small positive constant $\varepsilon$, we have
		$$\lim_{n \rightarrow \infty} \sup_{z^* \in \cC_1} \mathbb{P}(\text{reject } {\rm H}_0) = 1.$$
	\end{corollary}

	\section{Information-Theoretic Lower Bound}\label{lwrbd}

We have shown that our shadowing bootstrap method is honest and powerful when the product of the combinatorial distance and probabilistic distance satisfies $d(\cC_0, \cC_1) I(p,q) = \Omega(n^{\epsilon})$ for some small $\epsilon>0$. In this section, we will discuss the information-theoretic lower bound of community property test. We will give the lower bound of the minimax risk of all possible test $\psi$ for $\mathrm{H}_0: z \in \cC_0 \text{~v.s.~} \mathrm{H}_1: z \in \cC_1$, defined as
\[
   r(\cC_0, \cC_1) = \inf_{\psi} \Big\{\sup_{z\in \cC_0} \PP_z(\psi = 1) + \sup_{z\in \cC_1} \PP_z(\psi = 0)\Big\}.
\]
We will show that the combinatorial-probabilistic trade-off phenomenon appears in the lower bound as well, thus it essentially characterizes the hardness of the community property test.

\subsection{Packing Number of Communities}

In order to establish the lower bound, we first  introduce the concept of packing number of community class $\cC_1$. Similar to the minimax theory of the hypothesis testing for continuous parameters or the graph properties \citep{MR1742500,MR3909951}, we find that the packing number is also essential in the lower bound of community properties test. 

A key element in the definition of the packing number is the metric assigned to the community class $\cC_1$. 
Recall the community property distance $d(\cC_0,\cC_1)$ in Definition \ref{def:clusterdist}. It counts the misaligned edges $\cE_1(z_0,z_1)$ and $\cE_2(z_0,z_1)$ in Definition \ref{def:e12} for all $z_0 \in \cC_0$ and $z_1 \in \cC_1$. Our first insight is that the more misaligned edges there are, the easier it is to differentiate $\cC_1$ from $\cC_0$. This motivates us to consider the misaligned edge set $\cE_{1,2}(z_0,z_1)=\cE_1(z_0,z_1) \cup \cE_2(z_0,z_1)$ and use its cardinality as a ``metric" in the following definition of packing number. Our second insight is that how hard it is to differentiate $\cC_1$ from $\cC_0$ does not depends on the complexity of the entire set $\cC_1$ but the boundary set $B_{z_0}$ in Definition \ref{def:boundary}. Our shadowing bootstrap statistic in \eqref{eq:Wn} implies that $B_{z_0}$ is representative to $\cC_1$. Therefore, we give the following definition of packing number of  $B_{z_0}$ to characterize the hardness of test.



\begin{definition}[$\varepsilon$-packing of $B_{z_0}$]\label{def:pack} For any $z_0 \in \cC_0$, we say $ \{z_1,z_2,...,z_N\} \subseteq B_{z_0}$ is an $\varepsilon$-packing of $B_{z_0}$, if for any $z_j \neq z_k$ we have $|\cE_{1,2}(z_0,z_j) \cap \cE_{1,2}(z_0,z_k)| \leq \varepsilon$. The $\varepsilon$-packing number of $B_{z_0}$, denoted as $N(B_{z_0}, \varepsilon)$, is the maximum cardinality of any $\varepsilon$-packing of $B_{z_0}$.
\end{definition}

\begin{figure}[h]
	\centering
	\includegraphics[width=0.4\textwidth]{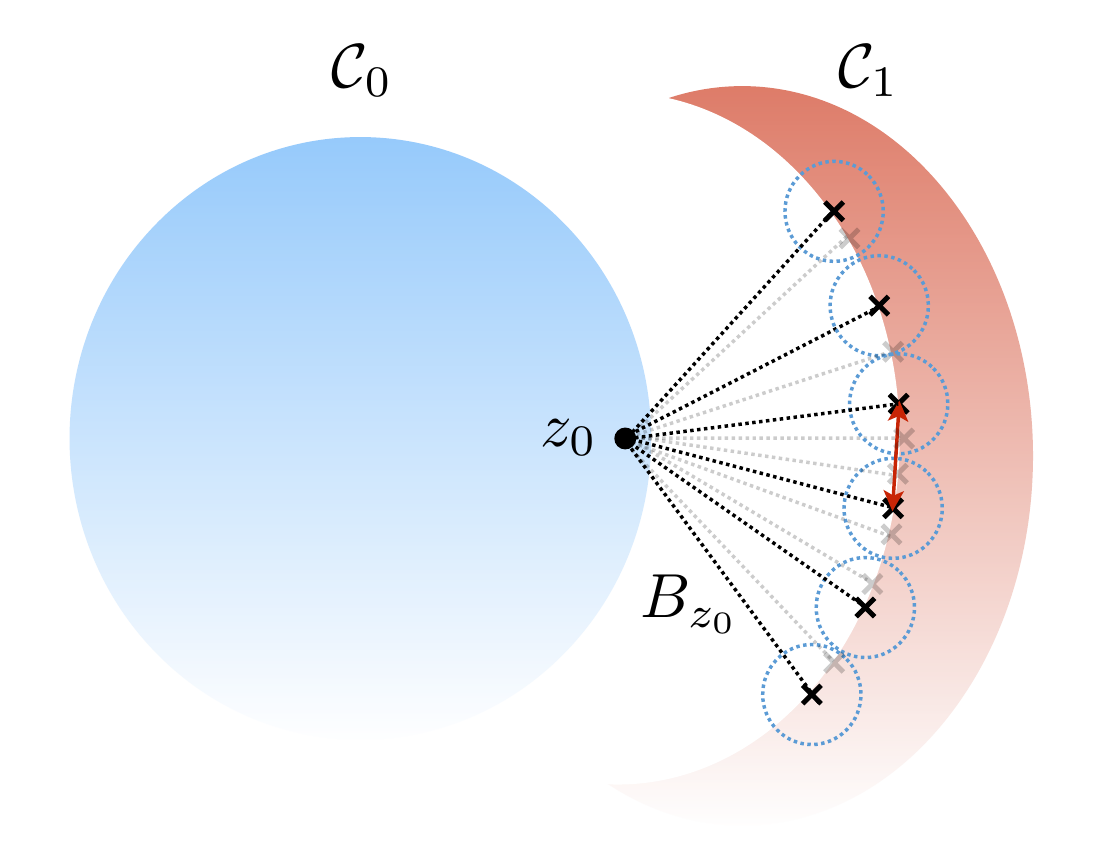}
	\caption{The $\varepsilon$-packing set of $B_{z_0}$. }\label{fig:lb-packing}
\end{figure}

We illustrate the packing set of $B_{z_0}$ in Figure \ref{fig:lb-packing}. 
By Definition \ref{def:boundary}, $B_{z_0}$ collects the alignments in $\cC_1$ which are closest to $z_0 \in \cC_0$. Therefore, these alignments are the hardest cases to test. The following theorem shows the lower bound of the community property test can be characterized by the packing number of these hardest cases.

\begin{theorem}\label{lb-main}
 	Suppose  $\cC_0, \cC_1 \subseteq \cK^n$, $1/ \rho_n =o(n^{1-c_2})$ for some constant $c_2>0$ and $p\le 1-\delta$ for some constant $\delta>0$. If there exists a $z_0 \in \cC_0$ such that $\log N\big (B_{z_0},\sqrt{d(z_0,\cC_1)} \big ) = O(\log n)$ and 
 	\begin{equation}\label{eq:cp-low}
 	    \limsup_{n \rightarrow \infty} \frac{d(z_0,\mathcal{C}_1)I(p,q)}{\log N\big (B_{z_0},\sqrt{d(z_0,\cC_1)} \big )} < 1,
 	\end{equation}
	then $\liminf\limits_{n \rightarrow \infty} r(\cC_0, \cC_1) \geq 1/2$.
\end{theorem}
\begin{remark}
We defer the proof of the theorem to Appendix \ref{proof: lb main sec}. The combinatorial-probabilistic trade-off in the lower bound is characterized by \eqref{eq:cp-low}. We cannot differentiate $\cC_1$ from $\cC_0$ if 
\[
{d(z_0,\mathcal{C}_1)I(p,q)} < \log N\big (B_{z_0},\sqrt{d(z_0,\cC_1)}\big), 
\]
for sufficiently large $n$. The packing entropy $\log N (B_{z_0},\sqrt{d(z_0,\cC_1)})$ is the lower bound of the signal strength. Our lower bound shows that the packing entropy of community class plays a similar role as the packing entropy in parametric hypothesis test \citep{MR1742500} and in graph property test \citep{MR3909951}. We derive the rate of packing entropy for Examples \ref{ex:same} and \ref{ex:group} in Proposition \ref{prop: lb-case study-pac num}. In general, the packing entropy is $O(\log n)$. Comparing to the upper bound $d(\cC_0,\cC_1)I(p,q)= \Omega(n^{\varepsilon})$ for some arbitrarily small constant $\varepsilon>0$ in Theorem \ref{t1-mtd}, there is a gap to $O(\log n)$ in the lower bound. We conjecture that this gap exists as both our upper and lower bounds are for general community property test. We  will find a finer analysis in future research. 
\end{remark}

The following theorem gives an alternative lower bound result relaxing the scaling conditions in Theorem \ref{lb-main}.

\begin{theorem}\label{lb-sec} Suppose  $\cC_0, \cC_1 \subseteq \cK^n$, $0<q<p \leq 1-\delta$ for some constant $\delta>0$ and $\lim_{n \rightarrow \infty}d(\cC_0,\cC_1) p  = \infty$. If one of the following conditions:
\begin{enumerate}
    \item  $d(\cC_0,\mathcal{C}_1)I(p,q) \le c$ 	for some sufficiently small constant $c$; 
    \item $\lim_{n \rightarrow \infty}d(\cC_0,\mathcal{C}_1)I(p,q) = \infty$,  but there exists a $z_0 \in \cC_0$ such that 
    \[
    \limsup_{n \rightarrow \infty} \frac{d(z_0,\mathcal{C}_1)I(p,q)}{\log N(B_{z_0},0)}< 1,
    \]
\end{enumerate}
 is satisfied, then $\liminf\limits_{n \rightarrow \infty} r(\cC_0, \cC_1) \geq 1/2$.
\end{theorem}

We defer the proof of the theorem to Appendix \ref{proof: lb main sec}. Notice that the scaling condition on $d(z_0, \cC_1)p$  is different from the one on $d(z_0,\cC_1)I(p,q)$ in the lower bound. $I(p,q)$ measures the difference between $p$ and $q$, whereas the condition $d(z_0,\cC_1)p = \Omega(\log n)$ is to guarantee that the edge connection probability cannot be too small.  Theorems~\ref{lb-main} and \ref{lb-sec} both show the lower bound with the combinatorial-probabilistic trade-off. Theorem~\ref{lb-main} has a sharper lower bound on $d(z_0,\mathcal{C}_1)I(p,q)$ under a stronger scaling condition. In comparison, Theorem~\ref{lb-sec} has a less sharp lower bound with weaker scaling conditions. 
When $d(\cC_0,\mathcal{C}_1)I(p,q)$ is bounded, we cannot differentiate two hypotheses. When  $d(\cC_0,\cC_1)$ goes to infinity, Theorem~\ref{lb-sec} condition (2) shows the lower bound $d(\cC_0,\mathcal{C}_1)I(p,q) < \log N(B_{z_0},0)$. If we have stronger scaling conditions in Theorem~\ref{lb-main}, we get a sharper lower bound $d(\cC_0,\mathcal{C}_1)I(p,q) < \log N(B_{z_0},\sqrt{d(z_0,\cC_1)})$.

\subsection{Case Study of Lower Bound}
In this section we apply the general theorems for the lower bound to Examples~\ref{ex:same} and \ref{ex:group}. By \eqref{eq:cp-low} in Theorem~\ref{lb-main}, a key quantity for the lower bound is the packing number $N(B_{z_0}, \sqrt{d(z_0, \cC_1)})$. The following proposition gives concrete results for the two examples.

\begin{proposition}\label{prop: lb-case study-pac num}
The packing number for Example~\ref{ex:same} is $N(B_{z_0}, \sqrt{d(z_0, \cC_1)}) = m$, and the packing number for Example~\ref{ex:group} is $N(B_{z_0}, \sqrt{d(z_0, \cC_1)}) = N(B_{z_0},0) = 1$.
\end{proposition}

The proof of the proposition is deferred to Appendix \ref{sec: proof even pack num}.  Table~\ref{tab:rate-even} summarizes important quantities for Examples~\ref{ex:same} and \ref{ex:group}.
\begin{table}
    \centering
	\begin{tabular}{c|c|c|c} 
		\hline
		\hline
		  & $d(\cC_0,\cC_1)$ &  $|B_{z_0}|$ & $N(B_{z_0},\sqrt{d(z_0,\cC_1)})$ \\ [0.5ex] 
		\hline
	Example	\ref{ex:same} & $2({n}/{K}-1)$ & $O(mn)$ & $m$\\ 
		\hline
	Example	\ref{ex:group} & $2m\wedge m'({n}/{K}-m \wedge m')$ & $O (K(n/K)^{m \wedge m'})$ & 1\\ [1ex]
		\hline
		\hline
	\end{tabular}
	\caption{Important values for even cases of Example~\ref{ex:same} and Example~\ref{ex:group} }\label{tab:rate-even}
\end{table}

Recall that $\lambda_1 = p/\rho_n$ and $\lambda_2 = q/\rho_n$. We present the lower bound of two examples below.
Applying  Theorem~\ref{lb-main} and Proposition~\ref{prop: lb-case study-pac num}, we have the following lower bound of same community test in Example \ref{ex:same}.
	\begin{corollary}\label{col-c1}
	For $\cC_0$ and $\cC_1$ defined in \eqref{eq:c0ex1}, if $1/ \rho_n =o(n^{1-c_2})$ for some constant $c_2>0$, $p<1-\delta$ for some constant $\delta>0$ and 
	\[
	\limsup\limits_{n \rightarrow \infty} 2n I(p,q)/(K \log m) <1,
	\]
	we have $\liminf\limits_{n \rightarrow \infty} r(\cC_0, \cC_1) \geq 1/2$.
	\end{corollary}

Applying  Theorem~\ref{lb-sec} and Proposition \ref{prop: lb-case study-pac num}, we have the following lower bound of same community test for groups in Example \ref{ex:group}.
	
	\begin{corollary}\label{col-c2}
		For $\cC_0$ and $\cC_1$ defined in \eqref{eq:c0ex2}, if $n p \rightarrow \infty$, $0 < q<p \le 1-\delta$ for some constant $\delta>0$ and 
		\[
		\limsup\limits_{n \rightarrow \infty} nI(p,q)<c,
		\]
		for some sufficiently small constant $c>0$, we have $\liminf\limits_{n \rightarrow \infty} r(\cC_0, \cC_1) \geq 1/2$.
	\end{corollary}

\section{General Framework for  Uneven Community Sizes}\label{sec: general framework}
In this section, we generalize our theory to the community property tests when  the community sizes in $\cC_0$ and $\cC_1$ are not necessarily  even, e.g., Example \ref{ex:num}. 
For any $z \in \cC_0 \cup \cC_1$, denote the community size $n_k(z) = |\{z(i)=k~|~ i \in [n]\}|$ for $k \in [K]$.
Let 
\begin{equation}\label{eq:ck}
      c_K = \max_{z\in \cC_0 \cup \cC_1}\max_{1 \le k \le K}|n_k(z) - n/K|.
\end{equation}
When the community sizes are even, we have $c_K = 0$. In this section, we consider the cases when $c_K$ could be larger than zero.
We will show that
the shadowing bootstrap method in Section \ref{mtd} can be applied to test the uneven community property as well. The information-theoretic lower bound is also similar to the one in Section \ref{lwrbd}. 

\subsection{General Symmetric Community Properties}

For the uneven community class, we still need some symmetry property for the assignments in $\cC_0$ and $\cC_1$. 
When community sizes are even, Definition \ref{def:cluster class} depicts the symmetry via the representative node set $\cN$ and the representative assignment $\tilde z$. However, for many community properties of interest, e.g., the community size test in Example \ref{ex:num}, we cannot find such $\cN$ and $\tilde z$. In Example \ref{ex:num}, we are interested in testing the community size and thus there is no representative nodes.
See Figure \ref{fig:exp} for  illustration.

Therefore, we define the following generalized symmetric community property pair.
\begin{definition}[Generalized symmetric community property pair] \label{def: sym} We say two disjoint community properties $\cC_0$ and $\cC_1$ is a {\it generalized symmetric property pair} if for any $z, z' \in \cC_0$, there exist permuations $\sigma \in S_K$ and $\tau \in S_n$ such that 
\begin{enumerate}
    \item $\tau \circ \sigma (z) := (\sigma(z(\tau(1))), \ldots, \sigma(z(\tau(n)))) = z'$ and
    \item $\cC_1$ is also closed under such transform $\tau\circ\sigma$, i.e., for any $z'' \in \cC_1$, $\tau\circ\sigma (z'') \in \cC_1 $.
\end{enumerate}
\end{definition}

Definition \ref{def: sym} generalizes the concept of symmetric community property in  Definition \ref{def:cluster class} via introducing the permutation transform. We can check that Examples \ref{ex:same} and \ref{ex:group} are still  symmetric by Definition \ref{def: sym}. See Figure \ref{fig:exp}(a) for an example of choosing $\sigma$ and $\tau$. On the other hand, the community sizes properties
\begin{equation}\label{eq:c0ex3}
    \cC_0 = \{z \in [K]^n: \text{all community sizes } = n/K\} \text{ and } \cC_1 = \cC_0^c,
\end{equation}
are also symmetric by Definition \ref{def: sym} but not Definition \ref{def:cluster class}. See Figure \ref{fig:exp}(b) for illustration.
 In fact, the following proposition shows that Definition \ref{def:cluster class} is a special case of Definition \ref{def: sym}. 

\begin{proposition}\label{prop:sym}
If $\cC_0, \cC_1 \subseteq \cK^n$ satisfy Assumption \ref{asmp:symtest}, then $\cC_0$ and $\cC_1$ is a generalized symmetric property pair. Moreover, the property pairs in \eqref{eq:c0ex1}, \eqref{eq:c0ex2} and \eqref{eq:c0ex3} are generalized symmetric property pairs.
\end{proposition}

We defer the proof of the proposition to Appendix~\ref{sec: proof of prop 6.1}. In Figure \ref{fig:exp}, we show how to choose concrete permutation transforms $\sigma$ and $\tau$ for Examples~\ref{ex:same} and~\ref{ex:num}.

	\begin{figure}[htbp]
	\centering
   \begin{tabular}{cc}
		\includegraphics[width=0.4\textwidth]{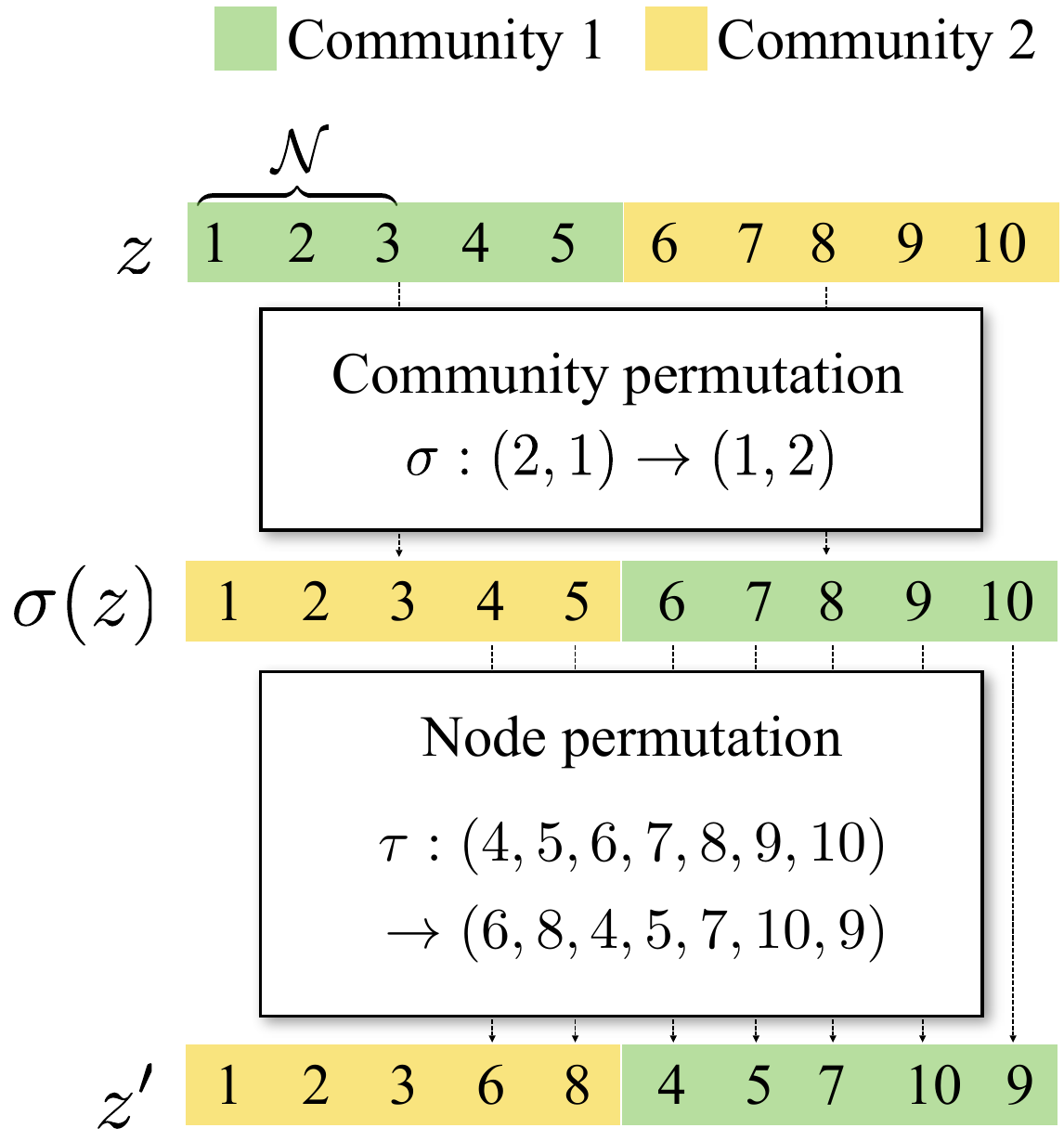}
		&\includegraphics[width=0.4\textwidth]{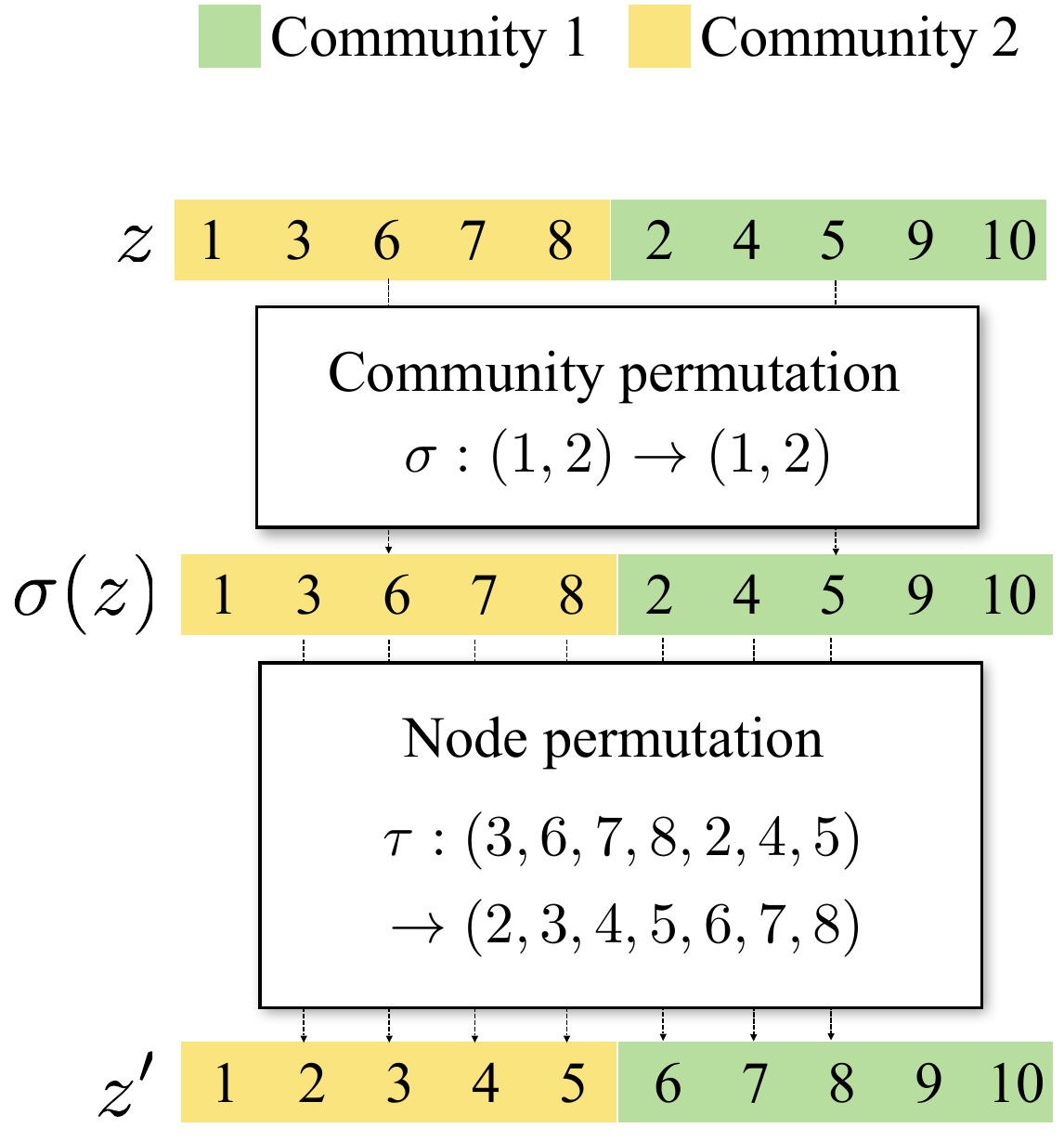}\\
				(a) Example \ref{ex:same} & (b) Example \ref{ex:num}	
	\end{tabular}
	\caption{Permutation of null assignments in Example~\ref{ex:same} and Example~\ref{ex:num}}\label{fig:exp} 
\end{figure}

\subsection{Shadowing Bootstrap for General Case}

We now generalize the testing method proposed in Section~\ref{mtd} to the uneven case. A key step is to generalize the boundary $B_{z_0}$ in Definition \ref{def:boundary}. Recall that for the even case, our insight is that the statistic $L$ in \eqref{eq:T} taking the supremum over $\cC_1$ is asymptotically equal to the $L_0$ in \eqref{eq:T0} taking the supremum over  $B_{z^*}$, which is much smaller than $\cC_1$.
Similar insight applies to the uneven case using  the following generalized definition of boundary.
\begin{definition}\label{def:boundary-general}
	For a given $z_0 \in \cC_0$, we define the boundary centered at $z_0$ with radius $r$ as
	$$B_{z_0} (r) = \{z \in \cC_1|d(z_0,z) \leq r \}.$$
\end{definition}

We illustrate the two types of boundary in Figure~\ref{fig:bz0_r1}. From  Figure~\ref{fig:bz0_r1}(a), we can see that $B_{z_0} = B_{z_0}(d(\cC_0,\cC_1))$. Therefore, Definition \ref{def:boundary-general} is a generalization of Definition \ref{def:boundary}. For the uneven case, $L_0$ is no longer asymptotically equal to $L$.  We need to enlarge $B_{z_0}$ to $B_{z_0} (r)$ for some $r > d(\cC_0,\cC_1)$ and modify the statistic $L_0$ in \eqref{eq:T0} by taking the supremum over $B_{z^*} (r)$.

\begin{figure}[htpb]
		\centering
		\begin{tabular}{cc}
			\includegraphics[height=0.3\textwidth]{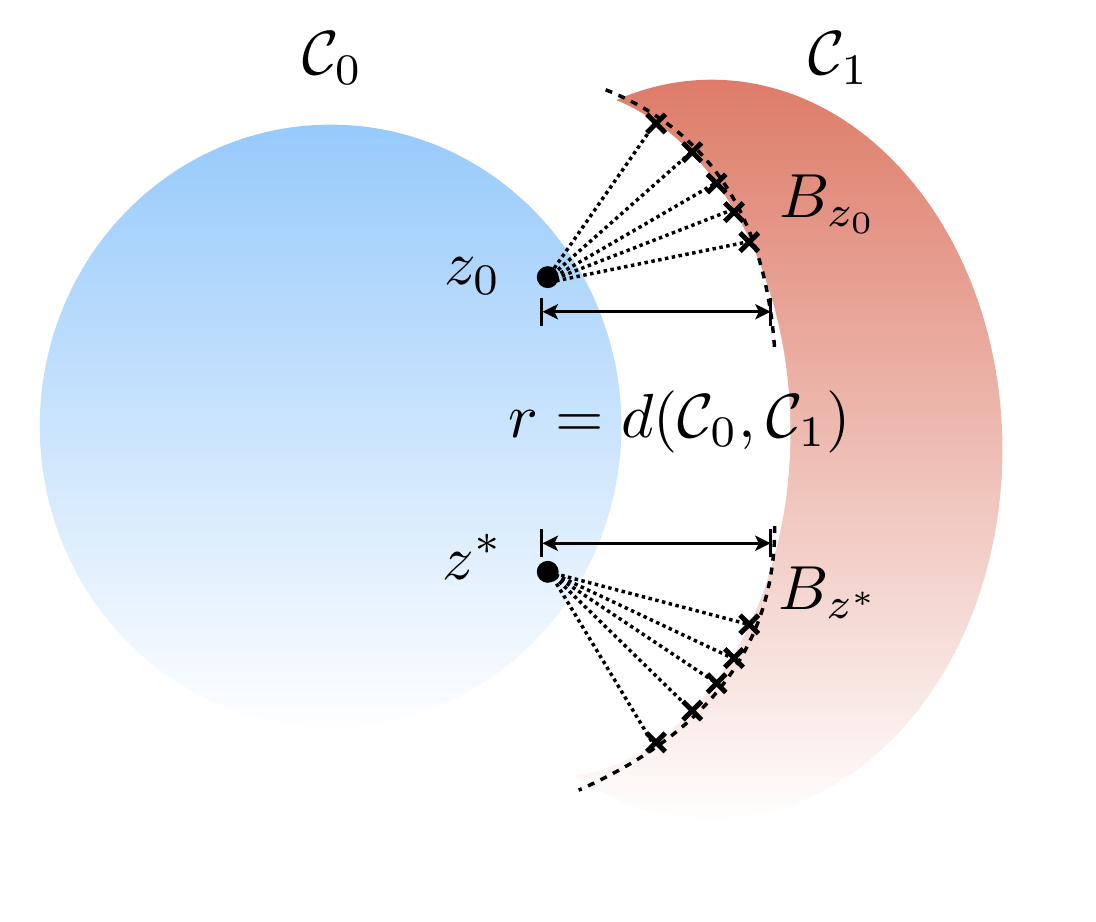}
			&\includegraphics[height=0.3\textwidth]{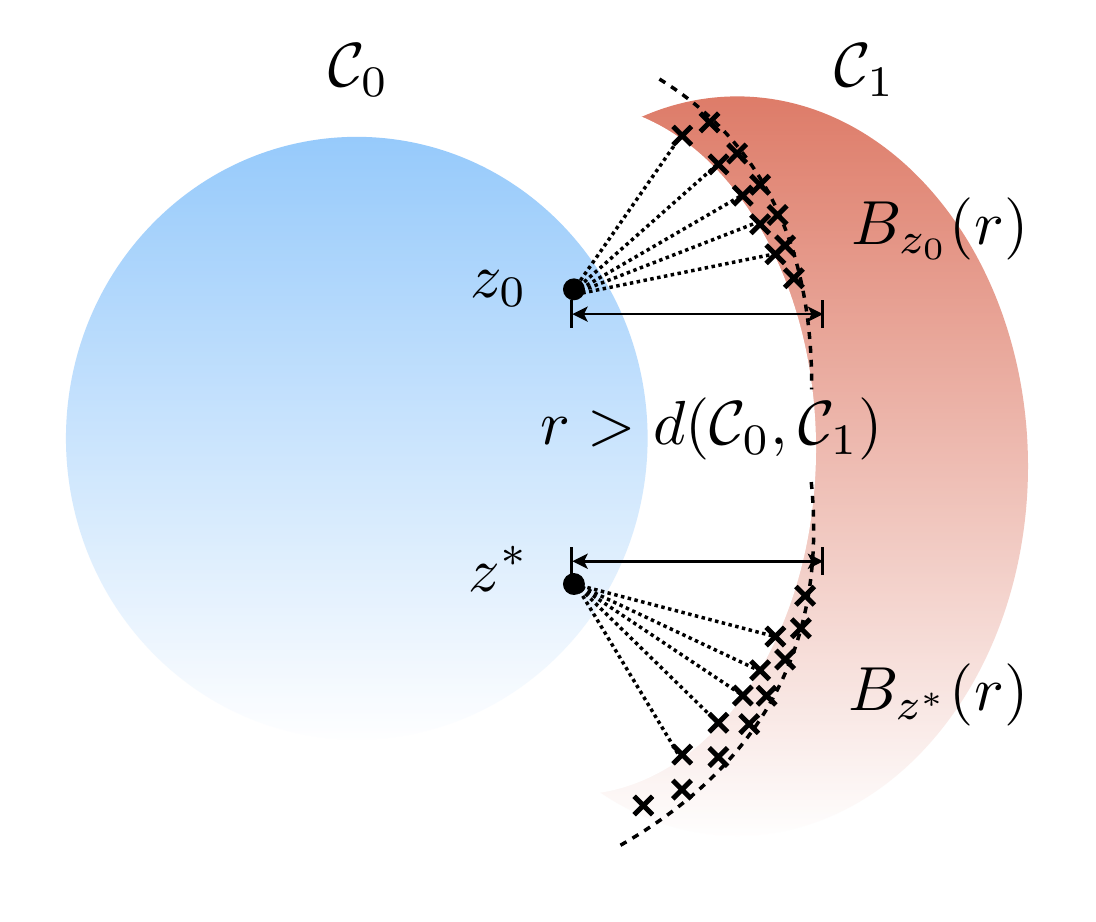}\\
			(a) Boundary $B_{z_0}$ in Definition \ref{def:boundary} & (b) Generalized boundary $B_{z_0}(r)$
		\end{tabular}
		\caption{The boundary $B_{z_0}$ defined previously for even cases is in essence a ball centered at $z_0$ with radius $r= d(\cC_0,\cC_1)$}\label{fig:bz0_r1}
	\end{figure}

In fact, we can still use the shadowing bootstrap method in Section \ref{mtd} to the uneven case. All procedures are exactly same as Section \ref{mtd} except that we only need to replace the bootstrap statistic $W_n$ in \eqref{eq:Wn} by
\begin{equation}\label{eq:Wn-g}
     W_n = \sup_{z \in B_{z_0}(r)}\sum_{1\le i<j \le n}\big( {\hat{\Ab}}_{ij} - \mathbb{E}_{\hat p, \hat q} ({\hat{\Ab}}_{ij})\big) \big( \mathbbm{1}[(i,j) \in \mathcal{E}_2(z_0,z)] -\mathbbm{1}[(i,j) \in \mathcal{E}_1(z_0,z)]\big)e_{ij},
\end{equation}
where $r$ is a tuning parameter to be specified in the following theorem.

	\begin{theorem}\label{t1-mtd-g}
Suppose $\cC_0$ and $\cC_1$ are generalized symmetric community property pair and $c_K = O(1)$. Suppose $d(\mathcal{C}_0, \mathcal{C}_1) = o(n^{c_1})$ for some constant $c_1 < 2$, and $1/ \rho_n =o(n^{1-c_2})$ for some constant $c_2>0$. We choose the radius $r$ in \eqref{eq:Wn-g} as  $r \geq r_K:= d(\cC_0,\cC_1)+{c_K^2{{p} K}}/(2({p}-{q}))$ and $r=d(\cC_0,\cC_1) + O(1)$. If  for any $z_0 \in \cC_0$, we have $|B_{z_0}(r)| = O(n^{c_0})$ for some positive constant $c_0$, then $$\lim_{n \rightarrow \infty}\sup_{z^* \in \cC_0} \mathbb{P}(\hat\LRT \geq q_{\alpha}) =\alpha \text{ and }
\lim_{n \rightarrow \infty} \sup_{z^* \in \cC_0} \mathbb{P}(\text{reject } {\rm H}_0) = \alpha.
$$
Moreover, if $d(\cC_0,\cC_1)I(p,q)= \Omega(n^{\varepsilon})$ for some arbitrarily small constant $\varepsilon>0$, we have
\[
\lim_{n \rightarrow \infty} \inf_{z^* \in \cC_1} \mathbb{P}(\text{reject } {\rm H}_0) = 1.
\]
	\end{theorem}
We defer the proof of theorem to Appendix \ref{sec:lrt-proof}. 
The scaling assumptions in Theorem~\ref{t1-mtd-g} are similar to Theorem~\ref{t1-mtd}. The condition $|B_{z_0}(r)| = O(n^{c_0})$ for some $c_0>0$ is similar to Assumption \ref{asmp: sct}. We need $c_K$ in \eqref{eq:ck} to be bounded to prevent a specific community from being too large.  By the theorem, we need to choose $r \ge r_K:= d(\cC_0,\cC_1)+{c_K^2{{p} K}}/(2({p}-{q}))$, while $p, q, c_K$ are unknown. In practice, we suggest to choose the radius as $r = d(\cC_0,\cC_1)+C{\hat{p} K}/{(\hat{p}-\hat{q})}$ for some sufficiently large $C$. In fact, for many concrete examples, even though $r_K$ is unknown, we can directly construct $B_{z_0}(r_K)$. The following proposition shows how to construct $B_{z_0}(r_K)$  for Examples~\ref{ex:same}-\ref{ex:num}.  Moreover, it shows the conditions on $d(\cC_0,\cC_1)$ and $|B_{z_0}(r_K)|$ in Theorem \ref{t1-mtd-g} are true for all these examples.

\begin{proposition}\label{prop:bz-cases}
For any $z_0 \in \cC_0$, $B_{z_0}(r_K)$ can be constructed as follows.
\begin{enumerate}
    \item Example~\ref{ex:same}: $B_{z_0}(r_K)$ is composed of all the assignments obtained from  reassigning one node of any $z_0 \in \cC_0$ in $[m]$ to a different community. See Figure \ref{fig:exp uneven move}(a) for an illustration.  Moreover, we have $d(\cC_0,\cC_1) = n/K$ and $|B_{z_0}(r_K)| = m(K-1)$. 
    \item Example~\ref{ex:group}: Suppose $m \wedge m' \le c_K$, $B_{z_0}(r_K)$ is composed of all the assignments obtained from reassigning nodes $m+1,\ldots, m+m'$ in any $z_0 \in \cC_0$ collectively to a different community. Moreover, we have $d(\cC_0,\cC_1) = n(m \wedge m')/K$ and $|B_{z_0}(r_K)| = K-1$. Suppose $m \wedge m' > c_K$, $B_{z_0}(r_K)$ is composed of all the assignments obtained from exchanging label of nodes $m+1,\ldots, m+m'$ collectively with another $m'$ nodes from a different community for any $z_0 \in \cC_0$. See Figure \ref{fig:exp uneven move}(b) for an illustration. Moreover, we have $d(\cC_0,\cC_1) = 2m \wedge m'(n/K-m \wedge m')$ and $|B_{z_0}(r_K)| = O(K(n/K)^{m \wedge m'})$.
    \item Example~\ref{ex:num}: For an arbitrary $z_0 \in \cC_0$, $B_{z_0}(r_K)$ can be constructed by reassigning any node of $z_0$ to a different community. See Figure \ref{fig:exp uneven move}(c) for an illustration. Moreover, we have $d(\cC_0,\cC_1) = n/K$ and  $|B_{z_0}(r_K)| = n(K-1)$.
\end{enumerate}
\end{proposition}

We defer the proof to Appendix~\ref{sec: proof of bz-g}. The construction of $B_{z_0}(r_K)$ is visualized in Figure \ref{fig:exp uneven move}. We also summarize the results in Table \ref{tab:rate-general}. 

 	\begin{figure}[htbp]
	\centering
   \begin{tabular}{ccc}
		\includegraphics[height=0.2\textwidth]{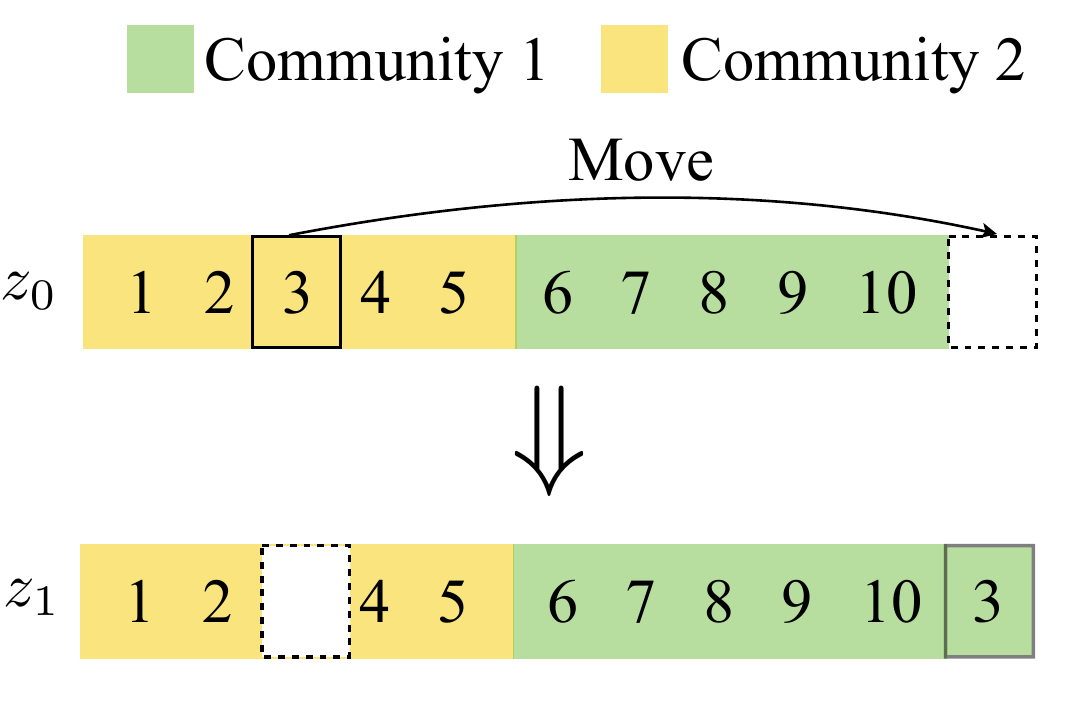}
		&\includegraphics[height=0.2\textwidth]{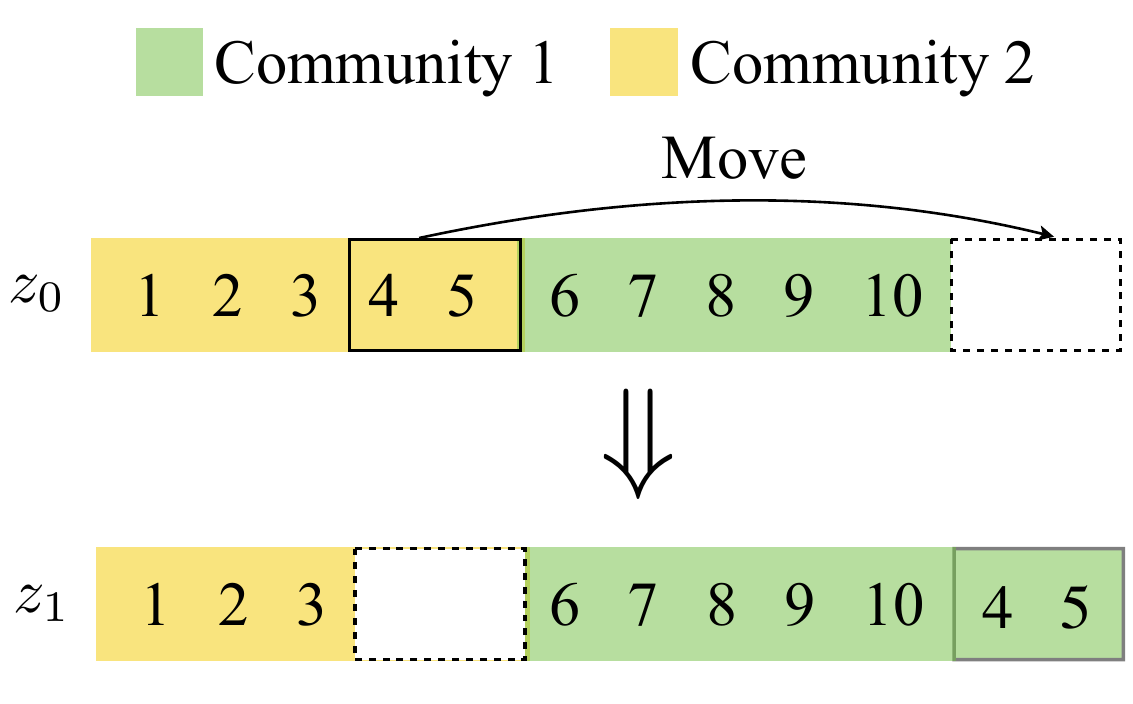}
		&\includegraphics[height=0.2\textwidth]{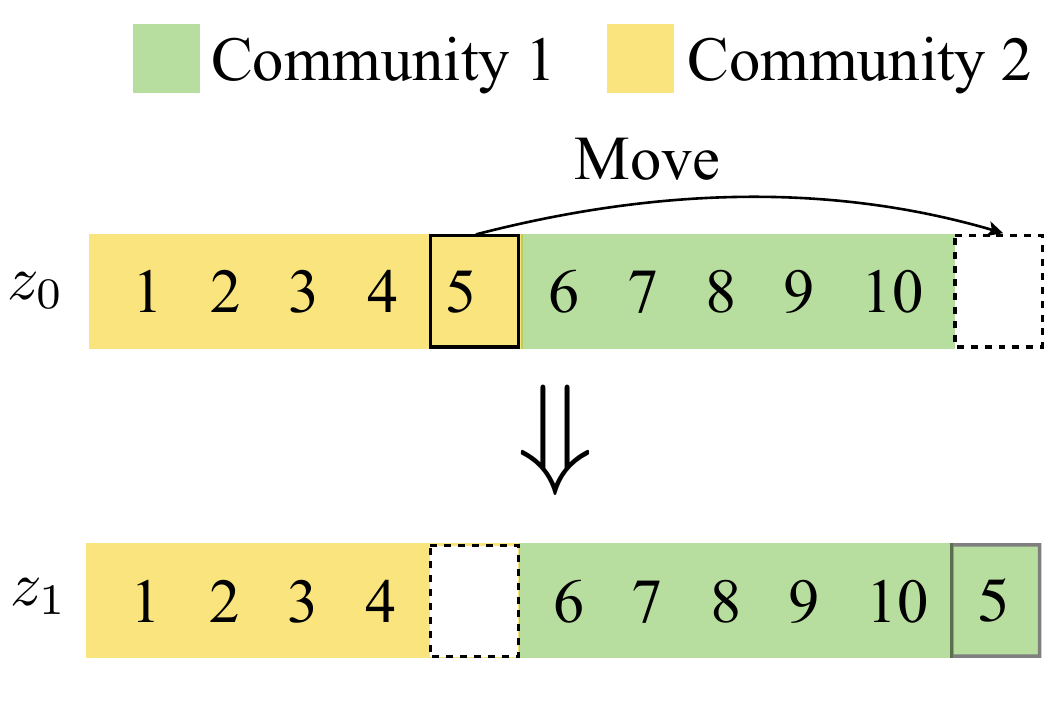}\\
				(a) Example \ref{ex:same} & (b) Example \ref{ex:group} & (c) Example \ref{ex:num}	
	\end{tabular}
	\caption{Construction of $B_{z_0}$ in Proposition \ref{prop:bz-cases}: (a) $\cC_0$ is that nodes $\{1,2,3 \}$ belong to the same community; (b) $\cC_0$ is that the nodes set $\{1,2,3\}$ and $\{4,5\}$ belong to the same community; (c) $\cC_0$ is that community 1 and community 2 have equal size of 5. }\label{fig:exp uneven move}
\end{figure}
 
 We therefore have the following corollary of Theorem \ref{t1-mtd-g}.
 	\begin{corollary}[Examples \ref{ex:same} -\ref{ex:num}]\label{case1-mtd-col-g}
		Suppose $1/ \rho_n =o(n^{1-c_2})$ for some constant $c_2>0$ and $c_K = O(1)$. We assume that $m\wedge m' =O(1)$ in Example \ref{ex:group}. For Examples \ref{ex:same} -\ref{ex:num}, with $B_{z_0}(r_K)$ constructed in Proposition~\ref{prop:bz-cases} our test for the hypothesis ${\rm H}_0: z^* \in \cC_0$ versus ${\rm H}_1: z^* \in \cC_1$ is honest, i.e.,
		$$\lim_{n \rightarrow \infty} \sup_{z^* \in \cC_0} \mathbb{P}(\text{reject } {\rm H}_0) = \alpha.
		$$
		Moreoever, if $I(p,q)n/K= \Omega(n^{\varepsilon})$ for some small positive constant $\varepsilon$, we have
		$$\lim_{n \rightarrow \infty} \sup_{z^* \in \cC_1} \mathbb{P}(\text{reject } {\rm H}_0) = 1.$$
	\end{corollary}

\subsection{General Lower Bound}
We can also generalize the information-theoretic lower bound in Theorem \ref{lb-main} to the uneven case. Similar to the even case, we need to define packing number of $B_{z_0}(r)$, which follows the  same definition of $N\big (B_{z_0}, \varepsilon\big)$ in Definition \ref{def:pack}. We then have the lower bound of the general case as follows.

\begin{theorem}\label{lb-main-g}
 	Suppose $1/ \rho_n =o(n^{1-c_2})$ for some constant $c_2>0$, $p \le 1-\delta$ for some constant $\delta>0$ and $c_K=O(1)$. If there exists a $z_0 \in \cC_0$ and some $r=d(z_0,\cC_1) + O(1)$ such that  $\log N\big (B_{z_0}(r),\sqrt{d(z_0,\cC_1)} \big ) = O(\log n)$, and 
 	\begin{equation}\label{eq:cp-low-g}
 	    \limsup_{n \rightarrow \infty} \frac{d(z_0,\mathcal{C}_1)I(p,q)}{\log N\big (B_{z_0}(r),\sqrt{d(z_0,\cC_1)} \big )} < 1,
 	\end{equation}
	then $\liminf\limits_{n \rightarrow \infty} r(\cC_0, \cC_1) \geq 1/2$.
\end{theorem}

\begin{remark}
 If we choose $r = d(z_0,\cC_1)$, as $B_{z_0} = B_{z_0}(d(z_0,\cC_1))$, \eqref{eq:cp-low-g} reduces to \eqref{eq:cp-low}. The relaxed assumption on $r = d(z_0,\cC_1) + O(1)$ can give us a better lower bound. For example, in the proof of Corollary \ref{col1-ex-rec},  we show a sharper minimax rate by using some $r > d(z_0,\cC_1)$. 
\end{remark}

We can also generalize Theorem \ref{lb-sec} to the following theorem.
\begin{theorem}\label{lb-sec-g} Suppose $0<q<p \leq 1-\delta$ for some constant $\delta>0$ and $\lim_{n \rightarrow \infty}d(\cC_0,\cC_1) p  = \infty$. If one of the following conditions:
\begin{enumerate}
    \item  $d(\cC_0,\mathcal{C}_1)I(p,q) \le c$ 	for some sufficiently small constant $c$; 
    \item $\lim_{n \rightarrow \infty}d(\cC_0,\mathcal{C}_1)I(p,q) = \infty$,  but there exists a $z_0 \in \cC_0$ and some $r=d(z_0,\cC_1) + O(1)$ such that $
    \limsup_{n \rightarrow \infty} {d(z_0,\mathcal{C}_1)I(p,q)}/{\log N(B_{z_0}(r),0)}< 1,
    $
\end{enumerate}
 is satisfied, then $\liminf\limits_{n \rightarrow \infty} r(\cC_0, \cC_1) \geq 1/2$.
\end{theorem}
We defer the proof of the above two theorems to Appendix \ref{proof: lb main sec}.

To apply the general lower bound theorem to Examples \ref{ex:same}-\ref{ex:num}, we need the following proposition on the packing number.

\begin{proposition}\label{prop: pk-num-g}
We have the packing number $N(B_{z_0}(r_K), \sqrt{d(z_0,\cC_1)})$ for three examples as follows:
\begin{itemize}
    \item Example~\ref{ex:same}: $N(B_{z_0}(r_K), \sqrt{d(z_0,\cC_1)}) = m$;
    \item Example~\ref{ex:group}: $N(B_{z_0}(r_K), \sqrt{d(z_0,\cC_1)}) = 1$;
    \item Example \ref{ex:num}: $N(B_{z_0}(r_K), \sqrt{d(z_0,\cC_1)}) = n$. 
\end{itemize}
\end{proposition}
We defer the proof to Appendix \ref{sec: proof pack num-g}. The results is also summarized in Table \ref{tab:rate-general}.

\begin{table}
    \centering
	\begin{tabular}{ c | c | c | c } 
		\hline
		\hline
		  & $d(z_0,\cC_1)$ &  $|B_{z_0}(r_K)|$ & $N(B_{z_0}(r_K),\sqrt{d(z_0,\cC_1)})$ \\ [0.5ex] 
		\hline
	Example	\ref{ex:same}  & ${n}/{K}$ & $m(K-1)$ & $m$\\ 
		\hline
	\makecell{Example	\ref{ex:group}  \\ $m\wedge m' \le c_K$}&  ${n (m \wedge m')}/{K}$ & $K-1$ & 1 \\ 
			\hline
	\makecell{Example	\ref{ex:group} \\ $m\wedge m' > c_K$}&  $2m \wedge m' (n/K - m \wedge m')$ & $O(K(n/K)^{m \wedge m'})$ & 1 \\ 
		\hline
	Example	\ref{ex:num}  & ${n}/{K}$ & $n(K-1)$ & $n$\\ [1ex]
		\hline
		\hline
	\end{tabular}
	\caption{Important values for general cases of Examples~\ref{ex:same}-\ref{ex:num}.}\label{tab:rate-general}
\end{table}

Since $r_K = d(\cC_0,\cC_1) + c_K^2 pK/(2(p-q))$, where $c_K^2 pK/(2(p-q))=O(1)$ and $d(\cC_0,\cC_1) = d(z_0,\cC_1)$ by the symmetry of $\cC_0,\cC_1$, we have that $r_K = d(z_0,\cC_1) +O(1)$. Applying Theorem~\ref{lb-main-g} and Proposition~\ref{prop: pk-num-g}, we have the following lower bound of same community test in Example \ref{ex:same}.
	\begin{corollary}\label{col-c1-g}
	For $\cC_0$ and $\cC_1$ defined in Example \ref{ex:same}, if $1/ \rho_n =o(n^{1-c_2})$ for some constant $c_2>0$, $p<1-\delta$ for some constant $\delta>0$, $c_K=O(1)$ and 
	\[
	\limsup\limits_{n \rightarrow \infty} n I(p,q)/(K \log m) <1,
	\]
	we have $\liminf\limits_{n \rightarrow \infty} r(\cC_0, \cC_1) \geq 1/2$.
	\end{corollary}

Applying  Theorem~\ref{lb-sec-g} and Proposition \ref{prop: pk-num-g}, we have the following lower bound of same community test for groups in Example \ref{ex:group}.
	
	\begin{corollary}\label{col-c2-g}
		For $\cC_0$ and $\cC_1$ defined in \eqref{eq:c0ex2}, if $n p \rightarrow \infty$, $0<q<p<1-\delta$ for some $\delta > 0$ and 
		\[
		\limsup\limits_{n \rightarrow \infty} nI(p,q)<c,
		\]
		for some sufficiently small constant $c>0$, we have $\liminf\limits_{n \rightarrow \infty} r(\cC_0, \cC_1) \geq 1/2$.
	\end{corollary}
For Example~\ref{ex:num}, applying Theorem~\ref{lb-main-g} and Proposition~\ref{prop: pk-num-g},  we have the following result.

	\begin{corollary}\label{col-c3-g}
		For $\cC_0$ and $\cC_1$ defined in \eqref{eq:c0ex3}, if $1/ \rho_n =o(n^{1-c_2})$ for some constant $c_2>0$, $p<1-\delta$ for some constant $\delta>0$ and 
	\[
	\limsup\limits_{n \rightarrow \infty} n I(p,q)/(K \log n) <1,
	\]
	we have $\liminf\limits_{n \rightarrow \infty} r(\cC_0, \cC_1) \geq 1/2$.
	\end{corollary}

Our lower bound result in Theorem \ref{lb-main-g} can also  provide a sharp threshold for exact recovery.
\begin{corollary}\label{col1-ex-rec}
	For the homogeneous SBM $\cM(n, K, p,q,z^*)$ with signal strength $\rho_n=n^{-c}$ for some $c \in (0,1)$ and $p<1-\delta$ for some $\delta>0$, when $\limsup\limits_{n \rightarrow \infty} nI(p,q)/(K \log n) < 1$, there exists $c_0 >0$ such that  $\inf_{\hat{z}}\sup_{z^*} \mathbb{P}(\hat{z} \neq z^*) \geq c_0$.
\end{corollary}

We defer the proof of corollary to Appendix \ref{sec: proof ex rec}. An optimal upper bound method can be provided by the MLE, which is proved in part~(1) of the proof for Theorem~3.2 in \cite{zhang2016minimax}. Thus we can see that our lower bound is sharp and gives the threshold of exact recovery. Currently the existing threshold focuses on the regime of $\rho_n = \log n/n$ and $\lambda_1, \lambda_2$ are constants that do not depend on $n$. In our method, we consider a different regime with $\rho_n=n^{-c}$ for some $c \in (0,1)$, and $\lambda_1 - \lambda_2$ is allowed to be $o(1)$.

\section{Numerical Results on Synthetic Data}\label{sec: num}
We conduct the shadowing bootstrap  on Examples~\ref{ex:same} and~\ref{ex:num}.
We test both hypotheses at the significance level $\alpha  = 0.05$. We consider the number of nodes $n = 200, 600, 1000$ and the number of clusters $K = 2$. The connection probabilities are set to be $p=(1+\Delta)\rho_n$ and $q=(1-\Delta) \rho_n$, where $\rho_n = (n/K)^{-0.3}$ and $\Delta$ is the parameter that controls the difference between $p$ and $q$. We choose $\Delta$ varying from $0$ to $0.8$. The maximum likelihood estimator is calculated via using singular value decomposition estimator as an initialization to boost the computation. For Example~\ref{ex:same}, we set $m = \lceil (n/K)^{\delta}/2 \rceil$ for $\delta = 0.3, 0.5, 0.7$ to explore the influence of $m$ on type-I and II errors.


 \begin{figure}[htp]
	\centering
	\begin{tabular}{cc}
	 (a) Different $n$ & (b) Different $m$\\
	\includegraphics[width=0.5\textwidth]{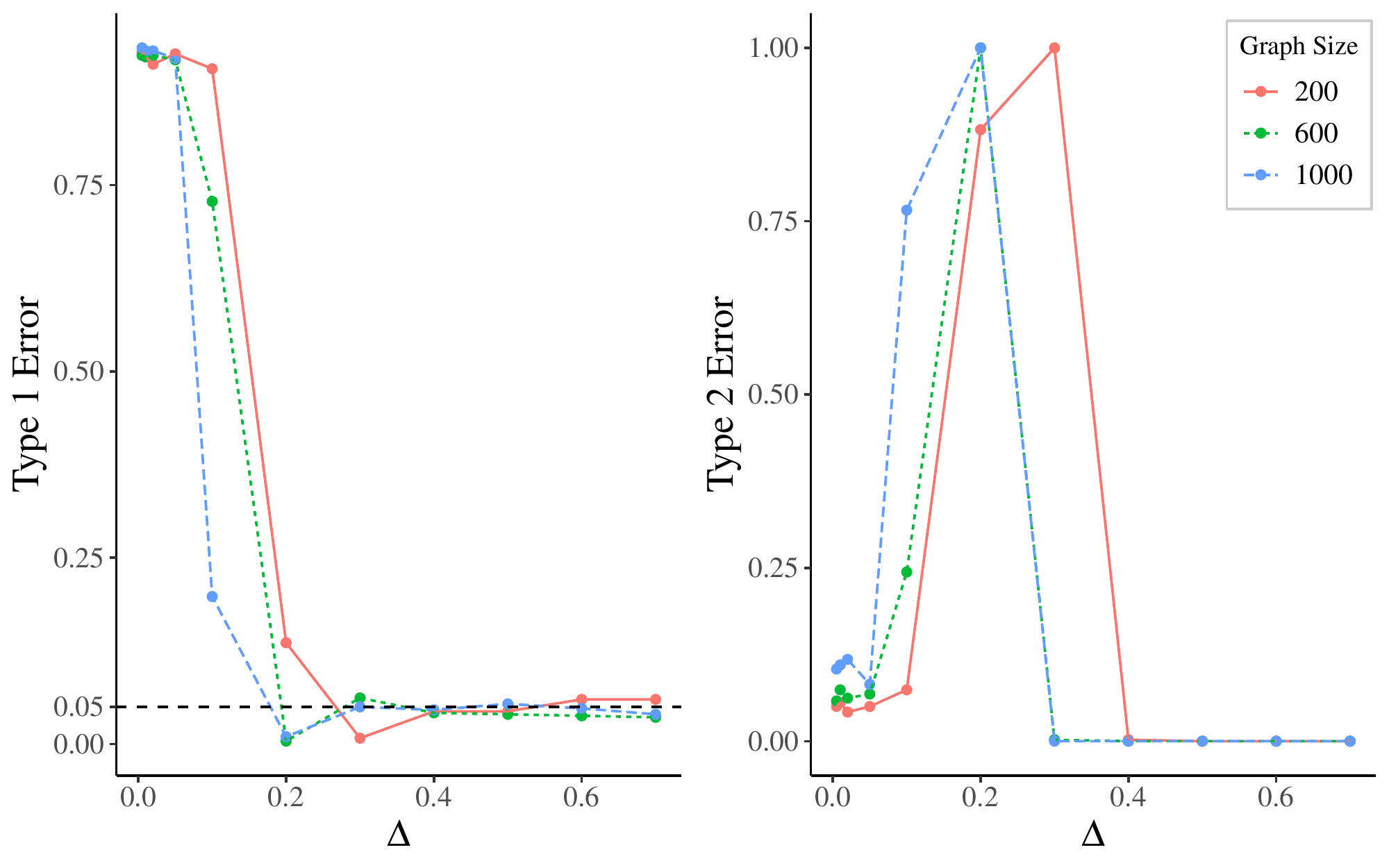}	&\includegraphics[width=0.5\textwidth]{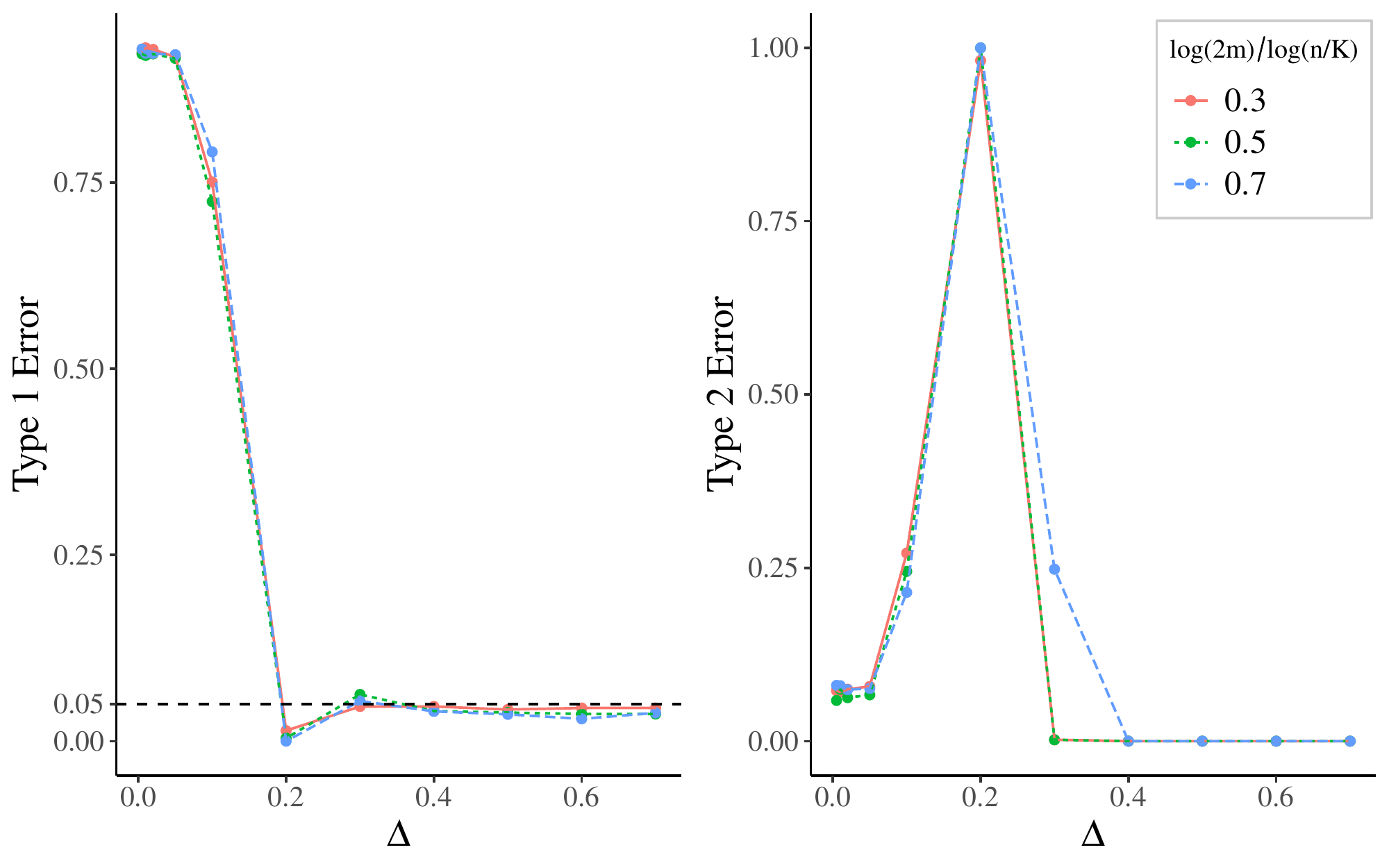}
	\end{tabular}
	\caption{Type-I and type-II errors for Example~\ref{ex:same} as $\Delta$ varies in $[0,0.8]$. Panel (a) shows the results for different graph sizes $n =$ 200, 600, 1000 with $m$ set as $\lceil (n/K)^{0.5}/2 \rceil$. Panel (b) shows the results $m = \lceil (n/K)^{\delta}/2 \rceil$ with $\delta =$ 0.3, 0.5, 0.7 and the   graph size $n=$ 600.}\label{fig:re-case1}
\end{figure}

 For both Examples \ref{ex:same} and \ref{ex:num}, the histograms of the shadowing bootstrap statistic in \eqref{eq:Wn} by choosing different  shadowing assignment $z_0 \in \cC_0$ are illustrated in  Figure~\ref{fig:comp-case1}. We can see that  the quantiles are almost the same for different $z_0$'s. This validates the rationale of our shadowing bootstrap method: we can estimate the quantile of $L_0$ in \eqref{eq:T0} by replacing the unknown truth $z^*$ with some shadowing assignment $z_0 \in \cC_0$. In Figure  \ref{fig:re-case1}, we show how the type-I and type-II errors vary with the signal strength $\Delta$, the graph size $n$ and the number of tested nodes $m$ for Example~\ref{ex:same}. Both the type-I and type-II errors are estimated via 500 repetitions . As $\Delta$ increases, the type-I error converges to the significance level $0.05$, which  shows that our method is honest. Type-II error is small when $\Delta$ is around zero as the test will always reject the null when the signal strength is too small, while it increases drastically as type-I error drops to 0.  When $\Delta$ is large enough, the type-II error converges to 0, showing that our test is powerful. 
 In Figure~\ref{fig:re-case1}(b), we can see that the type-I and type-II errors  for different $m$'s converge similarly  as $\Delta$ increases. This is consistent with Corollary \ref{case1-mtd-col} as the scaling condition is irrelevant to $m$. Simulation results for Example~\ref{ex:num} are shown in Figure~\ref{fig:re-case3}. The type 1 and type 2 error rates vary similarly as in Example~\ref{ex:same}.

\begin{figure}[tb]
	\centering
	\includegraphics[width=0.5\textwidth]{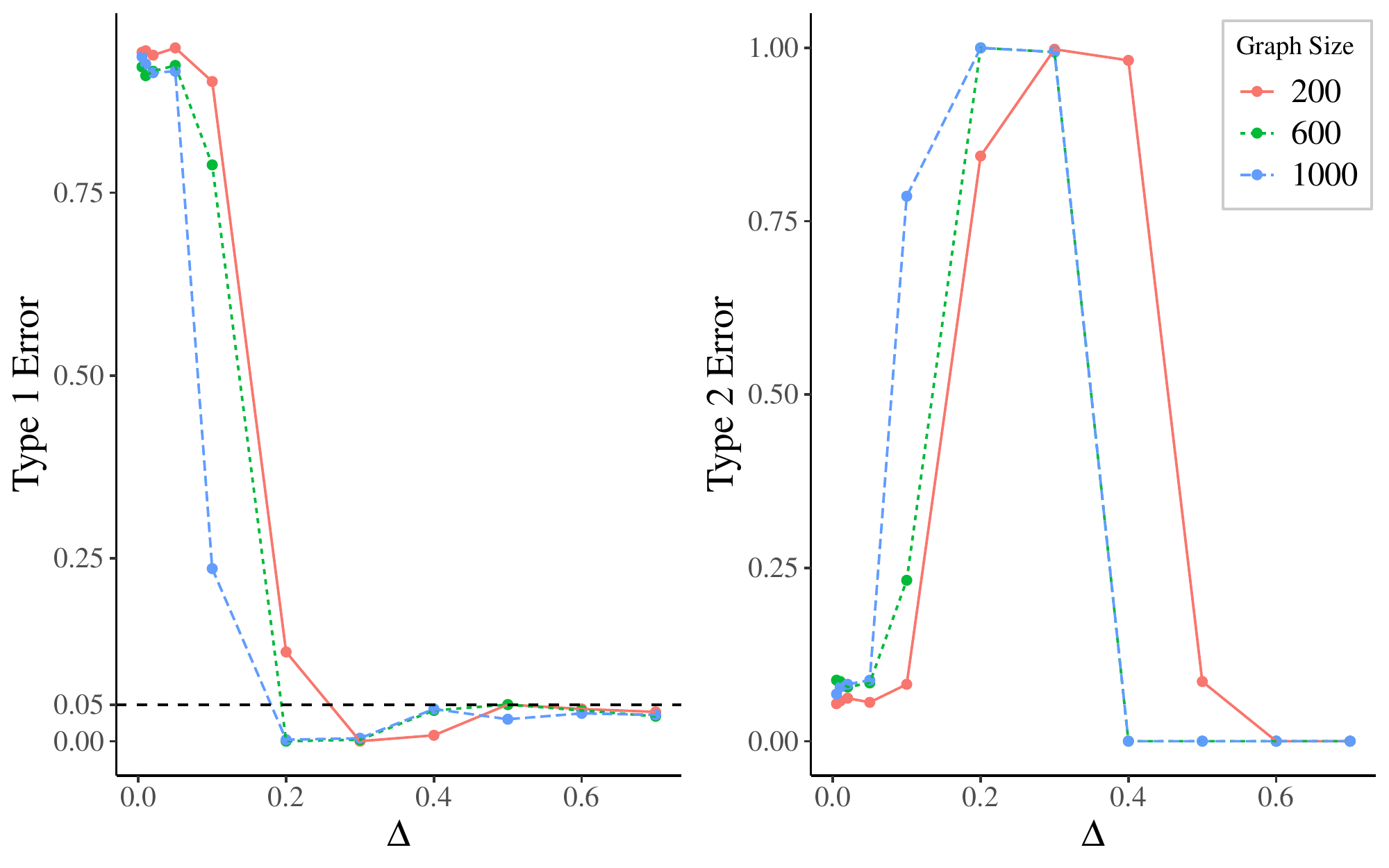}
	\caption{Type-I and type-II errors for Example~\ref{ex:num} as $\Delta$ varies in $[0,0.8]$ and the graph size $n = 200, 600, 1000$.}\label{fig:re-case3}
\end{figure}

 \begin{figure}[htbp]
	\centering
	\begin{tabular}{cccc}
	&$n=200$&$n=600$&$n=1000$\\
\rotatebox{90}{~~~~~~~~~~~~Example \ref{ex:same}}&		\includegraphics[width=0.3\textwidth]{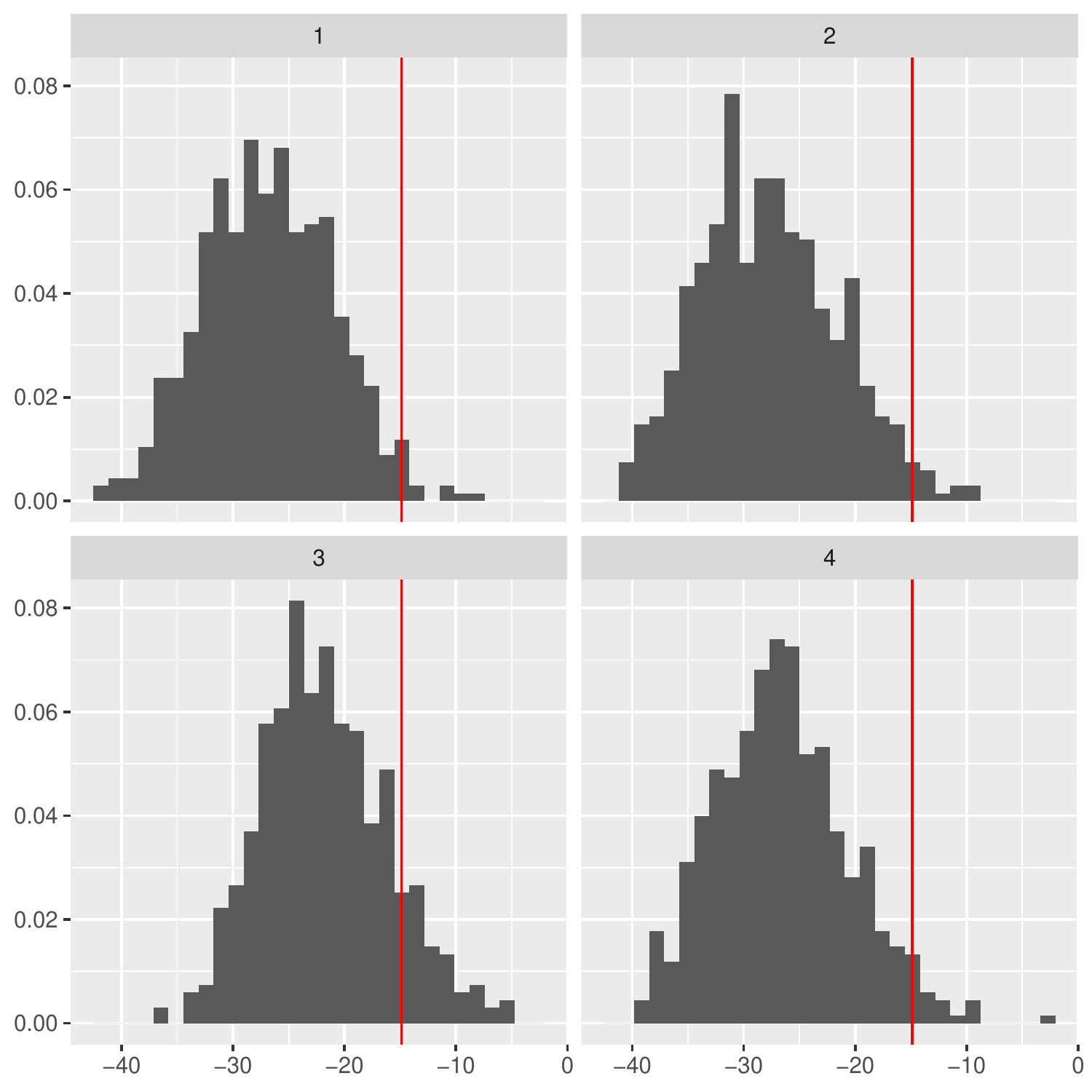}
		&\includegraphics[width=0.3\textwidth]{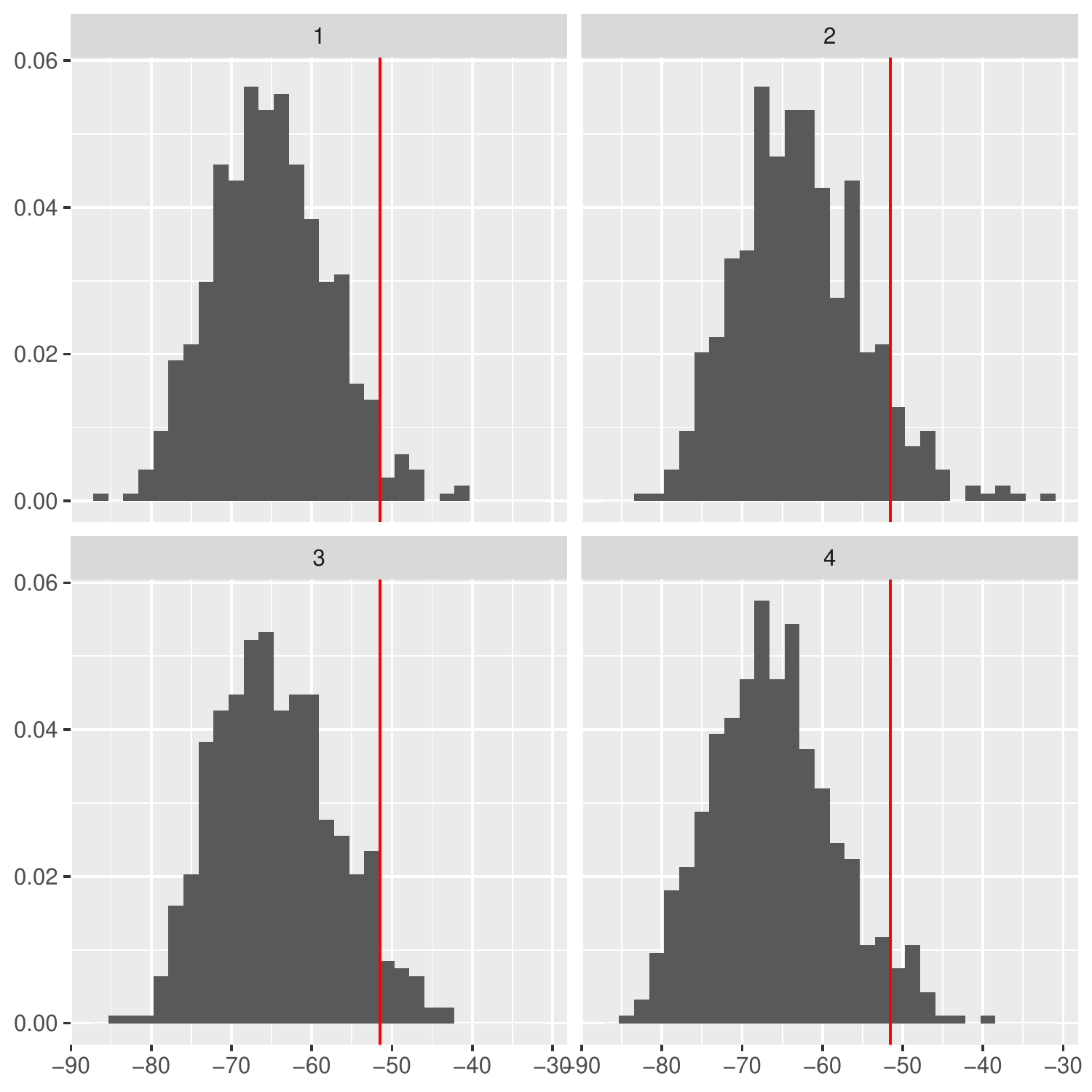}
		&\includegraphics[width=0.3\textwidth]{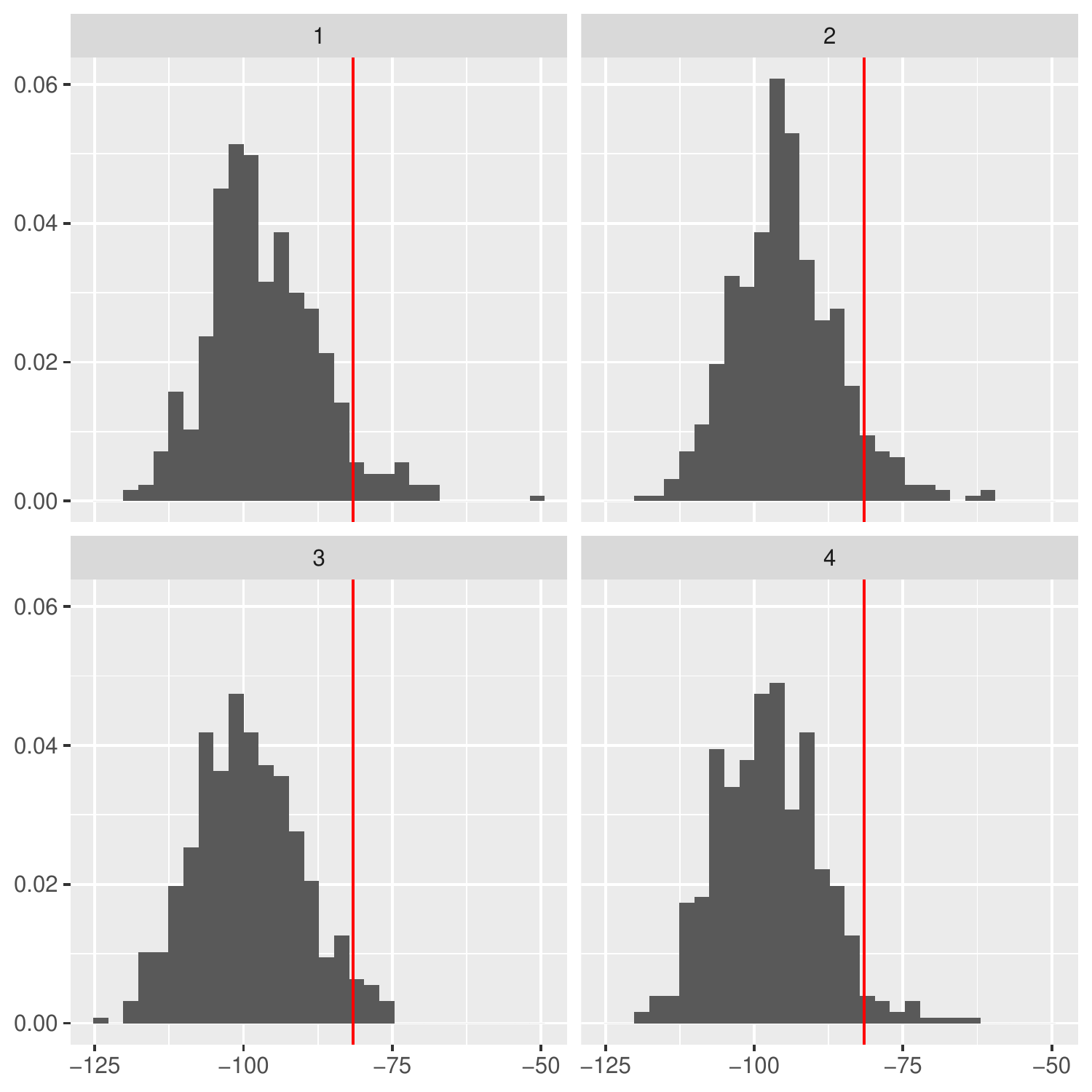}\\
\rotatebox{90}{~~~~~~~~~~~~Example \ref{ex:num}}&		\includegraphics[width=0.3\textwidth]{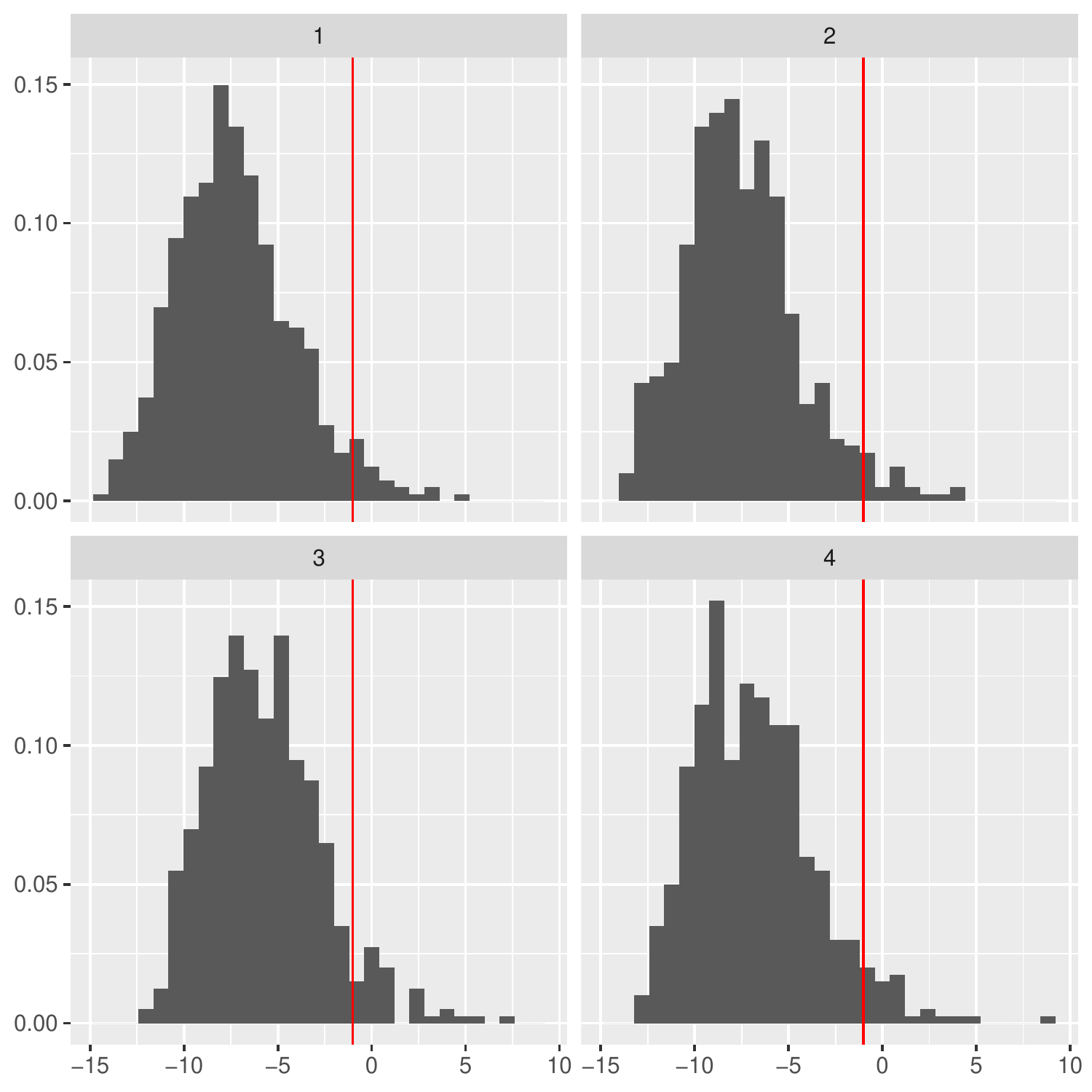}
		&\includegraphics[width=0.3\textwidth]{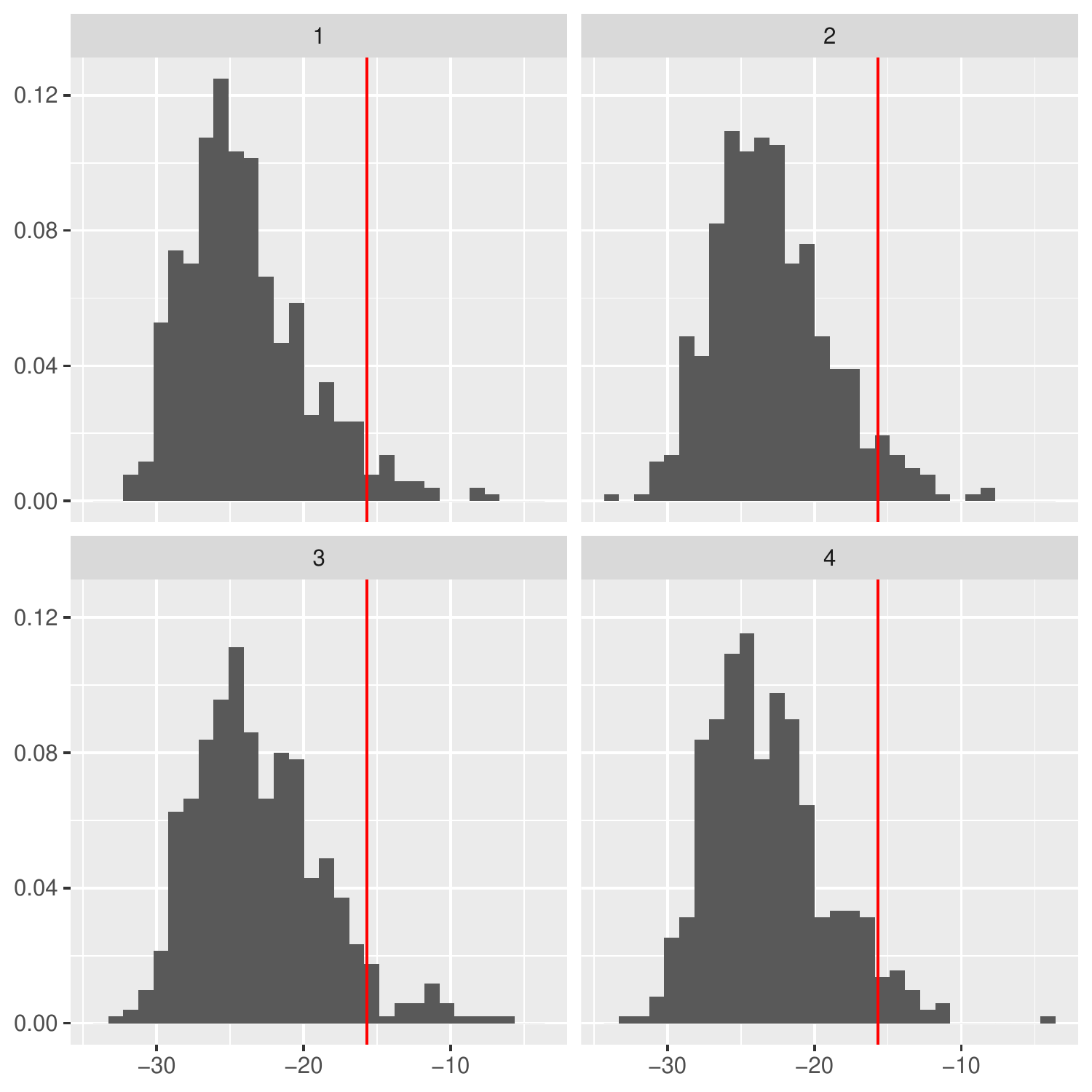}
		&\includegraphics[width=0.3\textwidth]{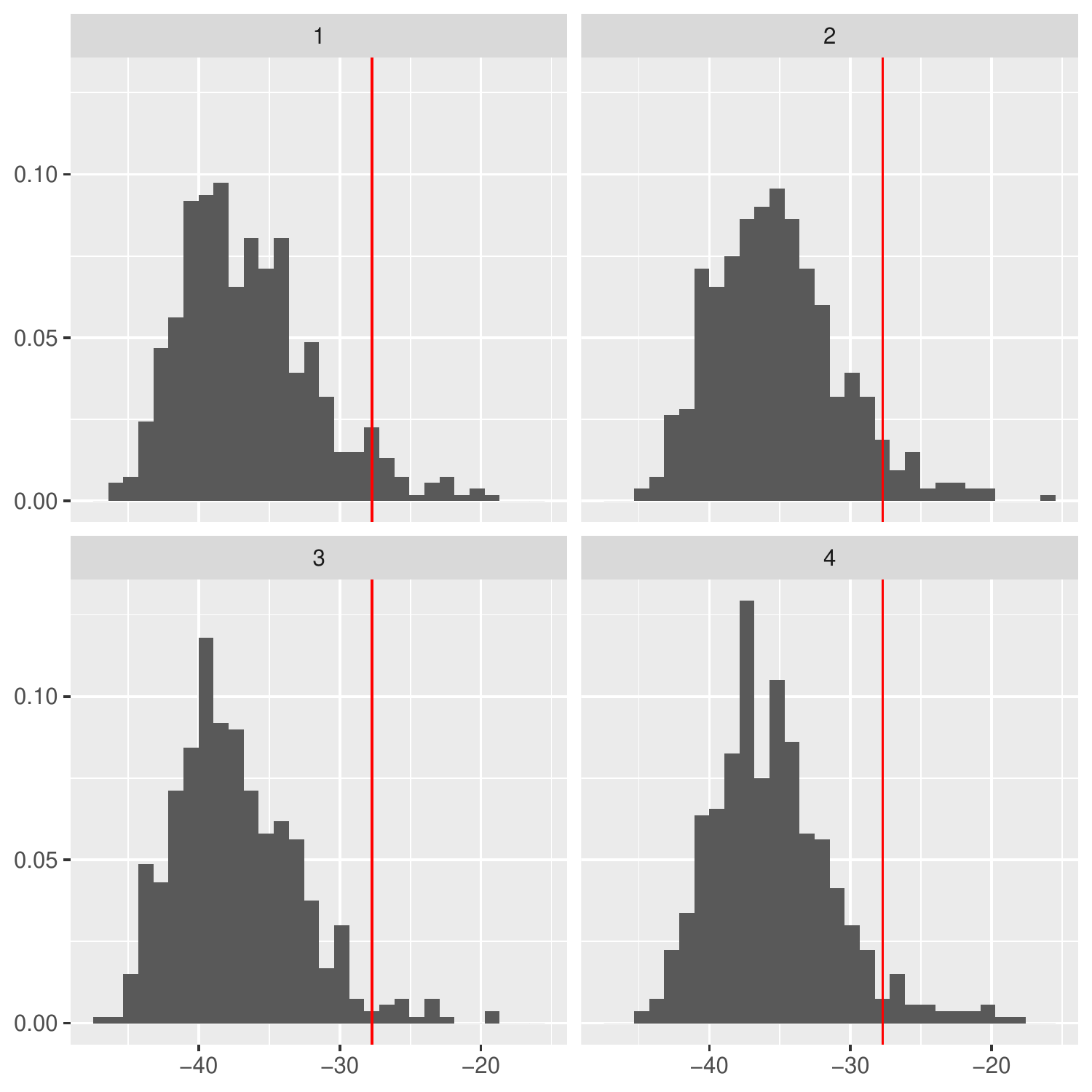}\\
	\end{tabular}
	\caption{The histograms of the transformed shadowing bootstrap statistic $g(\hat{p},\hat{q})W_n + g(\hat{p},\hat{q})\hat{\mu}_0$ for Examples~\ref{ex:same} and \ref{ex:num} with 500 bootstrap samples. We vary $n = 200, 600, 1000$ and for each $n$ we choose four different $z_0 \in \cC_0$. For Example \ref{ex:same}, we choose $\Delta=0.5$ and $m=\lceil (n/K)^{0.5}/2 \rceil$. For Example \ref{ex:num}, we choose $\Delta = 0.5$. Each panel shows the distribution of the statistic $W_n$ under different true $z_0$. According to the histogram, the distribution remains consistent for different true assignments in $\cC_0$.}\label{fig:comp-case1}
\end{figure}

We also compare our method with SIMPLE proposed in \cite{fan2019simple}. SIMPLE was designed to conduct a two sample test on whether two nodes belong to the same community. This is the case of Example \ref{ex:same}  in our paper with $m=2$. We also compare SIMPLE with our method  under Example \ref{ex:same} with $m >2$ by combining SIMPLE with Bonferroni correction.
We consider $n = 600$, $K = 2$ and  $m = 2, 3, 28$, where $m=28$ is chosen by setting $\delta=0.7$ in the formula $\lceil (n/K)^{\delta}/2 \rceil$. The probabilities $p$ and $q$ are chosen in the same way as the previous examples. The results of the comparison are shown in Table~\ref{table: comp with simple when m=2}. We applied the Bonferroni Correction to the SIMPLE method when doing the multiple comparison. From the results, our method outperforms SIMPLE in both the single-pair testing and the multiplicity testing. For the multiplicity testing results, Bonferroni Correction would result in a more conservative type-I error, yet the reported size from SIMPLE is still larger than the desired size of 0.05, indicating a lack of accuracy of the method under our setting of parameters. In comparison, the performance of our method is good and stable under all settings.
\begin{table}[htbp]
    \centering
    \begin{tabular}{c|c|ccccccc}
    \hline
     \hline
    $m=2$ & $\Delta$ & 0.1 & 0.2 & 0.3 & 0.4 & 0.5 & 0.6 & 0.7\\
    \hline
    \multirow{2}{*}{Our Method}& Size &0.770 & 0.008 & 0.056 & 0.040 & 0.060 & 0.048 & 0.056 \\
    & Power & 0.768 & 0.018 & 0.998 & 1.000 & 1.000 & 1.000 & 1.000\\
\hline
\multirow{2}{*}{SIMPLE}& Size &0.270 & 0.116 & 0.102 & 0.086 & 0.082 & 0.076 & 0.064 \\
    & Power & 0.312 & 0.756 & 0.990 & 1.000 & 1.000 & 1.000 & 1.000 \\
\hline
  \hline
    $m=3$ & $\Delta$ & 0.1 & 0.2 & 0.3 & 0.4 & 0.5 & 0.6 & 0.7\\
    \hline
    \multirow{2}{*}{Our Method}& Size &0.752 &0.014&0.048 &0.048 &0.044 &0.046 &0.046 \\
    & Power & 0.728 &0.018&0.998& 1.000& 1.000& 1.000& 1.000\\
\hline
\multirow{2}{*}{SIMPLE}& Size &0.344&0.118&0.070&0.074&0.070&0.084&0.070 \\
    & Power & 0.426&0.816&0.992 & 1.000 & 1.000 & 1.000 & 1.000 \\
\hline\hline
$m=28$ & $\Delta$ & 0.1 & 0.2 & 0.3 & 0.4 & 0.5 & 0.6 & 0.7\\
    \hline 
    \multirow{2}{*}{Our Method}& Size &0.792 &0.000&0.054&0.040&0.036&0.030&0.038 \\
    & Power &0.784 &0.000&0.752& 1.000& 1.000& 1.000& 1.000\\
\hline
\multirow{2}{*}{SIMPLE}& Size &0.938&0.232&0.126&0.088&0.072&0.06&0.044 \\
    & Power & 0.950&0.998& 1.000& 1.000& 1.000& 1.000& 1.000 \\
\hline
 \hline
    \end{tabular}
    \caption{Comparison of our method with SIMPLE for Example \ref{ex:same} with $m=2, 3, 28$. }
    \label{table: comp with simple when m=2}
\end{table}

\section{Real Application to Protein-Protein Network}
We apply our method to study the Protein-Protein Interaction (PPI) networks. We aim to test whether different protein functional groups belong to the same community in PPI \citep{tabouy2019variational}. The protein network is extracted from the STRING platform\footnote{\url{https://string-db.org}.} \citep{szklarczyk2015string}. For proteins $i$ and $j$, STRING dataset assigns an interaction score $v_{ij} \in [0,1]$ between the two proteins. We construct the adjacency matrix $\Ab$ of PPI as follows:
$$\Ab_{ij}=\left\{\begin{array}{ll}1 & \text { if } v_{i j}>0.95, \\  0 & \text { if } v_{i j} \leq 0.95. \end{array}\right.$$
We select the proteins from the two major orthologous groups: prokaryotic clusters (COGs) and eukaryotic clusters (KOGs) \citep{koonin2004comprehensive,galperin2019microbial} and remove those proteins  in both COGs and KOGs. We select 222 proteins in total. We further divide COGs and KOGs into the following 5 subgroups according to the Clusters of Orthologous Groups of proteins  database\footnote{\url{http://www.ncbi.nlm.nih.gov/COG/}}:
\begin{itemize}
    \item $G_1$: KOG-Information storage and processing;
    \item $G_2$: KOG-Cellular processes and signaling;
    \item $G_3$: COG-Cellular processes and signaling;
    \item $G_4$: COG-Information storage and processing;
    \item $G_5$: COG-Metabolism.
\end{itemize}

\begin{figure}[htb]
	\centering
	\includegraphics[width=0.4\textwidth]{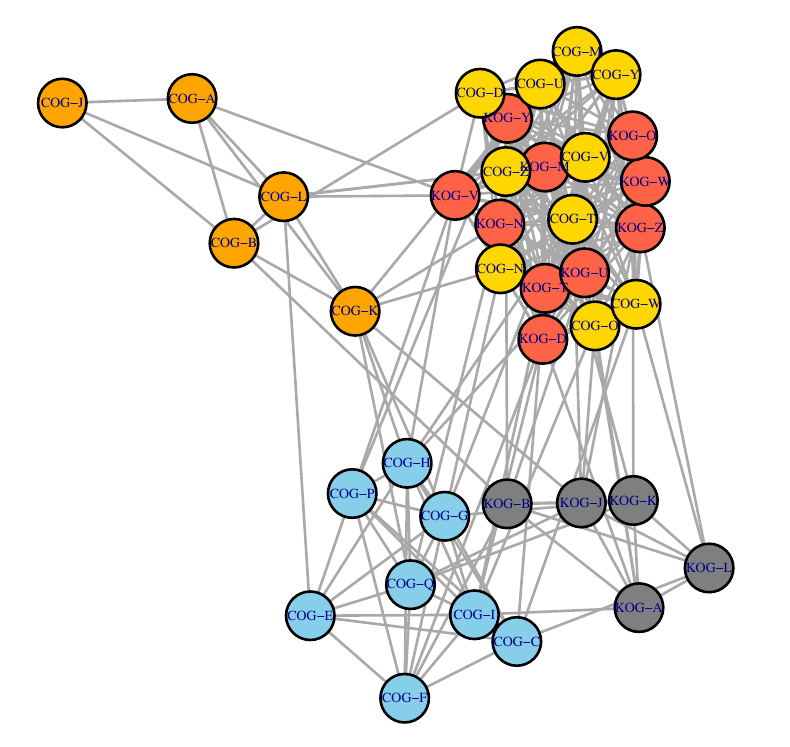}
	\caption{The functional network of protein interaction network. Each node represent a functional subgroup in KOGs or COGs (e.g. ``A"-RNA processing and modification\protect\footnotemark). The colors denote different functional categories. Gray nodes belongs to  $G_1$: KOG-Information storage and processing. Red nodes belong to $G_2$: KOG-Cellular processes and signaling. Yellow nodes belong to $G_3$: COG-Cellular processes and signaling. Orange nodes belong to $G_4$: COG-Information storage and processing. Blue nodes belong to $G_5$: COG-Metabolism.}\label{fig:ppi}.
\end{figure}
\footnotetext{\url{https://ftp.ncbi.nih.gov/pub/COG/COG2014/static/lists/homeCOGs.html}}

The protein-protein interacton network is visualized in Figure \ref{fig:ppi}. We aim to test whether any two subgroups above belong to the same community in PPI. We can formulate the hypthoses as follows:
\begin{align*}
 &\mathrm{H}_{0,ij}: \text{Subgroup $G_i$ and $G_j$ belong to the same community},\\
 &\mathrm{H}_{1,ij}: \text{Subgroup $G_i$ belong to community $a$, but $G_j$  belong to community $b \neq a$}.
\end{align*}
for $1 \le i \neq j \le 5$. These hypotheses are the same community test for groups in Example \ref{ex:group}.
We apply the shadowing bootstrap method to test $\mathrm{H}_{0,ij}$ for all $1 \le i \neq j \le 5$. We summarize the $p$-values of these hypothesis tests  in Table~\ref{table: subgr5}. 
\begin{table}[H]
    \centering
    \begin{tabular}{c|ccccc}
    \hline
    \hline
    Subgroups & $G_1$ &$G_2$& $G_3$&$G_4$&$G_5$\\
    \hline
   $G_1$ & - & $<10^{-6}$ & $<10^{-6}$ & 0.027806 & 0.013166 \\
  $G_2$& $<10^{-6}$ & - & 0.165860 & 0.000032 & 0.000178 \\
   $G_3$& $<10^{-6}$ & 0.165860 & - & 0.003800 & 0.029974 \\
  $G_4$ & 0.027806 & 0.000032 & 0.003800 & - & 0.023648 \\
  $G_5$& 0.013166 & 0.000178 & 0.029974 & 0.023648 & - \\
    \hline
    \hline
    \end{tabular}
    \caption{P-values of the test $\mathrm{H}_{0,ij}$ on whether two subgroups $G_i$ and $G_j$ are in the same community for $1 \le i \neq j \le 5$.}
    \label{table: subgr5}
\end{table}
According to the results, the p-value for $G_2$ and $G_3$ does not reach the significant level and we fail to reject the null hypothesis that they are in the same community. This is consistent with the fact that  both $G_2$ and $G_3$ belong to the cellular processing and signaling functional category. On the other hand, both $G_1$ and $G_4$ belong to the information storage and processing functional category, but the $p$-value between them is significant. This is consistent with the KOGs and COGs are heterogeneous in cellular processing \citep{koonin2004comprehensive}. We also notice that $p$-values between $G_1$ and $G_2$, $G_1$ and $G_3$, $G_2$ and $G_4$ are highly significant, as the information storage and processing and the cellular processes and signaling are different functional processes \citep{rehman2017three,brun2003functional,pal2005inference}.

{\setstretch{1.0}
	\setlength{\bibsep}{0.85pt}
	{\small
		\bibliographystyle{ims}
		\bibliography{Reference.bib}
	}
}

\newpage

\setcounter{page}{1}

\begin{center}
\textit{\large Supplementary material to}
\end{center}
\begin{center}
\title{\Large Combinatorial-Probabilistic Trade-Off:  CommunityProperties Test in the Stochastic Block Models}
\vskip10pt
\author{}
\end{center}

\setcounter{section}{0}
\renewcommand{\thesection}{\Alph{section}}

\maketitle

  This document contains the supplementary material to the paper
  ``Combinatorial-Probabilistic Trade-Off:  Community Properties Test in the Stochastic Block Models".  In Appendix \ref{sec: proof com prop}, we provide proofs of propositions related to community properties. Appendix \ref{sec: proof inf} proves that our testing method is honest and powerful.  In Appendix \ref{sec: proof lwrbd}, we prove the theorems related to lower bound and apply the general lower bound theorem to derive a sharp threshold for exact recovery. In Appendix
  \ref{sec: proof of tec lem}, we provide proofs for the technical lemmas that were used in proving the validity and power of the upper bound test.

\section{Proofs of Community Properties}\label{sec: proof com prop}
In this section we mainly focus on the proofs concerning community properties, including the generalization of symmetric community property pairs from even to uneven cluster sizes, the size of the ball $B_{z_0}(r_K)$ in three examples, and the packing number of the ball in each case. 
\subsection{Proof of Proposition~\ref{prop: lb-case study-pac num}}\label{sec: proof even pack num}
 \textbf{Example~\ref{ex:same}: } In this case, for a given $z_0 \in \cC_0$, we have derived the form of $B_{z_0}$. For any $z_i, z_j \in \cP\big(B_{z_0}, \sqrt{d(z_0,\cC_1)}\big)$, we know from Section~\ref{mtd} that they are transformed from $z_0$ by swapping one of the first $m$ nodes with another node from a different cluster. The node among the first $m$ to be swapped $s \in [m]$ cannot be the same for the two assignments, otherwise $|\mathcal{E}_{1,2}(z_0,z_i) \cap \mathcal{E}_{1,2}(z_0,z_j)| \geq |\mathcal{E}_{1}(z_0,z_i) \cap \mathcal{E}_{1}(z_0,z_j)| = n/K-1 \gg \sqrt{d(z_0,\cC_1)}$. Thus each $z \in \cP\big(B_{z_0}, \sqrt{d(z_0,\cC_1)}\big)$ corresponds to a different swapped node among the first $m$ nodes, and we have $N\big(B_{z_0}, \sqrt{d(z_0,\cC_1)}\big) \leq m$. On the other hand, for the given assignment $z_0$, we can construct the following set $\{z_k\}_{k=1}^m$: we take a set of nodes $\mathcal{S} = \{s_1,s_2,...,s_m\}$ from a cluster different from the cluster to which the first $m$ nodes of $z_0$ belong. Then for each $k$, we swap the cluster assignment of node $k$ with node $s_k$, $k=1,...,m$, and obtain the corresponding alternative assignment $z_k$. Then for any two alternative assignments $z_i$ and $z_j$ obtained this way, we have $|\mathcal{E}_{1,2}(z_0,z_i) \cap \mathcal{E}_{1,2}(z_0,z_j)| \leq 4$. Thus $N\big(B_{z_0}, \sqrt{d(z_0,\cC_1)}\big) = m$. 
 \\
\noindent \textbf{Example~\ref{ex:group}: } For a given $z_0 \in \cC_0$ and the corresponding boundary $B_{z_0}$, it can be perceived that $N(B_{z_0},\sqrt{d(z_0,\cC_1)}) = N(B_{z_0},0)=1$, because any $z \in B_{z_0}$ involves swapping the set $\mathcal{S}_2$ so that $\forall z_i, z_j \in B_{z_0}, |\mathcal{E}_{1,2}(z_0,z_i) \cap \mathcal{E}_{1,2}(z_0,z_j)| \geq m \wedge m'(n/K-m \wedge m')$. 

\subsection{Proof of Proposition \ref{prop:sym}}\label{sec: proof of prop 6.1}
To prove that Definition \ref{def:cluster class} is a special case of Definition \ref{def: sym} when the community size is even, it suffices for us to construct a concrete community label permutation $\sigma$ and  node label permutation $\tau$ satisfying Definition \ref{def: sym} based on $\cN$ and $\tilde{z}$. Here we use Figure~\ref{fig:exp} to illustrate the construction. Given any $z, z' \in \cC_0$, we first construct $\sigma$. Since $z_{\cN} \simeq z'_{\cN} \simeq \tilde{z}_{\cN}$, by Definition \ref{def:cluster class}, there must exist a $\sigma \in S_K$ mapping $z$ to $z'$ on the support $\cN$, i.e., $\sigma (z_{\cN}) = z'_{\cN}$.  For example, in Figure~\ref{fig:exp}, we construct a $\sigma$ swapping communities 1 and 2. After matching the community labels, we now construct $\tau$ in order to transform $\sigma(z)$ to $z'$. Since the community size is even and $\sigma (z_{\cN}) = z'_{\cN}$, $\sigma(z)$ and $z'$ have equal cluster sizes on the support of $\cN^c$. Therefore, there exists $\tau \in S_n$ such that $\tau(\sigma(z)_{\cN^c}) = z'_{\cN^c}$ and $\tau(\sigma(z)_{\cN}) = \sigma(z)_{\cN}=z'_{\cN}$. We can see the example of $\tau$ in Figure \ref{fig:exp}. Using $\sigma$ and $\tau$ constructed above, we can check that $\tau \circ \sigma(z) = z'$. We now check the last condition in Definition \ref{def: sym}. For any $z'' \in \cC_1$, since $\tau$ is invariant on $\cN$, we have $\tau \circ \sigma \left( z''_{\cN} \right) = \sigma \left( z''_{\cN} \right) \simeq z''_{\cN}$. By Definition \ref{def:cluster class}, the alternative community  $\cC_1$ is closed under permutation on the support of $\cN$, we have $\tau \circ \sigma (z'') \in \cC_1$. Therefore, we check that Definition \ref{def:cluster class} is a special case of Definition \ref{def: sym}.

Since the property pairs in \eqref{eq:c0ex1} and \eqref{eq:c0ex2} are symmetric property pairs, they are also generalized symmetric property pairs following the preceding arguments. As for the property pair in \eqref{eq:c0ex3}, we can see from Figure~\ref{fig:exp}(b) that for any two assignments $z, z' \in \cC_0$, since they have equal community sizes, we can take $\sigma$ to be the identity map and there exists $\tau \in S_n$ such that $\tau(z)=z'$. Then for any $z'' \in \cC_1$, since $\tau$ does not change the community sizes, we know that $\tau(z'')$ still have uneven community sizes and $\tau(z'') \in \cC_1$. Therefore, by Definition \ref{def: sym}, the property pair in \eqref{eq:c0ex3} is a generalized symmetric property pair. 
\subsection{Proof of Proposition~\ref{prop:bz-cases}}\label{sec: proof of bz-g}

\textbf{Example \ref{ex:same}: }To construct $B_{z_0}(r_K)$, we need to find all the assignments in $\cC_1$ whose distance from $z_0$ is no larger than $d(\cC_0, \cC_1)$ by an extra constant term. To construct assignments in $\cC_1$ closest to $z_0$, we would pick one node in $[m]$ and reassign it to a different community (see Figure~\ref{fig:exp uneven move} (a)). Assignments constructed in such ways will satisfy $d(z_0,z_1) = d(\cC_0, \cC_1) = n/K $. If we make community changes to any other nodes on the basis of such construction, then $d(z_0,z_1)$ would increase by at least $n/K - 2$, which exceeds the constant level. Thus $B_{z_0}(r_K)$ consists of all assignments constructed by moving one node of $z_0$ in $[m]$ to a different cluster. Since we can pick $m$ nodes in total and reassign them to $K-1$ different clusters, $|B_{z_0}(r_K)| = (K-1)m = O(m)$.
\\
\noindent \textbf{Example \ref{ex:group}: }For an arbitrary $z_0 \in \cC_0$, without loss of generality, we assume that $m'\le m$. Then when $m' \le c_K$, to construct assignments in $\cC_1$ that are closest to $z_0$, we need to reassign nodes $m+1, \ldots, m+m'$ collectively to a different community (see Figure~\ref{fig:exp uneven move} (b). Such constructed assignments have distance $d(z_0,z_1) = d(\cC_0, \cC_1) = m' n/K$. Similar to the previous example, any community changes to other nodes on the basis of such construction would result in increase of $d(z_0, z_1)$ by at least $n/K - m' -1$. Therefore, $B_{z_0}(r_K)$ consists of those assignments in $\cC_1$ constructed by reassigning nodes $m+1, \ldots, m+m'$. Since there are $K-1$ other clusters to reassign in total, we have $|B_{z_0}(r_K)| = K-1 = O(1)$. On the other hand, when $m' > c_K$, then we cannot reassign nodes $m+1,\ldots,m+m'$ collectively without exchanging with other nodes, otherwise the community size bound will be violated. Then $d(\cC_0,\cC_1)$ and $B_{z_0}(r_K)$ is exactly the same as the even case and the claim follows.  
 \\
\noindent \textbf{Example \ref{ex:num}: }As for the ball $B_{z_0}(r_K)$ for an arbitrary $z_0 \in \cC_0$, to transform $z_0$ into an assignment $z_1 \in \mathcal{C}_1$, the simplest way is to reassign an arbitrary node to a different community, and $d(z_0,z_1) = d(\mathcal{C}_0,\mathcal{C}_1) = n/K \asymp n$. Further community changes will result in increasing in $d(z_0,z_1)$ that exceeds the constant level. Since we can obtain such $z_1$ by reassigning any one of the $n$ nodes into the other $K-1$ clusters, we have $|B_{z_0}({r}_K)| = n(K-1) = O(n)$.

\subsection{Proof of Proposition \ref{prop: pk-num-g}}\label{sec: proof pack num-g}
The arguments for Example~\ref{ex:same} and Example~\ref{ex:group}
are almost the same as in the even cases and are hence omitted.

\textbf{Example~\ref{ex:num}: } For a given $z_0 \in \cC_0$, from previous discussion we can see that the ball $B_{z_0} (r)$ with $r=d(z_0,\cC_1) + O(1)$ is composed of all the assignments that differ from $z_0$ by one mis-aligned node. For any $z_i, z_j \in \cP \big ( B_{z_0}(r),\sqrt{d(z_0,\cC_1)} \big)$, the misaligned node $s$ cannot be the same, otherwise $|\mathcal{E}_{1,2}(z_0,z_i) \cap \mathcal{E}_{1,2}(z_0,z_j)| \geq n/K \gg \sqrt{d(z_0,\cC_1)}  $. Thus we have $ N\big ( B_{z_0}(r),\sqrt{d(z_0,\cC_1)} \big) \leq n$. Also since the set $\{z_k\}_{k=1}^n$ where each $z_k$ is obtained by reassigning the node $k$ into another cluster obviously satisfies the condition that $|\mathcal{E}_1(z_0,z_i) \cap \mathcal{E}_1(z_0,z_j)| + |\mathcal{E}_2(z_0,z_i) \cap \mathcal{E}_2(z_0,z_j)| \leq 1$, we have that $N\big ( B_{z_0}(r),\sqrt{d(z_0,\cC_1)} \big)=n$.

\section{Proof of Inference Results}\label{sec: proof inf}

In this section, we provide the proofs of the theorems on inference results. We will first prove Proposition \ref{prop: eq-quant} which implies that the quantile of the maximal leading term $L_0$ can be estimated without knowing the true assignment, then we prove the main Theorem~\ref{t1-mtd} using Proposition~\ref{prop: eq-quant} along with other lemmas. The proof of the technical lemmas will be deferred to Section~\ref{sec: proof of tec lem}.

In the following part of our paper, we use $c, C, c_1, c_2, C_1, C_2, \ldots$ to represent generic constants and their values may vary in different places.

\subsection{Proof of Proposition \ref{prop: eq-quant}}\label{sec: proof of invary quant on z0}
To prove Proposition~\ref{prop: eq-quant}, we need the following generalized version of Lemma~\ref{lm2-t1} stated previously 
\begin{lemma}[Shadowing symmetry lemma]\label{lm2-t1-g}
	  For a given $z \in \mathcal{C}_0$ and a given radius $r>0$, we list the assignments in the ball $B_{z}(r)$ as $z_1, z_2, \ldots, z_{|B_{z}(r)|}$. Define a $|B_{z}(r)|$-dimensional vector $\bL_z$ as
	 \[
(\bL_{z})_k = g(p,q)\bigg( \sum_{(i,j) \in \mathcal{E}_2(z,z_k)} \Ab_{ij}-\sum_{(i,j) \in \mathcal{E}_1(z,z_k)} \Ab_{ij} \bigg), \text{ for } k=1,2,\ldots, |B_{z}(r)|.
	 \]
	 	Suppose $\cC_0$ and $ \cC_1$ satisfy definition \ref{def: sym}, then for any $z_0, z_0' \in \cC_0$, we have $|B_{z_0}(r)| = |B_{z_0'}(r)|$ and  $\Cov(\bL_{z_0})$ equals to $\Cov(\bL_{z_0'})$ up to permutation, i.e., there existing a permutation $\uptau \in S_{|B_{z_0}(r)|}$ such that $\Cov(\bL_{z_0})_{kl} = \Cov(\bL_{z_0'})_{\uptau(k)\uptau(l)}$ for all $k,l = 1, \ldots, |B_{z_0}(r)|$.
	 \end{lemma}
We defer the proof of Lemma \ref{lm2-t1-g} to Section \ref{sec: proof lm2-t1-g}. Now we are ready to prove Proposition~\ref{prop: eq-quant}. In fact, the boundary in the definition of $L_0$ can be generalized to the ball $B_{z}(r)$ with $r \geq r_K:= d(\cC_0,\cC_1)+{c_K^2{{p} K}}/(2({p}-{q}))$ and $r=d(\cC_0,\cC_1) + O(1)$.
For the true assignment $z^* \in \cC_0$, we have that 
\begin{align*}                         
	L_0&=\sup_{z_k\in B_{z^*}(r)} \bigg \{g({p},{q})\big  (\sum_{(i,j)\in \mathcal{E}_2(z^*,z_k)} \Ab_{ij} - \sum_{(i,j)\in \mathcal{E}_1(z^*,z_k)} \Ab_{ij}  \big )  \bigg \}\\
	&=g({p},{q})\sup_{z_k \in B_{z^*}(r)} \bigg \{\sum_{i<j} \Big  \{ \big (\Ab_{ij}-\mathbb{E}(\Ab_{ij}) \big )\big ( \mathbbm{1}[(i,j) \in \mathcal{E}_2(z^*,z_k)]-\mathbbm{1}[(i,j) \in \mathcal{E}_1(z^*,z_k)]\big ) \Big \} \bigg \}\\&\quad +g({p},{q})\mu_0 +\delta_n\\
	&=g({p},{q})\sigma_0\sup_{k \in [|B_{z^*}(r)|]} \Big \{\frac{1}{\sigma_0}  \sum_{i<j} (\boldsymbol{X}_{ij})_k \Big \}+g({p},{q})\mu_0+\delta_n.
	\end{align*}
	where the vector $\boldsymbol{X}_{ij} \in \mathbb{R}^{|B_{z^*}(r)|}$ and $(\boldsymbol{X}_{ij})_k = \big (\Ab_{ij}-\mathbb{E}(\Ab_{ij}) \big )\big ( \mathbbm{1}[(i,j) \in \mathcal{E}_2(z^*,z_k)]-\mathbbm{1}[(i,j) \in \mathcal{E}_1(z^*,z_k)]\big ) $, $\delta_n = O(\rho_n)$, and $\sigma_0 = \sqrt{d(\mathcal{C}_0,\mathcal{C}_1) \big( p(1-p)+ q(1-q)\big)}, \mu_0=d(\mathcal{C}_0,\mathcal{C}_1)(q-p)$. We can see that for different $(i,j)$, the vector $\boldsymbol{X}_{ij}$ are independent of each other. For a fixed $k \in [|B_{z^*}(r)|]$, when $(i,j) \notin \mathcal{E}_{1,2}(z^*,z_k)$, $(\boldsymbol{X}_{ij})_k = 0$. When $(i,j) \in \mathcal{E}_{1,2}(z^*,z_k) $, under the regime $1/\rho_n = o(n^{1-c_2})$ for some positive $c_2$, there exists $B_n = 1/\sqrt{\rho_n} = o(n^{(1-c_2)/2})$ such that $\big |(\boldsymbol{X}_{ij})_k/\sqrt{\rho_n} \big | \leq B_n$ and $B_n^2 (\log 2d(C_0,\mathcal{C}_1) |B_{z_0}(r)|)^7/n \leq n^{-c_2/2}$, where $d(\mathcal{C}_0,\mathcal{C}_1) = o(n^2)$. Therefore, following a very similar proof as Theorem 2.2 and Corollary 2.1 in \cite{chernozhukov2013gaussian}, we have $$g({p},{q})\sup_{k \in [|B_{z^*}(r)|]} \Big \{{\sum_{i<j} (\boldsymbol{X}_{ij})_k}/{\sigma_0}  \Big \} \overset{d}{\rightarrow} \sup_{k \in [|B_{z^*}(r)|]} \tilde{Z}_k,$$ where $\tilde{Z} \sim N(0, \mathbf{\Sigma}_{z^*}/\sigma_0^2)$, and $\Sigma_{z^*} = \Cov(\bL_{z^*})$. Therefore, we have that 
	\begin{align*}
	    &\sup_{t \in \mathbb{R}} \left|\mathbb{P}(L_0 \leq t) - \mathbb{P}(\sigma_0  \sup_{k \in [|B_{z^*}(r)|]} \tilde{Z}_k +g(p,q)\mu_0 \leq t)\right|\\
	    &\le \sup_{t \in \mathbb{R}} \left|\mathbb{P}(L_0 \leq t) - \mathbb{P}(\sigma_0  \sup_{k \in [|B_{z^*}(r)|]} \tilde{Z}_k +g(p,q)\mu_0 +\delta_n \leq t)\right|\\&\quad +\sup_{t \in \mathbb{R}} \left|\mathbb{P}(\sigma_0  \sup_{k \in [|B_{z^*}(r)|]} \tilde{Z}_k +g(p,q)\mu_0 +\delta_n \leq t) - \mathbb{P}(\sigma_0  \sup_{k \in [|B_{z^*}(r)|]} \tilde{Z}_k +g(p,q)\mu_0 \leq t)\right|\\
	    &\le o(1)+\sup_{t \in \mathbb{R}} \left|\mathbb{P}\Big(\big| \sup_{k \in [|B_{z^*}(r)|]} \tilde{Z}_k -\left( t-g(p,q)\mu_0\right)/\sigma_0 \big| \le \delta_n/\sigma_0 \Big) \right|.
	\end{align*}
	We know that $\min_{k \in [|B_{z^*}(r)|]} \operatorname{Var}(\tilde{Z}_k) = \Omega(g(p,q)^2)=\Omega(1)$, $\log |B_{z^*}(r)| = O(\log n)$ and $\delta_n/\sigma_0 = O(n^{-1/2})$. Then by Lemma 2.1 in \cite{chernozhukov2013gaussian}, we have 
	\begin{align*}&\sup_{t \in \mathbb{R}} \left|\mathbb{P}\Big(\big| \sup_{k \in [|B_{z^*}(r)|]} \tilde{Z}_k -\left( t-g(p,q)\mu_0\right)/\sigma_0 \big| \le \delta_n/\sigma_0 \Big) \right| \\&\lesssim \frac{\delta_n}{\sigma_0}\left \{ \sqrt{2\log |B_{z^*}(r)| } +\sqrt{\min_{k \in [|B_{z^*}(r)|]} \operatorname{Var}(\tilde{Z}_k) \sigma_0 / \delta_n} \right \} \le n^{-1/4}.\end{align*}
	And thus we have
	$$ \sup_{t \in \mathbb{R}} \left|\mathbb{P}(L_0 \leq t) - \mathbb{P}(\sigma_0  \sup_{k \in [|B_{z^*}(r)|]} \tilde{Z}_k +g(p,q)\mu_0 \leq t)\right| = o(1).$$
	Following the same procedure with $z^*$ replaced by $z_0$, we also have 
	$$ \sup_{t \in \mathbb{R}} \left|\mathbb{P}(L_0(z_0) \leq t) - \mathbb{P}\Big(\sigma_0  \sup_{k \in [|B_{z_0}(r)|]} ( \tilde{Z}')_k +g(p,q)\mu_0 \leq t\Big)\right| = o(1),$$
	where $\tilde{Z}' \sim N(0, \mathbf{\mathbf{\Sigma}}_{z_0}/\sigma_0^2)$, and $\mathbf{\Sigma}_{z_0} = \Cov(\bL_{z_0})$.
	By Lemma~\ref{lm2-t1-g} we know that $\mathbf{\Sigma}_{z^*}$ and $\mathbf{\Sigma}_{z_0}$ are equal up to permutation. Therefore, the claim follows. We may also notice that the validity of the proof does not depend on the values of $p,q$ as long as the regime is $1/\rho_n = o(n^{1-c_2})$ for some constant $c_2>0$, and thus the statement is also true for $\hat{L}_0:=\sup_{z_k\in B_{z^*}(r)} \bigg \{g(\hat{p},\hat{q})\big  (\sum_{(i,j)\in \mathcal{E}_2(z^*,z_k)} \Ab_{ij} - \sum_{(i,j)\in \mathcal{E}_1(z^*,z_k)} \Ab_{ij}  \big )  \bigg \}$ with plugged-in estimators $\hat{p},\hat{q}$.
\subsection{Proof of Theorem~\ref{t1-mtd}}\label{sec:lrt-proof}
In fact, Proposition~\ref{prop:sym} shows that the symmetric community property pairs defined in Section~\ref{sec: def and ntn} are general symmetric property pairs under the general framework, and Theorem~\ref{t1-mtd-g} is a generalization of Theorem~\ref{t1-mtd} under uneven cluster sizes. Thus we can just prove the more general Theorem~\ref{t1-mtd-g} and the proof will also apply to Theorem~\ref{t1-mtd}.

The proof of the main theorem requires the help of Proposition~\ref{prop: eq-quant} and the following lemma that shows why the maximizer in the alternative assignment space can be restricted to the ball centered at the true assignment $z_0 \in \cC_0$. 
\begin{lemma}\label{lm1-t1}
	We denote $z^*$ as the true assignment, and $B_{z^*}(r_K)$ is the ball centered at $z^*$ with radius $r_K = d(z^*,\cC_1) + \frac{p K}{2(p-q)}c_K^2$, $c_K = O(1)$. Under the same conditions of Theorem~\ref{t1-mtd-g}, when $z^* \in \cC_0$
	\begin{equation}\label{eq: sup-bz}
	    \sup_{z \in \cC_1} \log f(\Ab;z,\hat{p},\hat{q}) = \sup_{z \in B_{z^*}(r_K)} \log f(\Ab;z,\hat{p},\hat{q}) +O_{P}(\rho_n);
	\end{equation}
	 Moreover, for any true assignment $z^*$, we have
	\begin{equation}\label{eq: sup-c0c1}
	    \sup_{z \in \cC_0 \cup \cC_1} \log f(\Ab;z,\hat{p},\hat{q}) = \log f(\Ab;z^*,\hat{p},\hat{q}) +O_{P}(\rho_n).
	\end{equation}
\end{lemma}
With help of this lemma, instead of taking the supremum over the entire assignment space $\mathcal{C}_1$, we are able to restrict the maximizer to a much smaller set $B_{z_0}(r_K)$ so that the Central Limit Theorem can be applied. Recall that the boundary $B_{z^*}$ defined in Section~\ref{sec: LRT test intro} is in essence a ball with radius $d(z^*, \cC_1)$. We defer the proof of Lemma~\ref{lm1-t1} to Appendix \ref{sec: proof lm b2}.

Now we are ready to present the proof of Theorem~\ref{t1-mtd-g}:

	The proof is mainly composed of three parts. The first part is to briefly illustrate the derivation of $L_0$ as the leading term of the log-likelihood ratio, the second part is to control the error caused by plugging in the estimators of connection probabilities $\hat{p},\hat{q}$, and the third part is to illustrate the multiplier bootstrap as a valid approximation of the LRT quantile. 
	\subsubsection{Derivation of the leading term for LRT}
	For a given true assignment $z^* \in \cC_0$, by Lemma~\ref{lm1-t1} we have:
    \begin{align*}
	\hat{\LRT}&=\log \frac{\sup_{z \in \cC_1} f(\Ab;z,\hat{p},\hat{q})}{\sup_{z \in \cC_0 \cup \cC_1} f(\Ab;z,\hat{p}, \hat{q})}\\
	&=\sup_{z \in \cC_1} \log f(\Ab;z,\hat{p},\hat{q}) - \log f(\Ab;z^*,\hat{p},\hat{q}) + O_P(\rho_n)\\
	&=\sup_{z_k \in B_{z^*}(r)} \Big (\log f(\Ab;z_k,\hat{p},\hat{q}) - \log f(\Ab;z^*,\hat{p},\hat{q})\Big ) + O_P(\rho_n),
	\end{align*}
	where $r \geq r_K:= d(\cC_0,\cC_1)+{c_K^2{{p} K}}/(2({p}-{q}))$ and $r=d(\cC_0,\cC_1) + O(1)$. In practice, due to the consistency of $\hat{p},\hat{q}$, when we choose the radius $r=d(\cC_0,\cC_1)+C\hat{p}K/(\hat{p}-\hat{q})$ for some sufficiently large $C$, we can make sure that the conditions on the radius is satisfied with probability $1-o(1)$. Thus we can see that the $\hat{\LRT}$ is essentially the supremum of the log-likelihood difference between the true assignment $z^*$ and the alternative assignments in the ball $B_{z^*}(r)$. We further expand the log-likelihood terms and can write 
	\begin{align*}                         
	\hat{\LRT}&=\sup_{z_k \in B_{z^*}(r)} \bigg \{g(\hat{p},\hat{q})\big  (\!\!\!\!\sum_{(i,j)\in \mathcal{E}_2(z^*,z_k)} \!\!\!\!\Ab_{ij} -\!\!\!\! \sum_{(i,j)\in \mathcal{E}_1(z^*,z_k)}\!\!\!\! \Ab_{ij}  \big ) + \log \big ( \frac{1-\hat{q}}{1-\hat{p}} \big ) \Big( n_1(z^*,z_k) - n_2(z^*,z_k)\Big) \bigg \} \\
	&\quad +O_P(\rho_n)\\
	&=\hat{L}_0 +\delta_n.
	\end{align*}
	where $\delta_n = \sup_{z_k \in B_{z^*}(r)} \left\{ \log \big ( {(1-\hat{q})}/{(1-\hat{p})} \big ) \Big( n_1(z^*,z_k) - n_2(z^*,z_k)\Big) \right\}+O_P(\rho_n) = O_P(\rho_n)$, and $\hat{L}_0 = g(\hat{p},\hat{q}) \sup_{z_k \in B_{z^*}(r)} \big  (\sum_{(i,j)\in \mathcal{E}_2(z^*,z_k)} \Ab_{ij} -  \sum_{(i,j)\in \mathcal{E}_1(z^*,z_k)} \Ab_{ij}  \big ) $. From Proposition~\ref{prop: eq-quant} we have that $\lim_{n \rightarrow \infty} \sup_{t \in \mathbb{R}}|\mathbb{P}( \hat{L}_0 < t) - \mathbb{P}( \hat{L}_0(z_0)<t)|=0$ for any $z_0 \in \cC_0$. Therefore, it suffices for us to prove that $\mathbb{P}(\hat{\LRT} \geq q_{\alpha}) =\alpha +o(1)$ for one given true assignment $z_0 \in \cC_0$. Now we are ready to prove the validity of multiplier boostrap for estimating the quantile based on the leading term.
	\subsubsection{Bounding of error caused by plugging in $\hat{p},\hat{q}$}
	From previous section we know that
	$$\hat{L}_0(z_0) = g(\hat{p},\hat{q})\sigma_0\sup_{k \in [|B_{z^*}(r)|]} \Big \{\frac{1}{\sigma_0}  \sum_{i<j} (\boldsymbol{X}_{ij})_k \Big \}+g(\hat{p},\hat{q})\mu_0 +O_P(\rho_n),$$
	where $(\boldsymbol{{X}}_{ij})_k = \big ({\Ab}_{ij}-\mathbb{E}({\Ab}_{ij}) \big )\big ( \mathbbm{1}[(i,j) \in \mathcal{E}_2(z_0,z_k)]-\mathbbm{1}[(i,j) \in \mathcal{E}_1(z_0,z_k)]\big )$. For any $z_0 \in \cC_0$, we give the following notations:
	$$T_0=\!\!\!\sup_{k \in [|B_{z_0}(r)|]} \!\!\!\Big \{\frac{1}{\sigma_0} \sum_{i<j} (\boldsymbol{X}_{ij})_k \Big \},\quad  \Xi_0=\!\!\!\sup_{k \in [|B_{z_0}(r)|]} \!\!\!\Big \{ \frac{1}{\sigma_0} \sum_{i<j} \{\xi_{ij}\}_k \Big \}
	,\quad \Xi_0'=\!\!\!\sup_{k \in [|B_{z_0}(r)|]} \!\!\!\Big \{ \frac{1}{\hat{\sigma}_0} \sum_{i<j}\{\hat{\xi}_{ij}\}_k \Big \},$$
	and denote
	$$\tilde{W}_n=W_n/\hat{\sigma}_0=\sup_{k \in [|B_{z_0}(r)|]} \Big \{\frac{1}{\hat{\sigma}_0} \sum_{i<j} (\boldsymbol{\hat{X}}_{ij})_k e_{ij} \Big \},$$
	where $(\boldsymbol{\hat{X}}_{ij})_k = \big (\hat{\Ab}_{ij}-\mathbb{E}_{\hat{p},\hat{q}}(\hat{\Ab}_{ij}) \big )\big ( \mathbbm{1}[(i,j) \in \mathcal{E}_2(z_0,z_k)]-\mathbbm{1}[(i,j) \in \mathcal{E}_1(z_0,z_k)]\big ) $ and the ajacency matrix $\boldsymbol{\hat{A}}$ is generated by $\hat{p},\hat{q}$, and $\hat{\sigma}_0 = \sqrt{d(\mathcal{C}_0,\mathcal{C}_1)\big(\hat{p}(1-\hat{p})+\hat{q}(1-\hat{q}) \big )}$. $\xi_{ij} $ and $\hat{\xi}_{ij} $ are the independent mean zero Gaussian vectors with covariance matrix equal to that of $\boldsymbol{X}_{ij}$ and $\boldsymbol{\hat{X}}_{ij}$ respectively ($\{\xi_{ij}\}_k = 0$ if $(i,j) \notin \mathcal{E}_{1,2}(z_0,z_k)$, and the same for $\{\hat{\xi}_{ij} \}_k$ ). $\{e_{ij}\}_{i<j}$ are i.i.d standard Gaussians. By Corollary 2.1 in \cite{chernozhukov2013gaussian}, we have 
	\[
		\sup _{t \in \mathbb{R}}\left|\PP\left(T_0\leqslant t\right)-\PP\left(\Xi_0 \leqslant t\right)\right|
		= o(1);
	\]

	Also, by Lemma 3.2 and Corollary 3.1 of \cite{chernozhukov2013gaussian} we have 
	\[
		\sup _{t \in \mathbb{R}}\left|\mathbb{P}\left(\tilde{W}_n\leqslant t| \hat \Xb_{ij}\right)-\mathbb{P}\left(\Xi_0' \leqslant t\right)\right|
		= o_{P}(1).
	\]

	We let $\mathbf{\Sigma}^{\Xi_0}$ and $\mathbf{\Sigma}^{\Xi_0'}$ be the covariance matrix of the vectors $\Big \{\sum_{i<j}\{\xi_{ij}\}_k/ \sigma_0 \Big \}_k$ and $\Big \{ {\sum_{i<j} \{\hat{\xi}_{ij}\}_k}/\hat{\sigma}_0 \Big \}_k$ respectively. 
	Thus for $k,l \in [|B_{z_0}(r)|]$ we have:
	\begin{align*}
	\mathbf{\Sigma}^{\Xi_0}_{k,l}&=\frac{1}{\sigma_0^2}\Cov (\sum_{ij} \{\xi_{ij}\}_k,\sum_{ij}\{\xi_{ij}\}_l)\\
	&=\frac{1}{\sigma_0^2}\Cov (\sum_{i<j} (\boldsymbol{X}_{ij})_k ,\sum_{i<j} (\boldsymbol{X}_{ij})_l)\\
	&=\frac{|{\mathcal{E}}_2(z_0,z_k) \cap {\mathcal{E}}_2(z_0,z_l) |q(1-q) + |{\mathcal{E}}_1(z_0,z_k) \cap {\mathcal{E}}_1(z_0,z_l) |p(1-p)}{d(\mathcal{C}_0,\mathcal{C}_1)\big(p(1-p)+q(1-q) \big )}.
	\end{align*}
	Accordingly,
	\begin{align*}
	\mathbf{\Sigma}^{\Xi_0'}_{k,l}&=\frac{1}{{\hat{\sigma}_0}^2}\Cov (\sum_{i<j} (\boldsymbol{\hat{X}}_{ij})_k ,\sum_{i<j} (\boldsymbol{\hat{X}}_{ij})_l)\\
	&=\frac{|{\mathcal{E}}_2(z_0,z_k) \cap {\mathcal{E}}_2(z_0,z_l) |\hat{q}(1-\hat{q}) + |{\mathcal{E}}_1(z_0,z_k) \cap {\mathcal{E}}_1(z_0,z_l) |\hat{p}(1-\hat{p})}{d(\mathcal{C}_0,\mathcal{C}_1)\big(\hat{p}(1-\hat{p})+\hat{q}(1-\hat{q}) \big )}\\
	&=\frac{|{\mathcal{E}}_2(z_0,z_k) \cap {\mathcal{E}}_2(z_0,z_l) |q(1-q) + |{\mathcal{E}}_1(z_0,z_k) \cap {\mathcal{E}}_1(z_0,z_l) |p(1-p) +O_P(d(\cC_0,\cC_1)\sqrt{\rho_n}/n)}{d(\mathcal{C}_0,\mathcal{C}_1)\big(p(1-p)+q(1-q) \big ) + O_P(d(\cC_0,\cC_1)\sqrt{\rho_n}/n)}.
	\end{align*}
	Then we have 
	\begin{align*}
	\Delta_0&=\max_{k,l}|\mathbf{\Sigma}^{\Xi_0}_{k,l}-\mathbf{\Sigma}^{\Xi_0'}_{k,l}|
	\le \left|\frac{O_P(d(\cC_0,\cC_1)\sqrt{\rho_n}/n)}{{\hat{\sigma}_0}^2}\right|+\left|\mathbf{\Sigma}^{\Xi_0}_{k,l}\frac{O_P(d(\cC_0,\cC_1)\sqrt{\rho_n}/n)}{{\hat{\sigma}_0}^2}\right|
	=O_P(\frac{1}{\sqrt{n^2 \rho_n}}).
	\end{align*}
	Thus by Lemma 3.1 in \cite{chernozhukov2013gaussian}, there exists a constant $C$ such that 
	$$\sup _{t \in \mathbb{R}}\left|\mathbb{P}\left(\Xi_0\leqslant t\right)-\mathbb{P}\left(\Xi_0' \leqslant t\right)\right| \le C \Delta_{0}^{1 / 3}\left(1 \vee \log \left(|B_{z_0}(r)|/ \Delta_{0}\right)\right)^{2 / 3}=o_P({n^{-1/6-c_2/12}}).$$
	and thus 
	\[
		\sup _{t \in \mathbb{R}}\left|\mathbb{P}\left(\Xi_0\leqslant t\right)-\mathbb{P}\left(\Xi_0' \leqslant t\right)\right|
		= o_P(1),
	\]
	and in turn we have 
\[
		\sup _{t \in \mathbb{R}}\left|\mathbb{P}\left(T_0 \leqslant t\right)-\mathbb{P}\left(\tilde{W}_n\leqslant t| \hat \Xb_{ij}\right)\right|
		= o_{P}(1).
	\]	
\subsubsection{Validity of multiplier bootstrap in estimating LRT quantile}
	Now recall that $C_{\tilde{W}_n}(\alpha)$ is the $\alpha$ quantile of $\tilde{W}_n$ conditional on $\hat{\Xb}_{ij}$, and we would like to control the order of $C_{\tilde{W}_n}(\alpha)$ in order to bound the error in estimating the quantile of $\hat{\LRT}$. Give a constant $t>\sqrt{2c_0}$, we have
	\begin{align*}
	    \mathbb{P}(\tilde{W}_n & \geqslant t \sqrt{\log n}|\hat{\Xb}_{ij})  = \mathbb{P}(\Xi_0' \geqslant t \sqrt{\log n})+o_{P}(1)\\
	    & \le \sum_{k \in [|B_{z_0}(r)|]} \mathbb{P}\left(\Big \{ \frac{1}{\hat{\sigma}_0} \sum_{i<j}\{\hat{\xi}_{ij}\}_k \Big \} \geqslant t\sqrt{\log n}\right) +o_{P}(1)\\
	    & \lesssim |B_{z_0}(r)| e^{-\frac{t^2}{2} \log n} +o_{P}(1)
	    = O_{P}\left( n^{c_0 - t^2/2} \right)+o_{P}(1) = o_{P}(1).
	\end{align*}
	Thus we know that $C_{\tilde{W}_n}(\alpha) = O_{P}(\sqrt{\log n})$. 
	We know that $q_{\alpha} = g(\hat{p},\hat{q})\hat{\sigma}_0 C_{\tilde{W}_n}(\alpha) + g(\hat{p},\hat{q}) \hat{\mu}_0$, $\hat{\LRT} = \hat{L}_0 + \delta_n$ and also $\lim_{n \rightarrow \infty} \sup_{t \in \mathbb{R}}|\mathbb{P}( \hat{L}_0 < t) - \mathbb{P}( \hat{L}_0(z_0)<t)|=0$. Therefore,
	\begin{align*}
	\mathbb{P}(\hat{\LRT} \geq q_{\alpha}) & = \mathbb{P}(\hat{L}_0 + \delta_n \geq q_{\alpha}) = \mathbb{P}(\hat{L}_0(z_0) + \delta_n \geq q_{\alpha})+o(1)\\
	&=\mathbb{P}(g(\hat{p},\hat{q})\sigma_0 T_0+ g(\hat{p},\hat{q}) \mu_0 +\delta_n \geq g(\hat{p},\hat{q})\hat{\sigma}_0 C_{\tilde{W}_n}(\alpha) + g(\hat{p},\hat{q}) \hat{\mu}_0 )+o(1)\\
	&=\mathbb{P}\bigg (T_0 \geq \frac{\hat{\sigma}_0}{\sigma_0}C_{\tilde{W}_n}(\alpha) +\frac{ \hat{\mu}_0-  \mu_0}{\sigma_0} -\frac{\delta_n}{g(\hat{p},\hat{q})\sigma_0} \bigg)+o(1).
	\end{align*}
	We have that $|\hat{\sigma}_0-\sigma_0|=O_{P}(\sqrt{d(\cC_0,\cC_1)}/n)$ and $| \hat{\mu}_0-  \mu_0|=O_{P}(d(\cC_0,\cC_1)\sqrt{\rho_n}/n)$. Therefore,
	\begin{align*}
	\mathbb{P}(\hat{\LRT} \geq q_{\alpha})&=\mathbb{P}\bigg (T_0 \geq C_{\tilde{W}_n}(\alpha)+ \frac{C_1\sqrt{d(\cC_0,\cC_1)}}{\sigma_0n}C_{\tilde{W}_n}(\alpha)+\frac{C_2d(\cC_0,\cC_1)\sqrt{\rho_n}}{\sigma_0n}\! -\!\frac{\delta_n}{g(\hat{p},\hat{q})\sigma_0} \bigg)+o(1)\\
	&=\mathbb{P}\bigg (T_0 \geq C_{\tilde{W}_n}(\alpha)+ C_1 \sqrt{{\log n}}/({n^2\rho_n})+C_2{\sqrt{d(\cC_0,\cC_1)}}/{n} +C_3\sqrt{\frac{\rho_n}{d(\cC_0,\cC_1)}} \bigg)+o(1)\\
	&=\mathbb{P}\big (T_0 \geq C_{\tilde{W}_n}(\alpha)+ \Delta_n \big)+o(1),
	\end{align*} 
	where $\Delta_n = o_{P}(n^{-c})$ for some positive constant $c>0$. Now from previous results we have
	\begin{align*}
	   &| \mathbb{P}(\hat{\LRT} \geq q_{\alpha})-\alpha| \le |\mathbb{P}\big (T_0 \geq C_{\tilde{W}_n}(\alpha)+ \Delta_n \big) - \mathbb{P}\big (\tilde{W}_n \geq C_{\tilde{W}_n}(\alpha)+ \Delta_n \big)| \\&\quad + |\mathbb{P}\big (\tilde{W}_n \geq C_{\tilde{W}_n}(\alpha)+ \Delta_n \big) -\mathbb{P}\big (\tilde{W}_n \geq C_{\tilde{W}_n}(\alpha)\big)|+o(1)\\
	   & \leq \mathbb{P}(|\tilde{W}_n - C_{\tilde{W}_n}(\alpha)|\leq \Delta_n) +o_{P}(1).
	\end{align*}
	Now we study the distribution of $\tilde{W}_n$: if we denote $Y_k = \frac{1}{\hat{\sigma}_0} \sum_{i<j}(\hat{\boldsymbol{X}}_{ij})_k e_{ij}$, then $Y_k | \hat{\boldsymbol{X}} \sim N(0, \sigma^2_k)$, where $\sigma^2_k = \sum_{i<j} (\hat{\boldsymbol{X}}_{ij})_k^2 / \hat{\sigma}_0^2$, and $\sup_k|\mathbb{E}(\sigma_k^2)-1| \le |\hat{\sigma}_0^2/\hat{\sigma}_0^2 -1|+o_{P}(1) =o_{P}(1) $. Also, $|(\hat{\boldsymbol{X}}_{ij})_k^2 | < 1$. Under the event $\cA = \{ \hat{p}=o(1)\} \cap \{ \hat{q}= o(1) \}$ with $\PP(\cA) = 1-o(1)$, by Bernstein's inequality, we have 
	$$ \mathbb{P}_{\hat{\bX}}(|\sigma^2_k - \mathbb{E}_{\hat{\bX}}(\sigma_k^2)| > 1/2) \leq 2\exp \left(-\frac{\frac{1+1}{8}\hat{\sigma}_0^4}{(\frac{1}{6} +1) \hat{\sigma}_0^2}\right) = 2\exp \left(-\frac{3\hat{\sigma}_0^2}{14}\right),$$
	where $\PP_{\hat \bX}$ and $\EE_{\hat \bX}$ denotes probability and expectation with $\hat{p}$ and $\hat{q}$ fixed and consider only the randomness of $\hat \bX$. Also
	\begin{align*}
	    \mathbb{P}_{\hat{\bX}}(\min_k \sigma_k^2 < 1/2)& \leq \sum_k \mathbb{P}_{\hat{\bX}}(|\sigma^2_k - \mathbb{E}_{\hat{\bX}}(\sigma_k^2)| > 1/2)\\
	    & = 2|B_{z_0}(r)| \exp \left(-\frac{3\hat{\sigma}_0^2}{14}\right) = o_{P}(1),
	\end{align*}
	where the last $o_P(1)$ term is due to the fact that $\hat{\sigma}_0^2 = \Omega_P(n \rho_n) = \Omega_P(n^{c_2})$ and $|B_{z_0}(r)|=O(n^{c_0})$. Then by Lemma~2.1 in \cite{chernozhukov2013gaussian}, we have 
	\begin{align*}
	    \mathbb{P}(|\tilde{W}_n - C_{\tilde{W}_n}(\alpha)|\leq \Delta_n) & = \mathbb{P}(|\max_k Y_k - C_{\tilde{W}_n}(\alpha)|\leq \Delta_n) \leq \sup_{z \in \mathbb{R}}\mathbb{P}(|\max_k Y_k - z|\leq \Delta_n)\\
	    & = O_{P} \left(\Delta_n\left\{ \sqrt{2 \log |B_{z_0}(r)| } +\sqrt{\log (\min_k \sigma_k^2 /\Delta_n) } \right\}\right) =o_{P}(1),
	\end{align*}
	and thus $\lim_{n \rightarrow \infty}\sup_{z^* \in \cC_0} \mathbb{P}(\hat{\LRT} \ge q_{\alpha}) = \alpha$. 
	
	As for the Type I error, from the preceding proof we see that $\mathbb{P}(\LRT \geq q_{\alpha}) = \alpha +o_{P}(1)$, and the convergence of the $o_{P}(1)$ term is independent of $z^* \in \cC_0$ due to the symmetry of $\cC_0$. Therefore, we have
	\begin{align*}
	\sup_{z^* \in \cC_0} \mathbb{P}(\text{reject } H_0 ) &=\sup_{z^* \in \cC_0} \mathbb{P}(\LRT \geq q_{\alpha})=\alpha+o_P(1),
	\end{align*}
	and hence the claim follows. As for the Type II error, when the true assignment is $z^* \in \cC_1$, by  \eqref{eq: sup-c0c1} in Lemma~\ref{lm1-t1}, we have
	\begin{align*}
	\LRT &=\log \frac{\sup_{{z} \in \cC_1} f(\Ab;{z},\hat{p},\hat{q})}{\sup_{{z} \in \cC_0 \cup \cC_1} f(\Ab;{z},\hat{p}, \hat{q})}\\
	&= \log \frac{\sup_{{z} \in \cC_1} f(\Ab;{z},\hat{p},\hat{q})}{ f(\Ab;z^*,\hat{p},\hat{q})}+\log \frac{ f(\Ab;z^*,\hat{p},\hat{q})}{\sup_{{z} \in \cC_0 \cup \cC_1} f(\Ab;{z},\hat{p},\hat{q})}=O_{P}(\rho_n).
	\end{align*}
	And since $\hat{\sigma}_0 \asymp \sqrt{d(\cC_0,\cC_1)\hat{p}}, \hat{\mu}_0 \asymp -d(\cC_0,\cC_1)\hat{p} = -\Omega_P(n^{c_2})$ and $C_{\Tilde{W}_n}(\alpha) = O_P(\sqrt{\log n})$, we have $q_{\alpha}=g(\hat{p},\hat{q})\hat{\sigma}_0 C_{\Tilde{W}_n}(\alpha) + g(\hat{p},\hat{q})\hat{\mu}_0 \rightarrow -\infty$. Since the convergence is independent of $z^*$, we have for any true assignment $z_1 \in \cC_1$,
	$$\inf_{z^* \in \cC_1}\mathbb{P}(\text{reject } H_0) = 1-\sup_{z^* \in \cC_1}\mathbb{P}(\LRT\leq q_{\alpha}) = 1-o_{P}(1).$$

\section{Proof of Theorems for the Lower Bound}\label{sec: proof lwrbd}
In this section, we will prove the theorems for the lower bound. Similar as the upper bound, since Theorem~\ref{lb-main-g} and Theorem~\ref{lb-sec-g} are general versions of Theorem~\ref{lb-main} and Theorem~\ref{lb-sec}, we will only prove the general versions and the proof can be applied to Theorem~\ref{lb-main} and Theorem~\ref{lb-sec}, too. Also, the proof of Theorem~\ref{lb-main-g} is actually based on the proof of Theorem~\ref{lb-sec-g} under a stronger regime. Therefore, we will prove the two theorems together: we will first prove Theorem~\ref{lb-sec-g} under more general conditions, and then we will apply the proof of Theorem~\ref{lb-sec-g} to the proof of Theorem~\ref{lb-main-g} under stronger conditions.
\subsection{Proof of Theorem~\ref{lb-main-g} and Theorem~\ref{lb-sec-g}}\label{proof: lb main sec}
The proof proceeds in the following order: we first prove the results under the two conditions of Theorem \ref{lb-sec-g}, namesly the proof of Theorem \ref{lb-sec-g} (1) and the proof of Theorem \ref{lb-sec-g} (2), then we provide the proof of Theorem \ref{lb-main-g}. 
\subsubsection{Proof of Theorem \ref{lb-sec-g} (1)}

	As for the minimax rate, we have: 
	\begin{align*}
	r(\cC_0, \cC_1) & = \underset{\psi}{\min} \bigg \{ \underset{z \in \cC_0}{\sup} \mathbb{P}_{z}(\psi=1) + \underset{z \in \cC_1}{\sup} \mathbb{P}_{z}(\psi=0)  \bigg\} \\
	& \geq \underset{\psi}{\min} \bigg \{ \mathbb{P}_{z_0}(\psi=1) +  \mathbb{P}_{z_1}(\psi=0)  \bigg\}.
	\end{align*}
	where $z_0$ and $z_1$ are fixed cluster assignments in $\cC_0$ and $\cC_1$ respectively. For a given adjacency matrix $\Ab$, we know that $\psi$ is a function of $\Ab$, and the only information of $\Ab$ relevant to classification of the true assignment is $\{\Ab_{ij}, (i,j) \in \mathcal{E}_1(z_0,z_1) \bigcup \mathcal{E}_2(z_0,z_1)  \}$. Larger size of $\mathcal{E}_1(z_0,z_1)$ and $\mathcal{E}_2(z_0,z_1)$ will provide more information and lead to smaller type I and type II error. Thus, the worst case is when the size of $\mathcal{E}_1(z_0,z_1)$ and $\mathcal{E}_2(z_0,z_1)$ obtains the infimum, i.e., $d(z_0,z_1)=n_1(z_0, z_1) \vee n_2(z_0, z_1)=d(\mathcal{C}_0, \mathcal{C}_1)$. 
	\\
	To obtain $\inf_{\psi} \big \{ \sup_{z \in \cC_0}\mathbb{P}_{z}(\psi=1) +  \sup_{z \in \cC_1}\mathbb{P}_{z}(\psi=0)  \big\}$, the optimal method $\tilde{\psi}$ must be the mode of the posterior distribution. For the convenience of notations, we denote $L(z,\Ab)$ as $f(\Ab; z,p,q)$, and $n_i$ as $n_i(z_0,z_1)$, $i =1,2$ for short: 
	\begin{align*}
	L(z_0, \Ab) &\propto p^{\sum_{(i,j) \in \mathcal{E}_1(z_0,z_1)} \Ab_{ij}} (1-p)^{n_1 - \sum_{(i,j) \in \mathcal{E}_1(z_0,z_1)} \Ab_{ij}} q^{\sum_{(i,j) \in \mathcal{E}_2(z_0,z_1)} \Ab_{ij}} (1-q)^{n_2 - \sum_{(i,j) \in \mathcal{E}_2(z_0,z_1)} \Ab_{ij}}\\
	L(z_1, \Ab) &\propto p^{\sum_{(i,j) \in \mathcal{E}_2(z_0,z_1)} \Ab_{ij}} (1-p)^{n_2 - \sum_{(i,j) \in \mathcal{E}_2(z_0,z_1)} \Ab_{ij}} q^{\sum_{(i,j) \in \mathcal{E}_1(z_0,z_1)} \Ab_{ij}} (1-q)^{n_1 - \sum_{(i,j) \in \mathcal{E}_1(z_0,z_1)} \Ab_{ij}}
	\end{align*}
	and correspondingly,
	\begin{align*}
	\tilde{\psi}(\Ab)= \left \{ \begin{array}{cl} 0, & \text{if } L(z_0,\Ab) > L(z_1,\Ab);\\
	1, &  \text{if } L(z_0,\Ab) \leq L(z_1,\Ab). \end{array} \right.
	\end{align*}
	Then $\mathbb{P}_{z_0}(\tilde{\psi}=1) =   \mathbb{P}_{ z_0}(L(z_0,\Ab) \leq L(z_1,\Ab))$ and $\mathbb{P}_{ z_1}(\tilde{\psi}=0) =\mathbb{P}_{z_1}(L(z_0,\Ab) > L(z_1,\Ab))$. Without loss of generality, we assume that $n_1(z_0,z_1) \geq n_2(z_0,z_1)$. Then, if we expend the size of $\mathcal{E}_2(z_0,z_1)$ to be the same as $\mathcal{E}_1(z_0,z_1)$, adding i.i.d entries $\{\Ab_{ij}, (i,j) \in \mathcal{E}_2^L(z_0,z_1) \backslash \mathcal{E}_2(z_0,z_1)\}$ conforming to the same distribution as $\{\Ab_{ij}, (i,j) \in \mathcal{E}_2(z_0,z_1)\}$, more information will be provided and the error rate will decrease, where $\mathcal{E}_2^{L}(z_0,z_1)$ denotes the set expended on $\mathcal{E}_2(z_0,z_1)$, and we have: 
	\begin{align*}
	\tilde{L}(z_0, \Ab) &\propto p^{\sum_{(i,j) \in \mathcal{E}_1(z_0,z_1)} \Ab_{ij}} (1-p)^{n_1 - \sum_{(i,j) \in \mathcal{E}_1(z_0,z_1)} \Ab_{ij}} q^{\sum_{(i,j) \in \mathcal{E}_2^L(z_0,z_1)} \Ab_{ij}} (1-q)^{n_1 - \sum_{(i,j) \in \mathcal{E}_2^L(z_0,z_1)} \Ab_{ij}}\\
	\tilde{L}(z_1, \Ab) &\propto p^{\sum_{(i,j) \in \mathcal{E}_2^L(z_0,z_1)} \Ab_{ij}} (1-p)^{n_1 - \sum_{(i,j) \in \mathcal{E}_2^L(z_0,z_1)} \Ab_{ij}} q^{\sum_{(i,j) \in \mathcal{E}_1(z_0,z_1)} \Ab_{ij}} (1-q)^{n_1 - \sum_{(i,j) \in \mathcal{E}_1(z_0,z_1)} \Ab_{ij}}
	\end{align*}
	Thus we can obtain lower bound on the minimax rate:
	\begin{align*}
	r(\tilde{\psi})&=   \mathbb{P}_{ z_0}(L(z_0,\Ab) \leq L(z_1,\Ab)) +\mathbb{P}_{ z_1}(L(z_0,\Ab) > L(z_1,\Ab)) \\
	& \geq \mathbb{P}_{ z_0}(\tilde{L}(z_0,\Ab) \leq \tilde{L}(z_1,\Ab)) +\mathbb{P}_{ z_1}(\tilde{L}(z_0,\Ab) > \tilde{L}(z_1,\Ab))\\
	& = \mathbb{P}_{z_0} \bigg(\sum_{(i,j) \in  \mathcal{E}_1(z_0,z_1)} \Ab_{ij} \leq \sum_{(i,j) \in  \mathcal{E}_2^L(z_0,z_1)} \Ab_{ij} \bigg) + \mathbb{P}_{z_1} \bigg(\sum_{(i,j) \in  \mathcal{E}_1(z_0,z_1)} \Ab_{ij} > \sum_{(i,j) \in  \mathcal{E}_2^L(z_0,z_1)} \Ab_{ij} \bigg) \\
	& \geq 2\mathbb{P}\Big(\sum_{u=1}^{n_1} X^u \geq \sum_{u=1}^{n_1} Y^u\Big).
	\end{align*}
	where $\{X_u\} \overset{\text{i.i.d}}{\sim} \text{Ber}(q), \{Y^u\} \overset{\text{i.i.d}}{\sim} \text{Ber}(p)$, and $\{X^u\}$ are independent to $\{Y^u\}$.\\
	\\
	Now $n_1=d(\mathcal{C}_0,\mathcal{C}_1)$, and both $p$ and $q$ change with $n_1$. We have $\mathbb{E}(|X^u - Y^u- \mathbb{E}[X^u - Y^u]|^3) \asymp p(1-q) + q(1-p)$. Since $0<q< p <1-\delta$, we have $\delta p < p(1-q)+q(1-p) < 2 p$. Thus $\mathbb{E}(|X^u - Y^u- \mathbb{E}[X^u - Y^u]|^3) \asymp p$. Similarly $\text{Var}(X^u - Y^u) \asymp p$. Thus, $$\frac{\sum_{u=1}^{n_1} \mathbb{E}(|X^u - Y^u- \mathbb{E}[X^u - Y^u]|^3)}{ \text{Var}\left( \sum_{u=1}^{n_1} \left\{X^u - Y^u\right\} \right)^{3/2} } \asymp \frac{n_1[p(1-q) + q(1-p)]}{ {n_1}^{3/2} (q(1-q) + p(1-p))^{3/2}} \asymp \frac{1}{\sqrt{n_1p}} \rightarrow 0$$ Therefore, by the Lyapunov’s Central Limit Theorem and the independence of $\{ X^u\}$ and $\{ Y^u\}$, as $n \rightarrow \infty$, we have $\sum_{u=1}^{n_1} X^u - \sum_{u=1}^{n_1} Y^u \overset{d}{\rightarrow} N(n_1(q - p), n_1 q(1-q)+  n_1 p(1-p))$. Therefore, $$\mathbb{P}(\sum_{u=1}^{n_1} X^u \geq \sum_{u=1}^{n_1} Y^u)= \mathbb{P} (\sum_{u=1}^{n_1} X^u - \sum_{u=1}^{n_1} Y^u \geq 0)= 1- \Phi(\frac{\sqrt{n_1}(p-q)}{\sqrt{p(1-p)+q(1-q)}})+o(1).$$
	When $\limsup\limits_{n \rightarrow \infty}n_1I(p,q)=O(1)$, we can see that $p-q=o(1)$. We have 
	\begin{align*}
	I(p,q)&=-2\log \Big (1- \frac{(\sqrt{p}-\sqrt{q})^2+(\sqrt{1-p} -\sqrt{1-q})^2}{2}\Big)\\
	&=\left( \frac{(p-q)^2}{(\sqrt{p}+\sqrt{q})^2} + \frac{(p-q)^2}{(\sqrt{1-p}+\sqrt{1-q})^2} \right)(1+o(1))\\
	&\geq \frac{\delta}{2}\frac{(p-q)^2}{{p(1-p)+q(1-q)}}(1+o(1)).
	\end{align*}
	Thus, if $\limsup\limits_{n \rightarrow \infty}n_1I(p,q)\leq \delta \Phi^{-1} (3/4)^2/2$, namely, $\limsup\limits_{n \rightarrow \infty}\sqrt{n_1}(p-q)/\sqrt{p(1-p)+q(1-q)} \leq \limsup\limits_{n \rightarrow \infty} \sqrt{2n_1I(p,q)/\delta}\leq \Phi^{-1} (3/4)$, we have $$\mathbb{P}(\sum_{u=1}^{n_1} X^u \geq \sum_{u=1}^{n_1} Y^u) \geq  1- \Phi(\sqrt{n_1}(p-q)/\sqrt{p(1-p)+q(1-q)}) \geq 1/4.$$and $$r(\cC_0, \cC_1) \geq 2(1- \Phi(\sqrt{n_1}(p-q)/\sqrt{p(1-p)+q(1-q)}) ) \geq 1/2.$$
	\subsubsection{Proof of Theorem \ref{lb-sec-g} (2)}
	When $d(z_0, \cC_1) I (p,q) \rightarrow \infty$, if there exists a $z_0 \in \cC_0$ and some $r=d(z_0,\cC_1) +O(1)$ such that $\limsup_{n \rightarrow \infty} d(\cC_0,\cC_1) I(p,q) / \log N(B_{z_0}(r),0) < 1$, then we take a 0-packing $\cP(B_{z_0}(r),0)$ (denoted $\cP(0)$ for short) of the ball $B_{z_0}(r)$, and we have:
	\begin{align*}
	r(\cC_0, \cC_1) & = \underset{\psi}{\min} \bigg \{ \underset{z \in \cC_0}{\sup} \mathbb{P}_z(\psi=1) + \underset{z \in \cC_1}{\sup} \mathbb{P}_z(\psi=0)  \bigg\} 
	 \geq \underset{\psi}{\min} \bigg \{ \mathbb{P}_{z_0}(\psi=1) +  \sup_{z \in \cP(0)} \mathbb{P}_z(\psi=0)  \bigg\}\\
	&=\underset{\psi}{\min} \bigg \{ \sum_{{\rm A}}\Big (\mathbb{P}_{ z_0}(\psi=1|\Ab={\rm A}) \mathbb{P}_{z_0}(\Ab={\rm A}) +  \sup_{z \in \cP(0)} \mathbb{P}_{z}(\psi=0 |\Ab={\rm A}) \mathbb{P}_z (\Ab={\rm A}) \Big )  \bigg\}\\
	&= \underset{\psi}{\min} \bigg \{ \sum_{{\rm A}}\Big (\mathbbm{1}(\psi({\rm A})=1) \mathbb{P}_{z_0}(\Ab={\rm A}) +  \mathbbm{1}(\psi({\rm A})=0 )\sup_{z \in \cP(0)}  \mathbb{P}_z(\Ab={\rm A}) \Big )  \bigg\},
	\end{align*}
	where the sum over ${\rm A}$ is the summation over all possible realizations of the adjacency matrix $\Ab$. Thus the optimal method $\tilde{\psi}$ in this scenario should be:
	\begin{align*}
	\tilde{\psi}(\Ab)= \left \{ \begin{array}{cl} 0, & \text{if } L(z_0,\Ab={\rm A})  \ge  \sup_{z \in \cP(0)} L (z,\Ab={\rm A} ); \\
	1, &  \text{if }  L(z_0,\Ab={\rm A}) < \sup_{z \in \cP(0)} L(z,\Ab={\rm A}).  \end{array} \right.
	\end{align*}
	and we have $L(z_0,\Ab={A}) < \sup_{z \in \cP(0)} L(z,\Ab={A})$
	\begin{align*}
	r(\cC_0, \cC_1) & \geq \mathbb{P}_{z_0}(\tilde{\psi } =1) + \sup_{z \in \cP(0)} \mathbb{P}_z(\tilde{\psi} = 0)\\
	&=\mathbb{P}_{z_0}\Big (L(z_0,\Ab={\rm A})  < \sup_{z \in \cP(0)} L(z, \Ab={\rm A}) \Big ) \\
	&\qquad+ \sup_{z \in \cP(0)} \mathbb{P}_z \Big (L(z_0,\Ab={\rm A})  \geq  \sup_{z \in \cP(0)} L(z,\Ab={\rm A}) \Big )\\
	&=\mathbb{P}_{z_0}\Big ( \sup_{z \in \cP(0)} \log  L(z,\Ab={\rm A}) - \log L(z_0, \Ab={\rm A}) > 0\Big ) \\
	&\qquad+ \sup_{z \in \cP(0)} \mathbb{P}_z\Big (\sup_{z \in \cP(0)} \log  L(z,\Ab={\rm A}) - \log L(z_0,\Ab={\rm A}) \leq  0 \Big ).
	\end{align*}
	Similar with the case when $d(z_0,\cC_1)I(p,q) = O(1)$, we can expand each $\mathcal{E}_2(z_0,z)$ to $\mathcal{E}_2^L(z_0,z)$ (or $\mathcal{E}_1(z_0,z)$ to $\mathcal{E}_1^L(z_0,z)$, we use the former notation for convenience) so that $\mathcal{E}_1(z_0,z)$ and $\mathcal{E}_2(z_0,z)$ are of equal sizes, and then we have
	\begin{align*}
	r(\cC_0, \cC_1) & \geq \mathbb{P}_{z_0}\bigg( \sup_{z \in \cP(0)} \Big( \sum_{(i,j) \in  \mathcal{E}_2^L(z_0,z)} \Ab_{ij} - \sum_{(i,j) \in  \mathcal{E}_1(z_0,z)} \Ab_{ij} \Big) >0\bigg)\\
	&\quad +\sup_{z \in \cP(0)} \mathbb{P}_{z}\bigg (\sup_{z \in \cP(0)} \Big( \sum_{(i,j) \in  \mathcal{E}_2^L(z_0,z)} \Ab_{ij} - \sum_{(i,j) \in  \mathcal{E}_1(z_0,z)} \Ab_{ij} \Big) \leq 0 \bigg )\\
	&\geq \mathbb{P}_{z_0}\bigg( \sup_{z \in \cP(0)} \Big( \sum_{(i,j) \in  \mathcal{E}_2^L(z_0,z)} \Ab_{ij} - \sum_{(i,j) \in  \mathcal{E}_1(z_0,z)} \Ab_{ij} \Big) >0\bigg)\\
	&= \mathbb{P}_{z_0}\bigg( \sup_{z \in \cP(0)} \Big( \sum_{u=1}^{n_1(z_0,z)} X^u_{z} -\sum_{u=1}^{n_1(z_0,z)} Y^u_{z}\Big) >0\bigg),
	\end{align*}
	where $\{ X^u_{z}\} \overset{\text{i.i.d}}{\sim} \text{Ber}(q)$, $\{ Y^u_{z}\} \overset{\text{i.i.d}}{\sim} \text{Ber}(p)$, $\{ X^u_{z_i}\} \perp \{ X^u_{z_j}\}, i \neq j$, $\{ Y^u_{z_i}\} \perp \{ Y^u_{z_j}\}, i \neq j$ and $\{ X^u_{z_i}\} \perp \{ Y^u_{z_j}\}, \forall i, j$. By Lemma~5.2 in 
	\cite{zhang2016minimax}, we know that there exists $\eta \rightarrow 0$ such that
	$$\mathbb{P}_{z_0} \Big (\sum_{u=1}^{n_1(z_0,z)} X^u_{z} -\sum_{u=1}^{n_1(z_0,z)} Y^u_{z} >0 \Big)  \geq \exp \big(-(1+\eta) d(z_0,\mathcal{C}_1)I(p,q)\big). $$ 
	
	When $ \limsup_{n \rightarrow \infty} d(z_0,\mathcal{C}_1) I(p,q) /\log |\cP(0)| <1$, for sufficiently large $n$ we have $(1+\eta)d(z_0, \cC_1)I(p,q) \le \log |\cP(0)|$, and since $x>1-(1/2)^x$ for $x>0$, we have that for $n$ large enough
	\begin{align*}
	    &\mathbb{P}_{z_0} \Big (\sum_{u=1}^{n_1(z_0,z)} X^u_{z} -\sum_{u=1}^{n_1(z_0,z)} Y^u_{z} >0 \Big)  \geq \exp \big(-(1+\eta) d(z_0,\mathcal{C}_1)I(p,q)\big)\\
	    & \geq \exp \big(-\log |\cP(0)|\big) = 1/|\cP(0)|
	    \geq 1-(1/2)^{1/|\cP(0)|},
	\end{align*}
and thus 
	\begin{align*}r(\cC_0, \cC_1) &\geq \mathbb{P}_{z_0}\bigg( \sup_{z \in \cP(0)} \Big( \sum_{u=1}^{n_1(z_0,z)}\!\!\! X^u_{z} -\!\!\!\!\!\sum_{u=1}^{n_1(z_0,z)}\!\!\! Y^u_{z}\Big) >0\bigg)
	 = 1- \mathbb{P}_{z_0}\bigg( \sup_{z \in \cP(0)} \Big( \sum_{u=1}^{n_1(z_0,z)} \!\!\!X^u_{z} -\!\!\!\!\!\sum_{u=1}^{n_1(z_0,z)} \!\!\!Y^u_{z}\Big) \le 0\bigg)\\
	&=1-\Pi_{z \in \cP(0)}\mathbb{P}_{z_0}\bigg(  \sum_{u=1}^{n_1(z_0,z)} \!\!\!X^u_{z} -\!\!\!\!\!\sum_{u=1}^{n_1(z_0,z)} \!\!\!Y^u_{z} \le 0\bigg)
	\geq 1-\left\{(1/2)^{1/|\cP(0)|} \right\}^{|\cP(0)|}=1/2.\end{align*}
	The statement is true for any 0-packing of the ball $B_{z_0}(r)$, and thus the statement follows.
	\subsubsection{Proof of Theorem \ref{lb-main-g}}
	Under the regime $1/\rho_n=o(n^{1-c_2})$, we take one $\sqrt{d(z_0,\cC_1)}$-packing $\cP(B_{z_0}(r),\sqrt{d(z_0,\cC_1)})$ (denoted $\tilde{\cP}$ for short) of the ball $B_{z_0}(r)$, similar with the proof of Theorem~\ref{t1-mtd}, by Corollary~2.1 in \cite{chernozhukov2013gaussian}, we have:
	\begin{align*}
	r(\cC_0, \cC_1) &\geq \mathbb{P}_{z_0}\bigg( \sup_{z \in \tilde{\cP}} \Big( \sum_{u=1}^{n_1(z_0,z)} X^u_{z} -\sum_{u=1}^{n_1(z_0,z)} Y^u_{z}\Big) >0\bigg)\\
	&=\mathbb{P}_{z_0}\bigg( \sup_{z \in \tilde{\cP}} \Big( \sum_{u=1}^{d(z_0,\mathcal{C}_1)} X^u_{z} -\sum_{u=1}^{d(z_0,\mathcal{C}_1)} Y^u_{z}\Big) > \delta_n \bigg)\\
	&=\mathbb{P}_{z_0}\bigg( \sup_{z \in \tilde{\cP}} \Big( \sum_{u=1}^{d(z_0,\mathcal{C}_1)} X^u_{z} -\sum_{u=1}^{d(z_0,\mathcal{C}_1)} Y^u_{z} + d(z_0,\mathcal{C}_1)(p-q) \Big) >d(z_0,\mathcal{C}_1)(p-q) +\delta_n \bigg)\\
	&=\mathbb{P}_{z_0}\bigg( \sup_{z \in \tilde{\cP}} \xi_z>\frac{d(z_0,\mathcal{C}_1)(p-q) }{\sigma_d} +\delta_n / \sigma_d \bigg) + o(1).
	\end{align*}
	where $\delta_n = \sup_{z \in \tilde{\cP}} \Big( \sum_{u=1}^{d(z_0,\mathcal{C}_1)} X^u_{z} -\sum_{u=1}^{d(z_0,\mathcal{C}_1)} Y^u_{z}\Big) - \sup_{z \in \tilde{\cP}} \Big( \sum_{u=1}^{n_1(z_0,z)} X^u_{z} -\sum_{u=1}^{n_1(z_0,z)} Y^u_{z}\Big)  = O(1)$ and $\sigma_d=\sqrt{d(z_0,\mathcal{C}_1)\big( p(1-p)+q(1-q) \big)}$, and $\{\xi_z\}_{z \in \tilde{\cP}}$ are standard Gaussian variables with the same covariance matrix as $ \{(\sum_{u=1}^{d(z_0,\mathcal{C}_1)} X^u_{z} -\sum_{u=1}^{d(z_0,\mathcal{C}_1)} Y^u_{z} + d(z_0,\mathcal{C}_1)(p-q))/\sigma_d\}_{z \in \tilde{\cP}}$. 
	
	By Lemma~2.1 in \cite{chernozhukov2013gaussian}, combined with the fact that $d(z_0,\cC_1)=\Omega_P(n)$ and $1/\rho_n = o(n^{1-c_2})$, we have that 
	$$\bigg|\mathbb{P}_{z_0}\bigg( \sup_{z \in \tilde{\cP}} \xi_z>\frac{d(z_0,\mathcal{C}_1)(p-q) }{\sigma_d} +\delta_n / \sigma_d \bigg) - \mathbb{P}_{z_0}\bigg( \sup_{z \in \tilde{\cP}} \xi_z>\frac{d(z_0,\mathcal{C}_1)(p-q) }{\sigma_d} \bigg)\bigg| \lesssim \frac{\delta_n}{\sigma_d}\sqrt{\log n}=o(1). $$
	We let $\{\tilde{X}^u_z\}_{u,z}$ be i.i.d $ \text{Ber}(q)$ random variables and $\{\tilde{Y}^u_z\}_{u,z}$ be i.i.d $ \text{Ber}(p)$ random variables, and $\{\tilde{X}^u_z\}_{u,z}$ and $\{\tilde{Y}^u_z\}_{u,z}$ are independent of each other. Then for each $z \in \tilde{\cP}$, $\sum_{u=1}^{d(z_0,\mathcal{C}_1)} \tilde{X}^u_{z} -\sum_{u=1}^{d(z_0,\mathcal{C}_1)} \tilde{Y}^u_{z} + d(z_0,\mathcal{C}_1)(p-q)$ shares the same distribution with $\sum_{u=1}^{d(z_0,\mathcal{C}_1)} X^u_{z} -\sum_{u=1}^{d(z_0,\mathcal{C}_1)} Y^u_{z} + d(z_0,\mathcal{C}_1)(p-q)$. We let $\{\tilde{\xi}_z\}_{z \in \tilde{\cP}}$ be the corresponding Gaussian analog of $ \{(\sum_{u=1}^{d(z_0,\mathcal{C}_1)} \tilde{X}^u_{z} -\sum_{u=1}^{d(z_0,\mathcal{C}_1)} \tilde{Y}^u_{z} + d(z_0,\mathcal{C}_1)(p-q))/\sigma_d\}_{z \in \tilde{\cP}}$. Then we have:
	\begin{align*}
	|\Cov (\xi_{z_i},\xi_{z_j})-\Cov (\tilde{\xi}_{z_i},\tilde{\xi}_{z_j})|= \left \{ \begin{array}{cl} 0 & \text{if } i=j ,\\
	O(\frac{1}{\sqrt{d(z_0,\mathcal{C}_1)}})&  \text{if }  i \neq j. \end{array} \right.
	\end{align*}
	Thus by Lemma~3.1 in \cite{chernozhukov2013gaussian}, we have $\Delta_{0}= O\big (1/\sqrt{d(z_0,\mathcal{C}_1)} \big )$, and 
	$$\sup_{t \in \mathbb{R}} \Big| \mathbb{P}_{z_0}\bigg( \sup_{z \in \tilde{\cP}} \xi_z>t\bigg) - \mathbb{P}_{z_0}\bigg( \sup_{z \in \tilde{\cP}} \tilde{\xi}_z>t\bigg)\Big| \leq C\Delta_{0}^{1/3}(\log |\tilde{\cP}|/\Delta_{0})^{2/3} =o(1),$$
	and thus 
	\begin{align*}
	&\mathbb{P}_{z_0}\bigg( \sup_{z \in \tilde{\cP}} \Big( \sum_{u=1}^{d(z_0,\mathcal{C}_1)} X^u_{z} -\sum_{u=1}^{d(z_0,\mathcal{C}_1)} Y^u_{z} + d(z_0,\mathcal{C}_1)(p-q) \Big) >d(z_0,\mathcal{C}_1)(p-q) \bigg)\\
	&=\mathbb{P}_{z_0}\bigg( \sup_{z \in \tilde{\cP}} \Big( \sum_{u=1}^{d(z_0,\mathcal{C}_1)} \tilde{X}^u_{z} -\sum_{u=1}^{d(z_0,\mathcal{C}_1)} \tilde{Y}^u_{z} + d(z_0,\mathcal{C}_1)(p-q) \Big) >d(z_0,\mathcal{C}_1)(p-q) \bigg)+o(1).
	\end{align*}
	Then similar with previous proof, we have when $\limsup \lim_{n \rightarrow \infty} d(z_0,\mathcal{C}_1) I(p,q) /\log |\tilde{\cP}| <1$, 
	$$r(\cC_0, \cC_1) \geq \mathbb{P}_{z_0}\bigg( \sup_{z \in B(r_K)} \Big( \sum_{u=1}^{n_1(z_0,z)} \tilde{X}^u_{z} -\sum_{u=1}^{n_1(z_0,z)} \tilde{Y}^u_{z}\Big) >0\bigg) +o(1)\geq 1/2+o(1).$$
	Also the resullts hold for any $\sqrt{d(z_0,\cC_1)}$-packing of the ball $B_{z_0}(r)$. Therefore, we proved the claim.
\subsection{Proof of Corollary \ref{col1-ex-rec}}\label{sec: proof ex rec}
We prove by contradiction. When $\limsup\limits_{n \rightarrow \infty} nI(p,q)/(K \log n) < 1$, we take an assignment $z_0$ satisfying that $\max_{k}|n_k(z_0)-n/K|=c_K=O(1)$, and let $z^*$ denote the true assignment. Then we consider the null hypothesis $\text{H}_0: \exists \sigma \in S_K \quad \text{s.t.} \quad \sigma(z^*)=z_0$. Then we can see that for this hypothesis test, $\cC_0 = \{z: z=\sigma(z_0), \sigma \in S_K\}$, and $\cC_1 = [K]^n \backslash \cC_0$. If there exists an estimator $\hat{z}(\Ab)$ that recovers the communities with high probability, then we can propose the testing procedure as: $\psi^{\text{exact}}= 0$ if $\hat{z}(\Ab) \in \cC_0$ and 1 otherwise. Thus the minimax rate $r(\cC_0, \cC_1) \le \sup_{z \in \cC_0}\mathbb{P}_z(\psi^{\text{exact}} = 1) + \sup_{z \in \cC_1}\mathbb{P}_z(\psi^{\text{exact}} = 0)=o_{P}(1)$. Now we consider the ball $B_{z_0}(r)$ with $r=d(z_0,\cC_1) +2c_K$: to change an assignment $z_0 \in \cC_0$ into $z_1 \in \cC_1$, the simplest way is to change the cluster label of one node, and $d(z_0,\mathcal{C}_1) = \min_{k}n_k(z_0)$, and $|d(z_0,z_1) - d(z_0,\cC_1)| \le 2c_K$ for any $z_1$ constructed such way. For any $z_i, z_j \in \cP\big (B_{z_0}(r), \sqrt{d(z_0,\cC_1)}\big )$, the mis-clustered node should be different. Otherwise, $|\mathcal{E}_{1,2}(z_0,z_i) \cap \mathcal{E}_{1,2}(z_0,z_j)| \geq n/K+O(1) \gg\sqrt{d(z_0,\cC_1)}$. Therefore, $N\big (B_{z_0}(r), \sqrt{d(z_0,\cC_1)}\big )=n$. Thus by Theorem~\ref{lb-main-g}, $\liminf_{n \rightarrow \infty} r(\cC_0, \cC_1) \ge 1/2 $, which is in contradiction to the previous conclusion that $r(\cC_0.\cC_1) = o(1)$. Thus the claim follows.
\section{Proof of technical lemmas}\label{sec: proof of tec lem}
Now we will provide proofs for the technical lemmas used for the proof of Theorem~\ref{t1-mtd-g}. 
\subsection{Proof of Lemma~\ref{lm2-t1}}\label{sec: proof lm2-t1-g}

	It suffices for us to prove Lemma \ref{lm2-t1-g}, the more general version of Lemma \ref{lm2-t1}. Due to the structure of $\bL_z$, it suffices for us to prove that the edge-wise distance between assignments are permutation-invariant. 
	
	For any given $z_0 \in \cC_0$ and $z_1, z_1' \in \cC_1$, we have:
	\begin{align*}
	n_1(z_0,z_1)&=\sum_{i<j, i,j \in [n]} \mathbbm{1}\big (z_0(i)=z_0(j),z_1(i)\neq z_1(j) \big )\\
	&=\sum_{i<j, i,j \in [n]} \mathbbm{1}\Big (\sigma \big (z_0(i)\big)=\sigma \big(z_0(j) \big),\sigma\big(z_1 (i)\big)\neq \sigma\big(z_1(j)\big) \Big )\\
	&=\sum_{\uptau(i)<\uptau(j), i,j \in [n]} \mathbbm{1}\Big (\tau \circ \sigma(z_0)\big(\tau(i)\big)=\tau \circ \sigma(z_0)\big(\tau(j)\big),\tau \circ \sigma(z_1)\big(\tau(i)\big)\neq \tau \circ \sigma(z_1)\big(\tau(j)\big) \Big )\\
	&=	n_1(\tau \circ \sigma(z_0),\tau \circ \sigma(z_1)).
	\end{align*}
	Then very similarly we have $n_2(z_0,z_1)=n_2\big (\tau \circ \sigma(z_0),\tau \circ \sigma(z_1)\big )$ and thus $d(z_0,z_1) = n_1(z_0,z_1) \vee n_2(z_0,z_1)=n_1\big (\tau \circ \sigma(z_0),\tau \circ \sigma(z_1)\big ) \vee n_2\big (\tau \circ \sigma(z_0),\tau \circ \sigma(z_1)\big ) =d\big (\tau \circ \sigma(z_0),\tau \circ \sigma(z_1)\big )$. This suggests that the permutation $\tau \circ \sigma$ does not change the distance between assignments. Also,
	\begin{align*}
	&|\mathcal{E}_1(z_0,z_1)\cap \mathcal{E}_1(z_0,z_1')|=\sum_{i<j, i,j \in [n]} \mathbbm{1}\big (z_0(i)=z_0(j),z_1(i)\neq z_1(j),z_1'(i)\neq z_1'(j)  \big )\\
	&=\sum_{i<j, i,j \in [n]} \mathbbm{1}\big (\tau \circ \sigma(z_0)(i)=\tau \circ \sigma(z_0)(j),\tau \circ \sigma(z_1)(i)\neq \tau \circ \sigma(z_1')(j),\tau \circ \sigma(z_1')(i)\neq \tau \circ \sigma(z_1)(j) \big )\\
	&=\big|\mathcal{E}_1\big(\tau \circ \sigma(z_0),\tau \circ \sigma(z_1)\big)\cap \mathcal{E}_1\big(\tau \circ \sigma(z_0),\tau \circ \sigma(z_1')\big)\big|.
	\end{align*}
	And similarly, 
	$$|\mathcal{E}_2(z_0,z_1)\cap \mathcal{E}_2(z_0,z_1')| = |\mathcal{E}_2(\tau \circ \sigma(z_0),\tau \circ \sigma(z_1))\cap \mathcal{E}_2(\tau \circ \sigma(z_0),\tau \circ \sigma(z_1'))|.$$
	Thus the cardinality of the intersection of the sets $\cE_i, i=1,2$ is also invariant under the permutation $\tau \circ \sigma$. 
	
	Now for any $z_0, z_0' \in \cC_0$, if $\tau \circ \sigma(z_0) = z_0'$, and $d(z_0, z_1) = d(z_0,\cC_1)$, from previous results we have $d(z_0',\tau \circ \sigma(z_1)) =d(z_0,\mathcal{C}_1)$. If there exists an assignment $z_1' \in \cC_1$ such that $d(z_0',z_1')< d(z_0',\tau \circ \sigma(z_1))$, then $d(z_0,(\tau \circ \sigma)^{-1}(z_1'))=d(z_0',z_1')< d(z_0, \mathcal{C}_1)$ due to the fact that $\tau \circ \sigma$ is a one to one mapping. Since $z_0 = (\tau \circ \sigma)^{-1}(z_0')$, we know that $\cC_1$ is closed under $(\tau \circ \sigma)^{-1}$ and $(\tau \circ \sigma)^{-1}(z_1') \in \cC_1$. This is contradictory to the fact that $z_1 = \operatorname{argmin}_{z \in \cC_1} d(z_0,z)$. Therefore, $d(z_0', \mathcal{C}_1) = d(z_0',\tau \circ \sigma(z_1))=d(z_0,\mathcal{C}_1)$.
	
	Similarly, if $z_1 \in B_{z_0}(r)$, then $\tau \circ \sigma(z_1) \in B_{z_0'}(r)$. If $z_1' \in B_{z_0'}(r)$, then $(\tau \circ \sigma)^{-1}(z_1') \in B_{z_0}(r)$. Therefore, $\tau \circ \sigma$ is a one to one mapping from $B_{z_0}(r)$ to $B_{z_0'}(r)$, and  $|B_{z_0}(r)|=|B_{z_0'}(r)|$.
	
	Now for a given radius $r$, we find the permutation $\uptau \in S_{|B_{z_0'}(r)|}$ such that $\uptau(z_i)=\tau \circ \sigma(z_i) = z_i'$ for $z_i \in B_{z_0}(r)$ and $z_i' \in B_{z_0'}(r)$. 
	
	
	When the true assignment is $z_0$, the $(k,l)$-th entry of the covariance matrix for the vector $\bL_{z_0}$ can be expressed as
	\begin{align*}
	&\Cov(\bL_{z_0})_{kl}
	=g({p},{q})^2 \Big( |\mathcal{E}_2(z_0,z_k) \cap \mathcal{E}_2(z_0,z_l)|q(1-q) + |\mathcal{E}_1(z_0,z_k) \cap \mathcal{E}_1(z_0,z_l)|p(1-p) \Big)\\
    &\quad=g({p},{q})^2 \Big( |\mathcal{E}_2(\tau \circ \sigma(z_0),\tau \circ \sigma(z_k)) \cap \mathcal{E}_2(\tau \circ \sigma(z_0),\tau \circ \sigma(z_l))|q(1-q)\\
	&\qquad\qquad\qquad + |\mathcal{E}_1(\tau \circ \sigma(z_0),\tau \circ \sigma(z_k)) \cap \mathcal{E}_1(\tau \circ \sigma(z_0),\tau \circ \sigma(z_l))|p(1-p) \Big)\\
	&\quad=\Cov \Big (g({p},{q})\big  (\sum_{ \mathcal{E}_2 (\uptau(z_0'),\uptau(z_k'))} \Ab_{ij} - \sum_{ \mathcal{E}_1(\uptau(z_0'),\uptau(z_k'))} \Ab_{ij}  \big ),g({p},{q})\big  (\sum_{ \mathcal{E}_2(\uptau(z_0'),\uptau(z_l'))} \Ab_{ij} - \sum_{ \mathcal{E}_1(\uptau(z_0'),\uptau(z_l'))} \Ab_{ij}  \big )\Big )\\
	&\quad=\Cov(\bL_{z_0'})_{\uptau(k)\uptau(l)}.
	\end{align*}
	Hence we finish the proof.
\subsection{Proof of Lemma~\ref{lm1-t1}}\label{sec: proof lm b2}
	The proof mainly follows from Lemma 2.3 and Lemma 2.6 in \cite{wang2017likelihood} with modifications for the function $F(\cdot)$ and $G(\cdot)$. We provide the sketch of proof as following:\\
	We first define the count statistics as proposed in \cite{wang2017likelihood}:
	$$\Ab_{i, j} |\left(z(i)=a, z(j)=b\right) \sim \operatorname{Ber}\left(\boldsymbol{\rm H}_{a, b}\right), i\neq j, a,b \in [K].$$
	where $\boldsymbol{\rm H}_{a, b} = p =\lambda_1 \rho_n \text{ if } a=b,$ and $\boldsymbol{\rm H}_{a, b} = q =\lambda_2 \rho_n \text{ if } a\neq b$. $\boldsymbol{\rm H}=\rho_n \boldsymbol{\rm S}$.
	$$\boldsymbol{O}_{a, b}(z)=\sum_{i=1}^{n} \sum_{j \neq i} \mathbbm{1}\left(z(i)=a, z(j)=b\right) \Ab_{i j},$$
	$$L=\sum_{i=1}^n\sum_{j=i+1}^n \Ab_{ij}, \mu_n=n^2 \rho_n.$$
	For two assignments $z, z'$, The confusion matrix is:
	$$\boldsymbol{R}_{k, a}(z, z')=n^{-1} \sum_{i=1}^{n} \mathbbm{1}\left(z(i)=k, z'(i)=a\right).$$
	By definition, we have $|n_k(z)-n/K| \leq c_K, \forall z \in \cC_0 \cup \cC_1, \forall k=1,2,...,K$. We let $\tilde{n}(z)$ denote the number of within-cluster edges, and assume 
	$$n_k(z) = n/K + a_k, |a_k| \leq c_K, k=1,2,...,K,$$
	$$\sum_{k=1}^K a_k =0.$$
	Then 
	\begin{align*}
	\tilde{n}(z)&=\frac{\sum_{k=1}^K (n/K+a_k)^2 -n}{2}=\frac{n^2/K-n}{2}+\frac{\sum_{k=1}^K a_k^2 }{2}\\
	& \leq \frac{n^2/K-n}{2}+Kc_K^2/2.
	\end{align*}
	Therefore, $\forall z,z' \in \cC_1$ we have $\tilde{n}(z) + n_2(z,z') - n_1(z,z')=\tilde{n}(z')$, $|n_2(z,z') - n_1(z,z')|=|\tilde{n}(z)-\tilde{n}(z')| \leq Kc_K^2/2$. Thus, we denote $z^*$ as the true assignment, and $\forall z \in \cC_0 \cup \cC_1$ we have
	\begin{align*}
	\log f(\Ab;z,\hat{p},\hat{q}) &=  \frac{1}{2} \Big (\sum_{a,b=1}^K \boldsymbol{O}_{a,b}(z) \log \frac{\hat{\boldsymbol{\rm H}}_{a,b}}{1-\hat{\boldsymbol{\rm H}}_{a,b}} \Big ) +\tilde{n}(z) \log (1-\hat{p}) + \big (n(n-1)/2 -\tilde{n}(z) \big )\log (1-\hat{q})\\
	&=\frac{1}{2} \Big (\sum_{a,b=1}^K \boldsymbol{O}_{a,b}(z) \big (\log \hat{\boldsymbol{\rm S}}_{a,b} + \log \rho_n - \log (1-\hat{\boldsymbol{\rm H}}_{a,b}) \big )\Big ) + C(z^*) + O_P(\rho_n)\\
	&=\frac{\mu_n}{2} \Big (\sum_{a,b=1}^K \frac{\boldsymbol{O}_{a,b}(z)}{\mu_n} \{ \log \hat{\boldsymbol{\rm S}}_{a,b} + O_{P}( \rho_n) \} \Big ) + \log \rho_n L + C(z^*)+O_P(\rho_n).
	\end{align*}
	where $C(z^*) = \tilde{n}(z^*) \log (1-\hat{p}) + \big (n(n-1)/2 -\tilde{n}(z^*) \big )\log (1-\hat{q})$. 
	We let $F(\boldsymbol{O}(z)/\mu_n)=\sum_{a,b=1}^K \frac{\boldsymbol{O}_{a,b}(z)}{\mu_n} \log \frac{\hat{\boldsymbol{\rm S}}_{a,b}}{1-\hat{\boldsymbol{\rm H}}_{a,b}}$ and $F(\boldsymbol{R}\boldsymbol{\rm S} \boldsymbol{R}^{\top}(z))=\sum_{a,b=1}^K (\boldsymbol{R}\boldsymbol{\rm S} \boldsymbol{R}^{\top}(z))_{a,b} \log \frac{\hat{\boldsymbol{\rm S}}_{a,b}}{1-\hat{\boldsymbol{\rm H}}_{a,b}}$, where $\boldsymbol{R}(z)=\boldsymbol{R}(z,z^*)$ and $\boldsymbol{R}\boldsymbol{\rm S}\boldsymbol{R}^{\top}(z)=\boldsymbol{R}(z,z^*)\boldsymbol{\rm S}\boldsymbol{R}(z,z^*)^{\top}$. We denote $\tilde{\cC} \subseteq \cC_0 \cup \cC_1$ as some subset of assignments, and we let $\mathcal{V}_G$ denote the set of $z \in \tilde{\cC}$ that maximizes $F(\boldsymbol{R}\boldsymbol{\rm S} \boldsymbol{R}^{\top}(z))$.
	Obviously $F(\cdot)$ is Lipschitz, for $\epsilon_n \rightarrow 0$ slowly,
	\begin{align*} 
	&\left|F\left(\boldsymbol{O}(z) / \mu_n \right)-F(\boldsymbol{R}\boldsymbol{\rm S} \boldsymbol{R}^{\top}(z))\right| \\ \leq & C \cdot \sum_{k, l}\left|\boldsymbol{O}_{k, l}(z) / \mu_n-\left(\boldsymbol{R} \boldsymbol{\rm S} \boldsymbol{R}^{\top}(z)\right)_{k, l}\right| \\=& O_{P}\left(\epsilon_{n}\right). 
	\end{align*}
	We choose some positive $\delta_n \rightarrow 0$ slowly enough such that $\delta_n/\epsilon_n \rightarrow \infty$. We take any $Z' \in \mathcal{V}_G$, then we define 
	$$J_{\delta_{n}}=\left\{z \in[K]^{n} : F(\boldsymbol{R}\boldsymbol{\rm S} \boldsymbol{R}^{\top}(z))-F(\boldsymbol{R}\boldsymbol{\rm S} \boldsymbol{R}^{\top}(Z'))<-\delta_{n}\right\}.$$
	Then we have 
	\begin{align*}
	\sum_{z \in J_{\delta_n}} e^{\log f(\Ab; z,\hat{p},\hat{q})} & \leq f(\Ab;Z',\hat{p},\hat{q})K^n e^{ O_{P}(\mu_n \epsilon_n) - \mu_n \delta_n/2 +O_P(\rho_n)}\\
	&=f(\Ab;Z',\hat{p},\hat{q})o_{P}(1).
	\end{align*}
	For $z \in \tilde{\cC} \backslash \left \{ J_{\delta_n}\cup \mathcal{V}_G \right \}$, $|F(\boldsymbol{R}\boldsymbol{\rm S} \boldsymbol{R}^{\top}(z))-F(\boldsymbol{R}\boldsymbol{\rm S} \boldsymbol{R}^{\top}(Z'))| \rightarrow 0$ and $\|\boldsymbol{R}\boldsymbol{\rm S} \boldsymbol{R}^{\top}({z})-\boldsymbol{R}\boldsymbol{\rm S} \boldsymbol{R}^{\top}(Z')\|_{\infty} \rightarrow 0$. Treating $\boldsymbol{R}(z)$ as a vector, choosing $z_{\perp}$ be such that $\boldsymbol{R}\left(z_{\perp}\right) :=\min _{\boldsymbol{R}\left(z_{0}\right) : z_{0} \in \mathcal{V}_{G}}\left\|\boldsymbol{R}(z)-\boldsymbol{R}\left(z_{0}\right)\right\|_{2}$ for a given $z \in \tilde{\cC} \backslash \left \{ J_{\delta_n}\cup \mathcal{V}_G \right \}$. Due to the consistency of $\hat{p},\hat{q}$, the function $F(\cdot)$ is a linear function with constant coefficients. We know that with probability $1-o(1)$:
	$$\left.\frac{\partial F\left((1-\epsilon) \boldsymbol{R}\boldsymbol{\rm S} \boldsymbol{R}^{\top}\left(z_{\perp}\right)+\epsilon\boldsymbol{R}\boldsymbol{\rm S} \boldsymbol{R}^{\top}(z) \right)}{\partial \epsilon}\right|_{\epsilon=0^{+}}<0.$$
 Given a matrix $A$, we denote the matrix maximum norm $\|A\|_{\infty} = \max_{jk} |A_{jk}|$.	Letting $\bar{z}=\min_{\sigma(z)}|\sigma(z) - z_{\perp}|$, and $\boldsymbol{X}(z)= \boldsymbol{O}(z)/\mu_n - \boldsymbol{R}\boldsymbol{\rm S}\boldsymbol{R}^{\top}(z)$, we have 
	\begin{align*} 
	& \mathbb{P}\left(\max _{z \notin \mathcal{S}\left(z_{\perp}\right)}\left\|\boldsymbol{X}(\overline{z})-\boldsymbol{X}\left(z_{\perp}\right)\right\|_{\infty}>\epsilon\left|\overline{z}-z_{\perp}\right| / n\right) \\ \leq & \sum_{m=1}^{n} \mathbb{P}\left(\max _{z : z=\overline{z},\left|\overline{z}-z_{\perp}\right|=m}\left\|\boldsymbol{X}(z)-\boldsymbol{X}\left(z_{\perp}\right)\right\|_{\infty}>\epsilon \frac{m}{n}\right) \\ \leq & \sum_{m=1}^{n} 2K^K n^{m}K^{m+2} \exp \left(-C \frac{m \mu_{n}}{n}\right) \rightarrow 0. 
	\end{align*}
	where $\mathcal{S}(z)=\{\sigma(z) | \sigma \in S_K \}$. Since $\boldsymbol{R }\boldsymbol{\rm S} \boldsymbol{R}^{\top}(\overline{z})-\boldsymbol{R} \boldsymbol{\rm S }\boldsymbol{R}^{\top}\left(z_{\perp}\right) = \Omega(|\bar{z}-z_{\perp}|)$, we have that 
	$$\frac{\boldsymbol{O}(\overline{z})}{\mu_{n}}-\frac{\boldsymbol{O}\left(z_{\perp}\right)}{\mu_{n}} = (1+o_P(1))\left(\boldsymbol{R }\boldsymbol{\rm S} \boldsymbol{R}^{\top}(\overline{z})-\boldsymbol{R} \boldsymbol{\rm S }\boldsymbol{R}^{\top}\left(z_{\perp}\right)\right).$$
	And thus we probability $1-o(1)$ uniform on all $z$, we have
	$$F\left(\boldsymbol{O}(\overline{z})/\mu_n\right) < F\left(\boldsymbol{O}(z_{\perp})/\mu_n\right).$$
    In turn, we have

	$$\log f(\Ab; z, \hat{p},\hat{q}) \le \log f\left(\Ab; z_{\perp},\hat{p},\hat{q}\right) +O_P(\rho_n).$$ 
	Since from Lemma A.1 in \cite{wang2017likelihood} the high probability is uniform on all assignments, we have that, with probability $1-o(1)$, for any $z \in \tilde{\cC} \backslash \mathcal{V}_{G}$ we can find $z' \in \mathcal{V}_G$ such that $\log f(\Ab; z, p,q) = \log f(\Ab; z', p,q) + O_P(\rho_n)$,
	and therefore,
	$$\sup_{z \in \tilde{\cC}} \log f(\Ab;z,\hat{p},\hat{q}) = \sup_{z \in \cV_G} \log f(\Ab;z,\hat{p},\hat{q}) +O_P(\rho_n). $$
	
	Now we consider $F(\boldsymbol{R}\boldsymbol{\rm S} \boldsymbol{R}^{\top}(z))$:
	\begin{align*}
	F(\boldsymbol{R}\boldsymbol{\rm S} \boldsymbol{R}^{\top}(z))&=\mathbb{E}(F(\boldsymbol{O}(z)/\mu_n))\\
	&=\frac{1}{\mu_n}\Big (C_1(z^*)+\log \frac{\hat{\lambda}_1 (1-\hat{p})}{\hat{\lambda}_2 (1-\hat{q})} \big (n_2(z^*,z)q-n_1(z^*,z)p\big ) \Big)\\
	&=\frac{1}{\mu_n}\Big (C_1(z^*)+\log \frac{\hat{\lambda}_1 (1-\hat{p})}{\hat{\lambda}_2 (1-\hat{q})} \big (n_2(z^*,z)\vee n_1(z^*,z)(\lambda_2-\lambda_1)+c_0(z)\big )\rho_n \Big),
	\end{align*}
	where $\hat{\lambda}_1 = \hat{p}/\rho_n, \hat{\lambda}_2 = \hat{q}/\rho_n$ and $C_1(z^*)=\log \frac{\hat{\lambda}_1}{1-\hat{p}} \tilde{n}(z^*)p +\log \frac{\hat{\lambda}_2}{1-\hat{q}} (n(n-1)/2-\tilde{n}(z^*))q$, and $c_0(z) \leq \lambda_1Kc_K^2/2, \forall z \in \cC_0 \cup \cC_1$. 
	
	Thus when $z^* \in \cC_0$ and $\tilde{\cC} = \cC_1$, it can be easily perceived that $\mathcal{V}_G \subseteq B_{z^*}(r_K)$ with high probability, and hence
	$$\sup_{z \in \cC_1} \log f(\Ab;z,\hat{p},\hat{q}) = \sup_{z \in B_{z^*}(r_K)} \log f(\Ab;z,\hat{p},\hat{q}) +O_P(\rho_n).$$
	Moreover, when $z^* \in \tilde{\cC}$, $\cV_G \subseteq B_{z^*}(r^*)$ with high probability, where $r^* = \lambda_1Kc_K^2/\{2(\lambda_1 - \lambda_2)\} = O(1)$. By Lemma~5.3 in \cite{zhang2016minimax}, for any $z \in B_{z^*}(r^*)$, if $z \neq z^*$, then $d(z^*, z) = \Omega (n)$. Therefore, $B_{z^*}(r^*) = \{z^*\}$. In other words, we have
	\begin{equation}\label{eq: lb-lm1-g}
	    \sup_{z \in \tilde{\cC} } \log f(\Ab;z,\hat{p},\hat{q}) = \log  f(\Ab;z^*,\hat{p},\hat{q}) +O_P(\rho_n).
	\end{equation}
	More concretely, if we take $\tilde{\cC} = \cC_0 \cup \cC_1$, we have 
	$$ \sup_{z \in \cC_0 \cup \cC_1 }\log f(\Ab;z,\hat{p},\hat{q}) =  \log f(\Ab;z^*,\hat{p},\hat{q})+O_P(\rho_n),$$
	and if $z^* \in \cC_1$, we have 
	$$ \sup_{z \in \cC_1 } \log f(\Ab;z,\hat{p},\hat{q}) = \log  f(\Ab;z^*,\hat{p},\hat{q}) +O_P(\rho_n).$$

\subsection{Consistency of Probability Estimation}\label{sec: supp}
 Recall the estimators $\hat p$ and $\hat q$ are defined in \eqref{eq: p q hat}, and that $\hat \lambda_1 = \hat p/\rho_n$ and  $\hat \lambda_2 = \hat q/\rho_n$. The following lemma shows that $\hat \lambda_1$ and $\hat \lambda_2$ are consistent.
\begin{lemma}\label{lm: consist of prob est}
	Under the same condition of Theorem \ref{t1-mtd-g}, we have $$|\hat{\lambda_i} - \lambda_i| = O(1/\sqrt{n^2 \rho_n}), \quad i=1,2$$
\end{lemma}
\begin{proof}
From Lemma 1 and Theorem 2 in \cite{bickel2013asymptotic}, we know that $|\log \big(\hat{p}/(1-\hat{p})\big)-\log \big(p/(1-p)\big)|=O(1/\sqrt{n^2 \rho_n})$ and $|\log \big(\hat{q}/(1-\hat{q})\big)-\log \big(q/(1-q)\big)|=O(1/\sqrt{n^2 \rho_n})$.
	We let $\nu_1$ and $\nu_2$ denote the logit of $p$ and $q$. Then since $(\nu_1,\nu_2)$ is a one-to-one function of $(p,q)$, we know the relationship between $(\hat{\nu}_1,\hat{\nu}_2)$ and $(\hat{p},\hat{q})$ should be $\hat{\nu}_1 = \log \hat{p}/(1-\hat{p})$ and $\hat{\nu}_2 = \log \hat{q}/(1-\hat{q})$. Then we have
	\begin{align*}
		\hat{\nu}_i - \nu_i &= \log \frac{\hat{\lambda}_i \rho_n}{1- \hat{\lambda}_i \rho_n} - \log \frac{\lambda_i \rho_n}{1-\lambda_i \rho_n}\\
		 &=\log \hat{\lambda}_i \rho_n - \log \lambda_i \rho_n + \log (1-\lambda_i \rho_n) -  \log (1-\hat{\lambda}_i \rho_n) \\
		 &=\log \left(1+\frac{\hat{\lambda}_i -\lambda_i}{\lambda_i} \right) - \log \left( 1+ \frac{({\lambda}_i - \hat{\lambda}_i) \rho_n}{1-{\lambda}_i \rho_n} \right)\\
		 &=(1+o(1))\frac{\hat{\lambda}_i -\lambda_i}{\lambda_i} + (1+o(1))\frac{(\hat{\lambda}_i - \lambda_i) \rho_n}{1-{\lambda}_i \rho_n} \\
		 & \asymp \hat{\lambda}_i -\lambda_i
	\end{align*}
	and thus by previous results we have $$|\hat{\lambda_i} - \lambda_i| = O(1/\sqrt{n^2 \rho_n}), \quad i=1,2$$
\end{proof}

\end{document}